\documentclass[a4paper, 11pt]{article}
\usepackage{amsmath, amssymb,amscd, amsthm}

\usepackage[mathscr]{eucal}
\usepackage{graphics}
\usepackage{fullpage}
\usepackage{url}
\usepackage{cancel}
\usepackage[margin=0.95in]{geometry}
\usepackage{float}

\newcommand\cyr{%
 \renewcommand\rmdefault{wncyr}%
 \renewcommand\sfdefault{wncyss}%
 \renewcommand\encodingdefault{OT2}%
\normalfont\selectfont} \DeclareTextFontCommand{\textcyr}{\cyr}

\newtheorem{theorem}{Theorem}
\newtheorem{lemma}[theorem]{Lemma}
\newtheorem{corollary}[theorem]{Corollary}
\newtheorem{proposition}[theorem]{Proposition}
\newtheorem{remark}[theorem]{Remark}

\def\mod{\operatorname{mod}}
\def\Q{\mathbb Q}

\long\def\symbolfootnote[#1]#2{\begingroup%
\def\thefootnote{\fnsymbol{footnote}}\footnote[#1]{#2}\endgroup}

\title{Finite field restriction estimates based on Kakeya maximal operator estimates}

\author{Mark Lewko\thanks{Supported by a NSF Postdoctoral Fellowship, DMS-1204206 and the IAS Fund for Math.}}

\date{}

\begin{document}

\maketitle

\begin{abstract}In the finite field setting, we show that the restriction conjecture associated to any one of a large family of $d=2n+1$ dimensional quadratic surfaces implies the $n+1$ dimensional Kakeya conjecture (Dvir's theorem). This includes the case of the paraboloid over finite fields in which $-1$ \emph{is} a square. We are able to partially reverse this implication using the sharp Kakeya maximal operator estimates of Ellenberg, Oberlin and Tao to establish the first finite field restriction estimates beyond the Stein-Tomas exponent in this setting.
\end{abstract}

\symbolfootnote[0]{2010 Mathematics Subject Classification 42B10; Secondary 11T24, 52C17}

\section{Introduction}\label{sec:intro}The Kakeya and restriction conjectures are two central open problems in Euclidean Fourier analysis. The two problems are inextricably linked. The restriction conjecture implies the Kakeya conjecture, and much of the progress on the restriction conjecture has exploited partial progress on the Kakeya conjecture. We refer the reader to \cite{FeffermanR}, \cite{TaoRecent} and \cite{Wolff} for a more detailed discussion of these topics and \cite{BG} for recent results. Both of these problems also have formulations over finite fields. The finite field Kakeya problem was introduced by Wolff \cite{Wolff} in 1999 and solved by Dvir in 2008 using the polynomial method \cite{Dvir}. The finite field restriction conjecture was introduced by Mockenhaupt and Tao \cite{MT} in 2002 and appears far from complete resolution. It seemed to have been widely believed that in finite fields the two problems do not enjoy an explicit connection. Here, we prove that such a formal connection does in fact exist. Most significantly, we will be able to exploit Kakeya maximal operator estimates to make progress on the finite field restriction conjecture.

Let $\mathbb{F}$ denote a finite field, and $\mathbb{F}^d$ the $d$-dimensional Cartesian product of $\mathbb{F}$. We will always assume the characteristic of the finite field $\mathbb{F}$ is strictly greater than $2$.  We let $e(\cdot)$ denote a non-principal character on $\mathbb{F}$ and for $x=(x_1,\ldots,x_d),\xi=(\xi_1,\ldots,\xi_d) \in \mathbb{F}^d$ we define the dot product by $x\cdot \xi= x_1 \xi_1 + x_2 \xi_2 + \ldots +x_d \xi_d$.  We endow the vector space $\mathbb{F}^{d}$ with the counting measure $dx$, and its (isomorphic) dual space $\mathbb{F}^d$ with the normalized counting measure $d\xi$ that assigns measure $|\mathbb{F}|^{-d}$ to each point. For a complex-valued function $f$ on $\mathbb{F}^d$ the $L^p$ norm is given by $ ||f||_{L^p(\mathbb{F}^d,dx)} = \left(\sum_{x \in \mathbb{F}^d} |f(x)|^p \right)^{1/p} $ and its Fourier transform (defined on the dual space $\mathbb{F}^d$) is given by
\[ \hat{f}(\xi) = \sum_{x \in \mathbb{F}^d} f(x) e(-x \cdot \xi).\]
Given a non-empty set $S \subset \mathbb{F}^d$ we define the surface measure of $S$, denoted $d\sigma$, which assigns the measure $|S|^{-1}$ to each point so that
\[ \int_{\xi \in S} g(\xi) d\sigma := \frac{1}{|S|}\sum_{\xi \in S} g(\xi).\]
Thus $||g||_{L^{q}(S,d\sigma)} = \left(|S|^{-1}\sum_{\xi \in S}  |g(\xi)|^q \right)^{1/q} $ and the inverse Fourier transform of $g$ on $S$ is given by
\[  (gd \sigma)^{\vee}(x) = \frac{1}{|S|} \sum_{\xi \in S} g(\xi)e(x \cdot \xi ).\]
For $1 \leq p,q \leq \infty$ we define $\mathcal{R}^{*}(q\rightarrow p)$ to be the best constant such that, for all $g$ supported on $S$, we have
\begin{equation}\label{eq:ExtDef}
|| (gd\sigma)^{\vee}||_{L^p(F^{d},dx) } \leq \mathcal{R}^{*}(q\rightarrow p) ||g||_{L^{q}(S,d\sigma)}.
\end{equation}
By duality, this is also equal to the best constant in the inequality
\[||\hat{f}||_{L^{q'}(S,d\sigma)} \leq  \mathcal{R}^{*}(q\rightarrow p)||f||_{L^{p'}(\mathbb{F}^d,dx) }.\]
When $S$ is an algebraic variety the restriction problem for $S$ seeks to classify the pairs of exponents $(q,p)$ for which $\mathcal{R}^{*}(q\rightarrow p)$ is bounded independent of the field size.

Here we will primarily be interested in the case of the $d$ dimensional paraboloid defined by $\mathcal{P} := \{(\xi, \xi \cdot \xi ) : \xi \in \mathbb{F}^{d-1} \}$. The finite field restriction problem associated to $\mathcal{P}$ has received a fair amount of study, see  \cite{IKParaboloid}, \cite{IKQuadratic}, \cite{KohRadial}, \cite{KangKoh}, \cite{LL}, \cite{LewkoNew}, \cite{MT}. To properly understand the restriction theory of the paraboloid, it is useful to consider it in the context of more general quadratic surfaces. We briefly assume familiarity with quadratic form theory (for definitions and references see section \ref{sec:HighBack}).  Given a non-degenerate quadratic form $Q:\mathbb{F}^{d-1} \rightarrow \mathbb{F}^{d-1}$, one might consider the quadratic surface $\mathcal{S}:=\{(\xi, Q(\xi) ) : \xi \in \mathbb{F}^{d-1}\}$. We will show that the restriction problem is formally equivalent for quadratic surfaces associated to equivalent quadratic forms. In odd characteristic, there are only two equivalence classes of quadratic forms. For even dimensional quadratic forms, these classes are distinguished by an invariant called the Witt index. This is defined to be the dimension of the maximal totally isotropic subspace of $\mathbb{F}^{d-1}$ (this is a subspace $V$ such that $Q(v)=0$ for every $v \in V$ with maximal dimension). If $d-1$ is odd then the Witt index is $\frac{d-2}{2}$. If $d-1$ is even then the Witt index is either $\frac{d-1}{2}$ or $\frac{d-3}{2}$. It turns out that the Witt index of the form associated to the surface is what determines the numerology of the associated restriction problem. If $d$ is odd and $-1$ is a square in $\mathbb{F}$ then the quadratic form $x\cdot x$ (which is associated to the paraboloid) has Witt index $\frac{d-1}{2}$. If $d$ is odd and $-1$ is not a square in $\mathbb{F}$ the Witt index of the quadratic form $x\cdot x$ is $\frac{d-3}{2}$.  Slightly generalizing Mockenhaupt and Tao's finite field restriction conjecture, if $d$ is odd and $\mathcal{S}$ is a surface associated to a quadratic form with Witt index $\frac{d-1}{2}$, then one conjectures
\begin{equation}\label{eq:conjOddSqare}
 || (fd \sigma)^{\vee}||_{L^{\frac{2d}{d-1}}(\mathbb{F}^d,dx)} \lesssim ||f||_{L^{\frac{2d}{d-1}}(\mathcal{S},d\sigma)}.
\end{equation}
If $d$ is odd and $\mathcal{S}$ is a surface associated to a quadratic form with Witt index $\frac{d-3}{2}$ then one conjectures
\begin{equation}\label{eq:conjEvenSqare}
 || (fd \sigma)^{\vee}||_{L^{\frac{2d}{d-1}}(\mathbb{F}^d,dx)} \lesssim ||f||_{L^{\frac{2d^2+2d}{d^2+3}}(\mathcal{S},d\sigma)}.
\end{equation}
If $d$ is even one conjectures
\begin{equation}\label{eq:conjEvenSqare}
 || (fd \sigma)^{\vee}||_{L^{\frac{2d}{d-1}}(\mathbb{F}^d,dx)} \lesssim ||f||_{L^{\frac{2d^2}{d^2-d+2}}(\mathcal{S},d\sigma)}.
\end{equation}
The Stein-Tomas method, which relies only on decay properties of $(d\sigma)^{\vee}$, gives the weaker estimate
\begin{equation}\label{eq:STpar}
 || (fd \sigma)^{\vee}||_{L^{\frac{2d+2}{d-1}}(\mathbb{F}^d,dx)} \lesssim ||f||_{L^{2}(\mathcal{S},d\sigma)}
\end{equation}
in all of these cases.
There has been a fair amount of work aimed at proving estimates beyond Stein-Tomas for the paraboloid in the second case (that is when $d$ is even, or $d$ is odd and the Witt index is when $\frac{d-3}{2}$). For instance, in the case of the $3$ dimensional paraboloid over a field in which $-1$ is not a square, one has the inequality
\begin{equation}
 || (fd \sigma)^{\vee}||_{L^{p}(\mathbb{F}^d,dx)} \lesssim ||f||_{L^{2}(\mathcal{P},d\sigma)}
\end{equation}
with $p=4$ by the Stein-Tomas method. Mockenhaupt and Tao \cite{MT} reduced this to $p >3.6$. This was later refined by Bennett,  Carbery, Garrigos, and Wright \cite{BCGW} to $p = 3.6$ (see also \cite{LL}), and by the current author \cite{LewkoNew} to $p > 3.6 - \delta$ for some small $\delta> 0$. The Mockenhaupt-Tao method was also extended by Iosevich and Koh \cite{IKParaboloid} to certain higher dimensional cases, where additional difficulties enter. All of this work, however, has been in cases where the conjectured inequalities remain $L^2$ based, which allows for methods that seem to have no analog for other $L^p$ norms. In the case of odd dimensions when the Witt index is larger ($\frac{d-1}{2}$), the Stein-Tomas inequality is the best one may obtain with an $L^2$ norm on the right. Indeed, no estimates beyond Stein-Tomas were previously known in this setting. This is analogous to the Euclidean setting where similar issues arise.

Our main result is the following improvement to the Stein-Tomas estimate for $\mathcal{P}$.
\begin{theorem}\label{thm:Main}Let $\mathcal{P}$ denote the $d=2n+1$ dimensional paraboloid over a field in which $-1$ is a square. We then have that
\begin{equation}\label{eq:Main}
|| (fd \sigma)^{\vee}||_{L^{\frac{2d+2} {d-1} - \delta_d}(\mathbb{F}^d,dx)} \lesssim_{d} ||f||_{L^{\frac{2d+2}{d-1}}(\mathcal{P},d\sigma)}
\end{equation}
for some $\delta_d >0$.
\end{theorem}
We will show that one can set $\delta_3 = \frac{4}{10}$ (see Theorem \ref{thm:ResWSP} for a sharper result) and $\delta_5 = \frac{1}{16}- \epsilon$ (for $\epsilon>0$). In general, however, part of the argument proceeds by an induction that leads to recursively defined estimates that are not easily analyzed.

The proof of this theorem will make use of a connection with the finite field Kakeya problem.  Recall that a Kakeya set in $\mathbb{F}^m$ is a subset $E \subseteq \mathbb{F}^{m}$ that contains a line in every direction. The finite field Kakeya problem asks for a lower bound the size of any such set. This problem was settled by Dvir in 2008 using the polynomial method. Indeed, one has the lower bound $|E| \gg_{m} |\mathbb{F}|^m$. We will show that this fact is implied by the finite field restriction conjecture. More precisely:
\begin{theorem}\label{thm:ResToKak}Assume the restriction conjecture holds for some $d=2n+1$ paraboloid over finite fields in which $-1$ is a square. Then a Kakeya set $E$ in dimension $n+1$ must have full dimension. In other words, $|E| \gg |\mathbb{F}|^{n+1} $.
\end{theorem}
More generally, the hypothesis can be replaced with the  restriction conjecture for any $d$ dimensional quadratic surface associated to a quadratic form with Witt index $\frac{d-1}{2}$, such as the hyperbolic paraboloid. One can, in fact, embed the $n+1$ dimensional Kakeya maximal operator into the finite field restriction problem. Indeed we show that every maximal totally isotropic subspace gives rise to a distinct such embedding. Note that the numerology here differs from the Euclidean setting where the $d$ dimensional restriction conjecture implies the $d$ dimensional Kakeya conjecture.

It is natural to ask if one can reverse this implication and use Kakeya estimates to prove restriction estimates. In the Euclidean case this was carried out in a celebrated paper of Bourgain \cite{Brest} in 1991. The main contribution of the current paper is to show that this can be done in the finite field case as well. Indeed this is the idea behind the proof of Theorem \ref{thm:Main}. Here we make use of the finite field Kakeya maximal operator estimates of Ellenberg, Oberlin, and Tao \cite{EOT}. While there are many analogies with the Euclidean case, the relationship between the two settings is still rather tenuous. Indeed, the arguments developed here appear to be significantly distinct from the arguments previously applied in the Euclidean case.

We start by presenting a $3$ dimensional version of our argument.  In this case we can prove stronger results. The analysis here will be based on, among other things, a careful analysis of the near-extremizers of the Stein-Tomas inequality. Here it will be convenient (when $d=2n+1$ is odd) to work with the hyperbolic paraboloid defined by $\mathcal{H}:=\{(\xi_1,\xi_2, \xi_1 \cdot \xi_2) : \xi_1,\xi_2 \in \mathbb{F}^{n} \} $. As we will see, the restriction problem for this surface is equivalent to the restriction problem for the Paraboloid over fields in which $-1$ is a square.
 \begin{theorem}\label{thm:Res3}Let $\mathcal{H}$ be the $3$ dimensional hyperbolic paraboloid. Then:
\begin{equation}
|| (fd \sigma)^{\vee}||_{L^{18/5}(\mathbb{F}^3,dx)} \lesssim ||f||_{L^{9/4}(\mathcal{H},d\sigma)}.
\end{equation}
\end{theorem}
This result is on the critical line in the sense that the exponent $q=9/4$ is sharp for $p=18/5$. Compared with the analogous problem in higher dimensions, the $3$ dimensional case enjoys two simplifying features. First, the relevant Kakeya problem is $2$ dimensional. This is helpful because the $2$ dimensional Kakeya problem is particularly well understood, and the relevant maximal inequalities are $L^2$ based. Second, this surface contains only a (uniformly) bounded number (indeed two) of maximal totally isotropic subspaces. This significantly restricts the manner in which a Kakeya problem can be embedded into the restriction problem. In $5$ and higher dimensions this is no longer the case, and we must deal with the possibility that Kakeya type sets are embedded into the restriction problem in many different ways.

In section \ref{sec:3dthmSumPro}, following the ideas of \cite{LewkoNew}, we show how to insert the Bourgain, Katz, and Tao \cite{BKT} incidence/sum-product estimate into the previous argument to obtain the following slight improvement.
\begin{theorem}\label{thm:ResWSP}Let $\mathcal{H}$ be the $3$ dimensional hyperbolic paraboloid. There exists a $\delta >0$ such that
\begin{equation}
|| (fd \sigma)^{\vee}||_{L^{18/5-\delta}(\mathbb{F}^3,dx)} \lesssim ||f||_{L^{\frac{18-5\delta}{8-5\delta}}(\mathcal{H},d\sigma)}
\end{equation}
holds.
\end{theorem}
Again the exponent on the right, $(18-5\delta)/(8-5\delta)$, is best possible given the exponent on the left.  As we have already mentioned, the finite field restriction problem was introduced as a model for analogous problems in the Euclidean Fourier analysis. More generally, however, restriction problems in the discrete setting may be viewed as part of the wider program of understanding exponential sums with arbitrary coefficients. This is an area in its infancy, but one which seems to hold much potential. See \cite{BourDir}, \cite{BourLam}, and \cite{Green} for problems and applications along these lines. The results of this paper, as well as previous work on these problems, are summarized in the table at the end of the paper.

\textbf{Acknowledgment:} The author would like to thank Doowon Koh for corrections to an earlier draft of this manuscript.

\section{Overview of the proof}\label{sec:overview}In this section, we give an overview of the proofs of Theorem \ref{thm:Res3} and Theorem \ref{thm:Main}.

We start by recalling how the Mockenhaupt-Tao argument works in the case that the $3$ dimensional paraboloid in the case of a finite field in which $-1$ is not a square. In this setting, the Stein-Tomas method gives the $L^4$ estimate
\begin{equation}\label{eq:STL4}
 || (fd \sigma)^{\vee}||_{L^{4}(\mathbb{F}^3,dx)} \lesssim ||f||_{L^{2}(\mathcal{P},d\sigma)} .
\end{equation}
One is primarily interested in lowering the exponent on the left. The argument, however, is indirect and starts off by obtaining an improvement in the right exponent. An appealing feature of $L^4$ estimates is that one can approach them by expanding out the norm. Indeed, expanding this out one sees that the left hand side of \eqref{eq:STL4} is controlled by the additive energy (see section \ref{sec:HighEnergy}) of the support of $f$. Some combinatorics/incidence theory then allows one to obtain the following improvement to the above inequality
\begin{equation}\label{eq:IN85}
|| (fd \sigma)^{\vee}||_{L^{4}(\mathbb{F}^3,dx)} \lesssim ||f||_{L^{8/5}(\mathcal{P},d\sigma)}.
\end{equation}
We now return to the problem of improving the left exponent. One considers the problem in the dual/restriction format
\begin{equation}\label{eq:DualForm}
 ||\hat{G}||_{L^{q'}(\mathcal{P},d\sigma)} \leq  \mathcal{R}^{*}(p\rightarrow q)||G||_{L^{p'}(\mathbb{F}^3,dx) }.
\end{equation}
Here the restriction operator is being applied to a function, $G$, on $\mathbb{F}^d$ (in comparison to the extension operator that is applied to a function on $\mathcal{P}$ which can be parameterized by $\mathbb{F}^{d-1}$). The Mockenhaupt-Tao machine (which uses an idea of Carbery \cite{Carbery}) shows that the restriction operator applied to a single vertical slice of $G$, say $G_{z}$, can be controlled by the extension operator applied to $G_{z}$ (thought of as a function on $\mathcal{P}$).  One can then use the inequality (\ref{eq:IN85}) to bound the contribution on each slice. This does not directly give an improvement to Stein-Tomas, but it does give a favorable local restriction estimate\footnote{By a local restriction estimate we mean an estimate of the form (\ref{eq:ExtDef}) where the constant has some dependence on the field size. A global restriction estimate will be an estimate of the form (\ref{eq:ExtDef}) with no dependence on the field size.} at a low exponent. One can then insert this local restriction estimate into the Stein-Tomas machine to prove a global restriction estimate at a higher exponent (but one that is still below the Stein-Tomas exponent).

The initial obstruction with using this approach over fields in which $-1$ is a square (we will use the 3 dimensional hyperbolic paraboloid $\mathcal{H}$ in the following discussion) is that the $L^2$ on the right side of the $L^4$ Stein-Tomas extension estimate (\ref{eq:STL4}) is sharp in this setting. This can be seen by testing the extension operator on the characteristic function of a line contained within the surface. A second obstruction is that the Mockenhaupt-Tao method for converting an improvement in the right exponent to an improvement in the left exponent makes crucial use of the fact that the inequality one is attempting to prove has a $L^2$ norm on the right (and thus allows for a $T T^{*}$ type argument). On the other hand, any extension estimate of the form
$$|| (fd \sigma)^{\vee}||_{L^{p}(\mathbb{F}^3,dx)} \lesssim ||f||_{L^{q}(\mathcal{H},d\sigma)} $$
with $p < 4$ must have $q >2$. Thus one cannot hope to rely on $L^2$ methods alone.

We now describe our approach, and how it circumvents these issues. We will consider the problem in the dual/restriction formulation (\ref{eq:DualForm}). The first step will be to prove that the only near extremizers to the inequality (\ref{eq:STL4}) are functions that are nearly the characteristic functions of lines. Although not formulated as such, this argument has a bilinear flavor. We then show that one may implement the Mockenhaupt-Tao argument (using the $L^2$ based methods) as long as each vertical slice $G_{z}$ is far from being (the characteristic function of) a line. It remains to consider the case that every vertical slice of $G$ is concentrated on a line. We divide this problematic case into two further cases. The first case occurs when the union of all of the lines (obtained from the vertical slices) lie in the union of a small number of planes. Note that if the function is supported on a single plane then its Fourier transform will be constant on lines perpendicular to the plane. Here, one can obtain sufficient estimates from an explicit understanding of how these perpendicular lines intersect the surface. In addition, one can combine this idea with an orthogonality argument to obtain satisfactory estimates when the number of planes involved is small. It is this part of the argument that avoids the use of $L^2$ methods.  In the final case, one has that each vertical slice $G_z$ is nearly supported on a line, but these lines are not contained in the union of a small number of planes. Informally, the hypothesis in this case ensures that the Fourier transform of $G$ will not concentrate on a small number of hyperplanes. This enables one to return to $L^2$ methods. More formally, we can control the restriction operator applied to $G$ by (an appropriate norm of) the associated Bochner-Riesz operator applied to $G$. One can then exploit a geometric argument (similar to C\'{o}rdoba's Kakeya estimate) to control the Bochner-Riesz operator in this case. It is in this step that we make use of Kakeya-type information (although we do not use it explicitly).
Sections \ref{sec:3dEnergy}, \ref{sec:3dMT}, and \ref{sec:Planar} are devoted to proving quantitative forms of these claims. The proof of Theorem \ref{thm:Res3} will be assembled from these parts in Section \ref{sec:3dthm}. In section \ref{sec:3dthmSumPro} we will describe how to use incidence theory to obtain a further small improvement. Roughly speaking, the incidence theory is inserted into the analysis of the $L^4$ norm, which is then used in the Mockenhaupt-Tao case in the above summary.

We now describe our generalization of this approach to higher dimensions. The structure of the argument is similar, however there are several additional complications that enter, and we proceed more crudely in several stages in the argument.  In the $3$ dimensional case described above, the only way that the additive energy of a slice $G_z$ can be large is if the support of $G_z$ is concentrated on a vertical or horizontal line. Note that the vertical lines are just the cosets of a $1$ dimensional (totally) isotropic subspace, and the horizontal lines are the cosets of the other $1$ dimensional (totally) isotropic subspace. In higher dimensions, the appropriate generalization is that $G_z$ must concentrate on affine totally isotropic subspaces (with respect to the bilinear form associated to the quadratic surface). However, there are now large families of totally isotropic subspaces on which the set can concentrate. This prevents one from using an argument similar to the one we develop in the $3$ dimensional case. Our argument now proceeds by induction on dimension. The argument starts by converting the additive energy estimate into a incidence problem. If the incidence problem is not too degenerate (we will not formally define this  here), then one immediately obtains a favorable additive energy estimate. If the problem is degenerate, we can deduce that a large portion of the set concentrates on a lower dimensional subsurface. If this subsurface is fully degenerate (as a quadratic surface) then it follows that a large subset of the set must lie in an affine totally isotropic subspace. If the subsurface is not fully degenerate, one can then reduce to a lower dimensional instance of the problem. To implement this induction, however, one must work with all quadratic surfaces and not just the paraboloid. These arguments make fairly extensive use of the theory quadratic forms over finite fields (such as Witt's decomposition theorem). Unfortunately, this iterative procedure leads to recursively defined estimates that are not easily analyzed (and are likely very inefficient). Combining this with the Mockenhaupt-Tao method, as in the $3$ dimensional case, we are able to obtain a favorable restriction estimate unless each slice of $G_z$ is concentrated on a small number of maximal totally isotropic affine subspaces.

On the other hand, we are able to convert the Kakeya maximal operator estimate of Ellenberg, Oberlin, and Tao into a mixed norm restriction estimate for $\mathcal{P}$ (see (\ref{eq:reparam}) for a precise formulation). The mixed norm involves cosets of a (any) maximal totally isotropic subspace. This immediately yields an improvement over the Stein-Tomas exponent whenever the slices of $G$ are concentrated on cosets of the \emph{same} maximal totally isotropic subspace. It is possible, however, that each slice $G_z$ concentrates on a coset of a distinct maximal totally isotropic subspace. Fortunately, some elementary combinatorics ensures that the slices $G_z$ cannot be too equidistributed with respect to the cosets of distinct maximal isotropic subspaces, which is enough to complete the proof.

One will note that the higher dimensional argument we just described omits an analog of the planar function case. We have been able to subsume this case into the mixed norm Kakeya estimate. On the other hand, this approach is more crude and we no longer obtain a $q$ exponent on the critical line.

\section{Notation}\label{sec:notation}
We use the usual asymptotic notation. We write $X\lesssim Y$ to indicate that there exists a universal constant $C$ such that $X \leq C Y$. In addition, we will write $X \sim Y$ to indicate that $Y/2 \leq X \leq 2 Y$.
As defined in the introduction, we will use the notation $ ||f||_{L^p(\mathbb{F}^n,dx)} = \left(\sum_{x \in \mathbb{F}^n} |f(x)|^p \right)^{1/p} $ and $||g||_{L^{q}(\mathcal{S},d\sigma)} = \left(|\mathcal{S}|^{-1}\sum_{\xi \in \mathcal{S}}  |g(\xi)|^q \right)^{1/q} $. This long form notation, which was used in \cite{MT}, is intended to avoid confusion over which measure/space the norm is taken with respect to. In order to streamline computations, when there is no possibility of confusion, we will  abbreviate  $ ||f||_{L^p(\mathbb{F}^n,dx)}$ as simply $||f||_{L^p_x}$. Given a multivariate function, say $f(x_1,x_2)$, we will define
$||f||_{L_{x_1}^p}(x_2) = \left(\sum_{x_1}|f(x_1,x_2)|^{p} \right)^{1/p} $
to be the quantity obtained by fixing $x_2$ and computing the $L^p$ norm in the $x_1$ variable. Similarly, we define the mixed norm $||f||_{L_{x_2}^{q} L_{x_1}^p} = \left( \sum_{x_2} \left(\sum_{x_1}|f(x_1,x_2)|^{p} \right)^{q/p} \right)^{1/q} $.
We refer the reader to \cite{BP} for the basic theory of these spaces.  We will find it convenient to identify a $d$ dimensional surface $\mathcal{S}$ with $\mathbb{F}^{d-1}$ via the map
$$\xi \rightarrow (\xi, Q(\xi)).$$
Here, $\xi \in \mathbb{F}^{d-1}$ and $(\xi,Q(\xi) ) \in \mathbb{F}^{d-1} \times \mathbb{F}$.

We define a VH line (we use VH to stand for vertical or horizontal) in $\mathbb{F}^2$ to be a line of the form
$$\{(x,y_0) : x \in \mathbb{F} \}\text{, or } \{(x_0,y) : y \in \mathbb{F} \}.  $$
We call lines of the first form type 1 and lines of the second form type 2. We extend these definitions to lines in $\mathbb{F}^3$. To be precise, a \emph{VH line} in $\mathbb{F}^3$ takes the form
$$\{(x,y_0,t_0) : x \in \mathbb{F} \}\text{, or } \{(x_0,y,t_0) : y \in \mathbb{F} \}, $$
with the first being type $1$ and the second being a type $2$. We will say that a set $E \subseteq \mathbb{F}^2$ (respectively, $E \subseteq \mathbb{F}^3$) is a $VH(\alpha)$ set if for any VH line $\ell$ we have that
$$ |E \cap \ell | \leq |\mathbb{F}|^{\alpha} .$$
We define a VH plane to be any plane in $\mathbb{F}^3$ consisting of VH lines, excluding planes of the form $\{(x_1,x_2,t) : x_1,x_2 \in \mathbb{F} \} $ for $t$ fixed. With this exclusion, we see that all of the lines in a VH plane will be of the same type. If the lines in a given VH plane are all type 1 (respectively, type 2) we call the plane type 1 (respectively, type 2).
We will say that $E\subseteq \mathbb{F}^3$ has planar entropy $e$ if the minimal number of VH planes needed to cover $E$ is $|\mathbb{F}|^e$.

We note that (when considering the $3$ dimensional hyperbolic paraboloid $\mathcal{H}$) VH lines are precisely the totally isotropic affine subspaces of the associated bilinear form. This point of view from quadratic form theory will be important in our higher dimensional arguments (see section \ref{sec:HighBack}), but does not add much in the $3$ dimensional case we present first.

\section{Restriction implies Kakeya}\label{sec:KaktoRes} We recall the Kakeya maximal operator estimates proven by Ellenberg, Oberlin, and Tao \cite{EOT}. First we need to introduce some notation. In $\mathbb{F}^m$ given $b,\eta \in \mathbb{F}^{m-1}$ we define the (non-horizontal) line
\[\ell_{(b,\eta)} := \{(b+ \eta t,t), t \in \mathbb{F}  \}. \]
For $F : \mathbb{F}^m \rightarrow \mathbb{C}$ we define the Kakeya maximal function $F^{*} : \mathbb{F}^{m-1} \rightarrow \mathbb{R}$ (which should be thought of as a function on the set of non-horizontal directions)  by
\begin{equation}\label{eq:KakeyaMaxDef}
F^{*}(\eta) := \max_{b \in \mathbb{F}^{m-1}} \sum_{x \in \ell(b,\eta)} |F(x)| .
\end{equation}
We then have that:
\begin{proposition}(Ellenberg-Oberlin-Tao) For any $F : \mathbb{F}^m \rightarrow \mathbb{R}$ we have that the estimate
\begin{equation}\label{eq:KakeyaMaxIn}
 || F^{*}||_{L^{m}(\mathbb{F}^{m-1},d\theta)} := \left( \frac{1}{|\mathbb{F}|^{m-1}} \sum_{\theta \in \mathbb{F}^{m-1}} |F^{*}(\theta)|^{m} \right)^{1/m} \lesssim_{m}  ||F||_{L^{m}(\mathbb{F}^{m},dx)}.
\end{equation}
holds. Here $d\theta$ denotes normalized counting measure on  $\mathbb{F}^{m-1}$.
\end{proposition}
Let us recall that a Kakeya maximal function estimate of the form
$$|| F^{*}||_{L^{q'}(\mathbb{F}^{m-1},d\theta)}  \leq K_{n}(p' \rightarrow q')  ||F||_{L^{p'}(\mathbb{F}^{m},dx)}.$$
has a dual formulation. Let $dv$ denote the normalized counting measure on $\mathbb{F}^{n-1}$ that assigns $|\mathbb{F}^{(n-1)}|^{-1}$ to each point. The above Kakeya maximal estimate then is equivalent to the following estimate
\begin{equation}\label{eq:dualKakeya}
 \left|\left| \int_{v \in \mathbb{F}^{m-1}} h(v) 1_{\ell_{(x_0(v),v)}} dv   \right|\right|_{L^{p}(\mathbb{F}^{m},dx)} \leq K_{m}^{*}(q \rightarrow p) ||h||_{L^{q}(\mathbb{F}^{m-1},dv)}
\end{equation}
for all functions $h$ on $\mathbb{F}^{m-1}$ and all functions $x_0 : \mathbb{F}^{m-1} \rightarrow \mathbb{F}^{m-1}$, where $K_m^{*}(q \rightarrow p) = K_{n}(p' \rightarrow q')$.

Here we will prove that the finite field restriction conjecture for the $d=2n+1$ hyperbolic paraboloid implies the $n+1$ dimensional Kakeya set conjecture. For convenience we set $m=n+1$. With this notation we will be using a $d$ dimensional restriction estimate to prove a $m$ dimensional Kakeya estimate. Curiously, our argument does not give the full range of Kakeya maximal estimates, but does give an estimate sufficient to prove that any Kakeya set in $\mathbb{F}^m$ must have size $\gg |\mathbb{F}|^{m}$.  Perhaps, one can refine this argument to obtain better Kakeya maximal operator estimates from the conjectured restriction exponents. However, since the full range of Kakeya maximal estimates have been unconditionally proven, we will not distract ourselves with these issues.  We prove the following:
\begin{proposition}\label{prop:ResKak}Let $d=2n+1$ and $m=n+1$. Assume the $d$ dimensional restriction conjecture for the hyperbolic paraboloid holds. Then we have that
$$ K_{m}^{*} \left(\frac{2m-1}{2m-2} \rightarrow \frac{2m-1}{2m-2} \right) \lesssim |\mathbb{F}|^{\frac{m-1}{2m-1}}.$$
Moreover, any $m$ dimensional Kakeya set has size $\gg |\mathbb{F}|^m$.
\end{proposition}The first part is a corollary to Theorem \ref{thm:ResToKak} below. To see the second claim, which is stated in the introduction as Theorem \ref{thm:ResToKak}, write the first estimate in the form
$$ || F^{*}||_{L^{2m-1}(\mathbb{F}^{m-1},d\theta)} \lesssim |\mathbb{F}|^{\frac{m-1}{2m-1}} ||F||_{L^{2m-1}(\mathbb{F}^m,dx)} .$$
If we take $F$ to be the characteristic function of a Kakeya set of size $|\mathbb{F}|^{\gamma}$, the above inequality takes the form
$$|\mathbb{F}| \lesssim |\mathbb{F}|^{\frac{\gamma+m-1}{2m-1}}$$
which implies the second claim.
We recall the finite field restriction conjecture for the hyperbolic paraboloid. We consider the $d=2n+1$ dimensional hyperbolic paraboloid defined by $\{(\xi, \eta, \xi\cdot\eta) : \xi,\eta \in \mathbb{F}^n \}$. Given a function $f: \mathbb{F}^n\times \mathbb{F}^n \rightarrow \mathbb{C}$ we consider the extension operator (for $x=(x_1,x_2,t ) \in \mathbb{F}^{n} \times \mathbb{F}^n \times \mathbb{F}$)
\[ (fd \sigma)^{\vee}(x_1,x_2,t) := \frac{1}{|\mathbb{F}|^{2n}} \sum_{\xi,\eta \in \mathbb{F}^n} f(\xi,\eta)e(\xi\cdot x_1 + \eta \cdot x_2 + t\xi\cdot \eta) . \]
The conjecture states
\[ || (fd \sigma)^{\vee}||_{L^{2d/(d-1)}(\mathbb{F}^{d},dx)} \lesssim ||f||_{L^{2d/(d-1)}(\mathcal{H},d\sigma)}. \]
Thus Proposition \ref{prop:ResKak} follows from the following relation.
\begin{theorem}\label{thm:ResToKak}Let $\mathcal{R}^{*}(p \rightarrow q$ denote the optimal constant in the restriction problem for the for the hyperbolic paraboloid $\mathcal{H}$ as defined above in $d=2n+1$ dimensions. In addition, set $m=n+1$.  We then have that
\[ K_{m}^{*}(q \rightarrow p) \lesssim |\mathbb{F}|^{(m-1)(1-1/p)} \left(\mathcal{R}^{*}(2q \rightarrow 2p)\right)^{2}.\]
\end{theorem}
\begin{proof}Clearly, it suffices to prove (\ref{eq:dualKakeya}) for $h$ non-negative. Let $b : \mathbb{F}^n \rightarrow \mathbb{F}^n$ be an arbitrary function. We consider the extension operator (defined above) applied to the function (for $\xi, \theta \in \mathbb{F}^n$)
$$f(\xi,\theta) :=   h^{1/2}(\theta)e(-b(-\theta) \cdot \xi).$$
We now have that
\[ (fd \sigma)^{\vee}(x_1,x_2,t)  =\frac{1}{|\mathbb{F}|^{2n}}\sum_{\theta \in \mathbb{F}^{n}}  \sum_{\xi \in \mathbb{F}^n} h^{1/2}(\theta)e(-b(-\theta) \cdot \xi)  e(\xi\cdot x_1 + \theta \cdot x_2 + t \xi \cdot \theta)   \]
\[ =\frac{1}{|\mathbb{F}|^{2n}}\sum_{\theta \in \mathbb{F}^{n}}  \sum_{\xi \in \mathbb{F}^n} h^{1/2}(\theta)  e(\xi\cdot (x_1-b(-\theta)+t \theta) + \theta \cdot x_2)   \]
\[= \frac{1}{|\mathbb{F}|^{n}}\sum_{\theta \in \mathbb{F}^{n}}   h^{1/2}(\theta) \delta\left( x_1-b(-\theta)+t \theta\right) e(\theta \cdot x_2)  \]
\[= \frac{1}{|\mathbb{F}|^{n}}  \sum_{\theta \in \mathbb{F}^{n}}  h^{1/2}(\theta) 1_{\ell(-b(\theta),-\theta)}(x_1,t) e(\theta \cdot x_2).   \]
By orthogonality of characters we have that
\[ \left|\left| \frac{1}{|\mathbb{F}|^{n}}  \sum_{\theta \in \mathbb{F}^{n}}  h^{1/2}(\theta) 1_{\ell(b(\theta),\theta)}(x_1,t) e(\theta \cdot x_2) \right|\right|_{L^{2}_{x_2}} = \left(\frac{1}{|\mathbb{F}|^{n}}\sum_{\theta \in \mathbb{F}^{n}}  h(\theta) 1_{\ell(b(\theta),\theta)}(x_1,t) \right)^{1/2}. \]
By H\"{o}lder's inequality, we have, provided $r \geq 2$, that
\[ ||(fd \sigma)^{\vee}||_{L^{r}_{x_1,t}L^{2}_{x_2}} \lesssim |\mathbb{F}|^{n(1/2 -1/r)} ||(fd \sigma)^{\vee}||_{L^{r}_{x_1,x_2,t}} \lesssim |\mathbb{F}|^{n(1/2 -1/r)} \mathcal{R}^{*}(s \rightarrow r)  ||f||_{L^{s}(\mathcal{H},d\sigma)}\]
 \[=  |\mathbb{F}|^{n(1/2 -1/r)} \mathcal{R}^{*}(s \rightarrow r) \left(\frac{|\mathbb{F}|^n}{|\mathbb{F}|^{2n}}  \sum_{\theta \in \mathbb{F}^{n}} |h(\theta)|^{s/2}    \right)^{1/s}    \]
\[=  \mathcal{R}^{*}(s \rightarrow r) |\mathbb{F}|^{n/2 -n/r} ||h||_{L^{s/2}(\mathbb{F}^{n},d\theta)}^{1/2}. \]
Now
\[ ||(fd \sigma)^{\vee}||_{L^{r}_{x_1,t}L^{2}_{x_2}} = \left|\left| \left( \frac{1}{|\mathbb{F}|^{n}}\sum_{\theta \in \mathbb{F}^{n}}  h(\theta) 1_{\ell(b(\theta),\theta)}(x_1,t) \right)^{1/2}\right|\right|_{L^{r}_{x_1,t}}  \]
\[= \left|\left|  \frac{1}{|\mathbb{F}|^{n}} \sum_{\theta \in \mathbb{F}^{n}}  h(\theta) 1_{\ell(b(\theta),\theta)}(x_1,t)\right|\right|_{L^{r/2}_{x_1,t}}^{1/2}. \]
Thus
\[\left|\left|  \frac{1}{|\mathbb{F}|^{n}} \sum_{\theta \in \mathbb{F}^{n}}  h(\theta) 1_{\ell(-b(\theta),-\theta)}(x_1,t)\right|\right|_{L^{r/2}_{x_1,t}} =||(fd \sigma)^{\vee}||_{L^{r}_{x_1,t}L^{2}_{x_2}}^2  \]
\[\lesssim \mathcal{R}^{*}(s \rightarrow r)^2 |\mathbb{F}|^{n -2n/r} ||h||_{L^{s/2}(\mathbb{F}^{n},d\theta)}. \]
Setting $n=m-1$, $r=2p$ and $s=2q$ completes the proof.
\end{proof}
\section{Standard results from harmonic analysis}\label{sec:StanHarm}
In this section we will recall some results from discrete Fourier analysis.  We start by recalling the notion of Fourier dimension. Given $S \subset \mathbb{F}^d$ we denote by  $d\sigma$ the normalized counting measure on $S$. The inverse Fourier transform of $d\sigma$ is then given by
\[(d\sigma)^{\vee}(x) = \frac{1}{|S|}\sum_{\xi \in S} e(x\cdot \xi).\]
Note that $(d\sigma)^{\vee}(0)=1$, however for certain $S$ we may hope that $|(d\sigma)^{\vee}(x)|$ is small for $x\neq 0$. In particular, we define the Fourier dimension of $S$ to be the largest $\tilde{d}$ such that
\[|(d\sigma)^{\vee}(x)| \leq |\mathbb{F}|^{-\tilde{d}/2}  \]
for all $x \neq 0$.  If $S$ has dimension $\beta$, that is if $|S|\sim |\mathbb{F}|^\beta$, then the Fourier dimension of $S$ is at most $\beta$ as can be seen from Plancherel's theorem. Up to logarithmic factors this will hold for a random surface/subset of size $|\mathbb{F}|^\beta$. However, if the surface contains additive structure the Fourier dimension maybe smaller. It is also convenient to define the Bochner-Riesz kernel $K$ associated to $S$ by $K(x) = (d\sigma)^{\vee}(x) - \delta_{0}(x)$ (where the delta function $\delta_{0}$ is defined to be $1$ at $0$ and $0$ otherwise). As usual, we let $d=2n+1$.

\begin{lemma}\label{lem:FTcomp}Let $S=\mathcal{H}$ the $d$ dimensional hyperbolic paraboloid and d$\sigma$ be as above. Let $x=(x_1,x_2,t) \in \mathbb{F}^d$ with $x_1,x_2 \in \mathbb{F}^n$ and $t \in \mathbb{F}$. Then
\begin{equation}
  (d\sigma)^{\vee}(x)=\begin{cases}
    1, & \text{if $x=(0,0,0)$},\\
    0, & \text{if $t = 0$ and $(x_1,x_2) \neq (0,0)$ },\\
    |\mathbb{F}|^{-n} e(-x_2 \cdot x_1/t), & \text{otherwise}.
  \end{cases}
\end{equation}
Thus the $d=2n+1$ dimensional hyperbolic paraboloid $\mathcal{H}$ has Fourier dimension $2n$.
\end{lemma}
\begin{proof}The first two cases are obvious. We assume that $t \neq 0$ and $x_2 \neq 0$ (the case with $t \neq 0$ and $x_1 \neq 0$ will following by symmetry). We have that
$$(d\sigma)^{\vee}(x_1,x_2,t) = |\mathbb{F}|^{-2n}\sum_{\xi_1,\xi_2 \in \mathbb{F}^n} e(x_1\cdot \xi_1 + x_2\cdot \xi_2 + t \xi_1 \cdot \xi_2) = |\mathbb{F}|^{-2n}\sum_{\xi_1 \in \mathbb{F}^n}e(\xi_1 \cdot x_1) \sum_{\xi_2 \in \mathbb{F}^n} e(\xi_2 \cdot (x_2 +t\xi_1 ))   $$
$$= |\mathbb{F}|^{-n} \sum_{\xi_1 \in \mathbb{F}^n }e(\xi_1 \cdot x_1) \delta(x_2 +t \xi_1)  = |\mathbb{F}|^{-n} e(-x_2 \cdot x_1/t).$$
Lastly, assume that $t \neq 0$ and $(x_1,x_2)=(0,0)$. We then have
$$(d\sigma)^{\vee}(0,0,t) = |\mathbb{F}|^{-2n}\sum_{\xi_1,\xi_2 \in \mathbb{F}^n} e(t \xi_1 \cdot \xi_2) = |\mathbb{F}|^{-n}.$$
This completes the proof.
\end{proof}
A similar argument gives the following result for the standard paraboloid. A proof can be found in \cite{MT} (see equation (17) there, however note the use of slightly different notation).
\begin{lemma}\label{lem:FTcompP}Let $\mathcal{P}$ be the $d$ dimensional paraboloid and d$\sigma$ the normalized surface measure on $\mathcal{P}$. Let $x=(\underline{x},t) \in \mathbb{F}^d$ with $\underline{x} \in \mathbb{F}^{2n}$ and $t \in \mathbb{F}$. Then
\begin{equation}
  (d\sigma)^{\vee}(x)=\begin{cases}
    1, & \text{if $x=(0,0,0)$},\\
    0, & \text{if $t = 0$ and $\underline{x} \neq 0$ },\\
    |\mathbb{F}|^{-n} S(t)  e(\underline{x}\cdot \underline{x} /4t), & \text{otherwise}.
  \end{cases}
\end{equation}
Here $S(t)= |\mathbb{F}|^{-n} \left(\sum_{\xi \in \mathbb{F}}e(t\xi^2) \right)^{2n}$ (where $|S(t)|= 1$ for $t \neq 0$). Thus the $d=2n+1$ dimensional paraboloid $\mathcal{P}$ has Fourier dimension $2n$.
\end{lemma}

Over the next few lemmas we recall the Stein-Tomas argument. The first lemma exploits the decay of the Fourier transform $(d\sigma)^{\vee}$.
\begin{lemma}\label{Lem:STsupp}Let $S$ have Fourier dimension $\tilde{d}$ and $p,q \geq 2$ and $0\leq \theta \leq 1$. Let $f: \mathbb{F}^d \rightarrow \mathbb{C}$ such that $ |f| \geq \lambda$ (on its support) and $||f||_{L^{(q/\theta)'}}=1$. Then
\[||\hat{f}||_{L^{p'}(S,d\sigma)} \lesssim 1+  |\mathbb{F}|^{-\tilde{d}/4} \lambda^{-\theta/(q-\theta)}.\]
\end{lemma}
\begin{proof}We decompose $(d\sigma)^{\vee} = \delta_{0}+ K$ where $\delta_{0}(x)$ is $1$ for $x=0$ and $0$ otherwise. By Plancherel's identity we have
\[||\hat{f}||_{L^{p'}(S,d\sigma)}^2 \leq ||\hat{f}||_{L^{2}(S,d\sigma)}^2 \leq  |\left<f, f \right>| + |\left<f, f*K \right>|.\]
We then have
\[||\hat{f}||_{L^{p'}(S,d\sigma)}^2  \leq ||f||_{L^2(\mathbb{F}^d, dx) }^2 + |\mathbb{F}|^{-\tilde{d}/2}||f||_{L^1(F^d,dx)}^{2}.\]
Notice that $||f||_{L^2(\mathbb{F}^d, dx) } \leq ||f||_{L^{(q/\theta)'}(\mathbb{F}^d, dx)}=1$. Now $|\text{supp}(f)|^{(q-\theta)/q} \lambda \leq 1$ so $|\text{supp}(f)| \leq \lambda^{-q/(q-\theta)}$. Therefore,
\[ \sum_{x \in \mathbb{F}^d} |f(x)| \leq (\sum_{x \in \mathbb{F}^d} |f(x)|^{q/(q-\theta)})^{(q-\theta)/q} |\text{supp}(f)|^{\theta/q} \leq \lambda^{-\theta/(q-\theta)}.  \]
Thus:
\[||\hat{f}||_{L^{p'}(S,d\sigma)} \lesssim 1+  |\mathbb{F}|^{-\tilde{d}/4} \lambda^{-\theta/(q-\theta)}   \]
which completes the proof of the claim.
\end{proof}

\begin{lemma}\label{Lem:STinfty}Let $f: \mathbb{F}^d \rightarrow \mathbb{C}$, $p,q \geq 2$ and $0\leq \theta \leq 1$. Furthermore, assume $||f||_{L^{\infty}} \leq \lambda$ and $||f||_{L^{(q/\theta)'}}=1$. Then
\[ ||\hat{f}||_{L^{p'}(S,d\sigma)} \leq \mathcal{R}^{*}(p\rightarrow q) \lambda^{(1-\theta)/(q-\theta)}.\]
\end{lemma}

\begin{proof}Letting $||f||_{L^{(q/\theta)'}}=1$ and using that $(q/\theta)' = \frac{q}{q-\theta}$ we have:
\[||f||_{q'} =  \left(\sum_{x \in \mathbb{F}^d } |f(x)|^{q/(q-1)}\right)^{(q-1)/q} \leq \left(\sum_{x \in \mathbb{F}^d } |f(x)|^{q/(q-\theta)} \lambda^{q/(q-1) - q/(q-\theta)}\right)^{(q-1)/q} \leq \lambda^{(1-\theta)/(q-\theta)} \]
thus
\[||\hat{f}||_{L^{p'}(S,d\sigma)} \leq \mathcal{R}^{*}(p \rightarrow q) \lambda^{(1-\theta)/(q-\theta)}.\]
\end{proof}
One may easily calculate that $\mathcal{R}^{*}(2 \rightarrow 2) = \left( \frac{|\mathbb{F}|^d}{|S|} \right)^{1/2}$ for any codimension $1$ surface $\mathcal{S}$. Thus we have the following:
\begin{lemma}\label{lem:RestL2} Let $S$ be a codimension $1$ surface in $\mathbb{F}^d$ with surface measure $d\sigma$. Then
$$||\hat{F}||_{L^2(S,d\sigma)} \lesssim |\mathbb{F}|^{1/2}||f||_{L^2(\mathbb{F}^d,dx)}.$$
Moreover, if $f \sim 1$ on its support $E$ such that $|E|=|\mathbb{F}|^{\gamma}$, then
$$ ||\hat{F}||_{L^2(S,d\sigma)} \lesssim  ||f||_{L^{\frac{2\gamma}{1+\gamma}}(\mathbb{F}^d,dx)}.$$
\end{lemma}
Taking $\lambda= | \mathbb{F}|^{-\tilde{d}(q-\theta)/4} (\mathcal{R}^{*}(p \rightarrow q))^{-(q-\theta)}$ in Lemma \ref{Lem:STinfty} and Lemma \ref{Lem:STsupp}, we recover the formulation of the Stein-Tomas theorem given in \cite{MT}.
\begin{lemma}\label{lem:mtST}Let $p,q \geq 2$ and assume that $S$ has Fourier dimension $\tilde{d} >0$. Then for any $0 < \theta < 1$ we have that
\[\mathcal{R}^{*}(p \rightarrow q/\theta) \lesssim 1 +  \mathcal{R}^{*}(p \rightarrow q)^{\theta} |\mathbb{F}|^{\frac{-\tilde{d}(1-\theta)}{4}}.\]
\end{lemma}
We also find it useful to have the following variants of Lemma \ref{Lem:STinfty} and Lemma \ref{Lem:STsupp}:
\begin{lemma}\label{lem:stdecay}Let $S \subset \mathbb{F}^d$ with surface measure $d\sigma$. Assume $S$ has Fourier dimension $\tilde{d}$. Let $f: \mathbb{F}^d \rightarrow \mathbb{C}$ such that $f \sim 1$ on its support $E \subseteq \mathbb{F}^d$. Moreover, assume that $|E|=|\mathbb{F}|^\gamma$. Then
\begin{equation}\label{eq:decay}
 ||\widehat{f}||_{L^{2}(S,d\sigma)} \lesssim  ||f||_{L^{2}(\mathbb{F}^d, dx)} +||f ||_{L^{\frac{4\gamma}{4\gamma-\tilde{d}}}(\mathbb{F}^d, dx)}.
 \end{equation}
\end{lemma}
\begin{proof}Following the proof of Lemma \ref{Lem:STsupp}, we may bound
$$||\hat{f}||_{L^{2}(S,d\sigma)}^2 \lesssim |\left<f, f*(d\sigma)^{\vee} \right>| \lesssim ||f||_{L^2(\mathbb{F}^d,dx)}^2 + |\left<f, f*K \right>| \lesssim |E| + |\mathbb{F}|^{-\tilde{d}/2} |E|^2.$$
This implies the claim.
\end{proof}

\begin{lemma}\label{Lem:STinfty2}Let $f: \mathbb{F}^d \rightarrow \mathbb{C}$ such that $f \sim 1$ on its support $E\subseteq \mathbb{F}^d$, with $|E| = |\mathbb{F}|^{\gamma}$. Also assume $\mathcal{R}^{*}(p \rightarrow q) \lesssim |\mathbb{F}|^{\alpha}$.  Then
\[ ||\hat{f}||_{L^{p'}(S,d\sigma)} \leq ||f||_{L^{\frac{q\gamma}{q\gamma - \gamma + \alpha q}}(\mathbb{F}^d.dx)} .\]
\end{lemma}
\begin{proof}We have
$$||\hat{f}||_{L^{p'}(S,d\sigma)} \leq |\mathbb{F}|^{\alpha} ||f||_{L^{q'}(\mathbb{F}^d,dx)} = |\mathbb{F}|^{\alpha + \gamma/q'} \lesssim ||f||_{L^{\frac{q\gamma}{q\gamma - \gamma + \alpha q}}(\mathbb{F}^d.dx)} $$
\end{proof}
In addition, we will use the following consequence of Lemma \ref{lem:mtST}:
\begin{lemma}\label{lem:eprem}($\epsilon$-removal lemma) Let $S$ have Fourier dimension $\tilde{d}>0$ and assume that for every $\epsilon >0$ one has $\mathcal{R}^{*}(p\rightarrow q) \lesssim_{\epsilon} |\mathbb{F}|^{\epsilon}$. Then $\mathcal{R}^{*}(p \rightarrow q+\delta) \lesssim_{\delta} 1$ holds for any $\delta>0$.
\end{lemma}

We note that the use of such $\epsilon$-removal lemmas have been previously used in both the finite field \cite{LewkoNew} and Euclidean settings \cite{TaoBR}.

We now prove the equivalence of the restriction problems associated to (surfaces associated to) equivalent quadratic forms. Let $A$, $B$ and $M$ be invertible $(d-1) \times (d-1)$ dimensional matrices over $\mathbb{F}$, such that $A = M^{T}B M$. In other words the quadratic forms associated to $A$ and $B$ are equivalent. Consider the non-degenerate quadratic surfaces
$$\mathcal{S} := \{(x, x^{T} A x) \} $$
and
$$\mathcal{U} := \{(x, x^{T} B x) \}.$$
We show that the restriction problems associated to $\mathcal{S}$ and $\mathcal{U}$ are equivalent.
\begin{lemma}\label{lem:ParEqu}For any $p,q \geq 0$, the claim that the inequality
\begin{equation}\label{eq:Equiv1}
 || (fd \sigma)^{\vee}||_{L^{p}(\mathbb{F}^d,dx)} \leq C ||f||_{L^{q}(\mathcal{S},d\sigma)}
\end{equation}
holds is equivalent to the claim that the inequality
\begin{equation}\label{eq:Equiv2}
|| (gd \sigma)^{\vee}||_{L^{p}(\mathbb{F}^d,dx)} \leq C ||g||_{L^{q}(\mathcal{U},d\sigma)}.
\end{equation}
holds. Indeed, the optimal constant for each inequality is the same.
\end{lemma}
\begin{proof}We identify functions on $\mathcal{S}$ (respectively $\mathcal{U}$) with functions on $\mathbb{F}^{d-1}$ in the usual manner. The left hand side of (\ref{eq:Equiv1}) is given by
\[ \left( \sum_{x \in \mathbb{F}^{2n}, t \in \mathbb{F} }\left| |\mathbb{F}|^{-2n} \sum_{\xi \in \mathbb{F}^{2n} } f(\xi) e (x \cdot \xi + t \xi^{T} A \xi  ) \right|^p \right)^{1/p}.\]
This is equal to
\[ \left( \sum_{x \in \mathbb{F}^{2n}, t \in \mathbb{F} }\left| |\mathbb{F}|^{-2n} \sum_{\xi \in \mathbb{F}^{2n} } f(M\xi) e (x \cdot M \xi + t (M\xi)^{T} A (M \xi)  ) \right|^p \right)^{1/p}\]
\[= \left( \sum_{x \in \mathbb{F}^{2n}, t \in \mathbb{F} }\left| |\mathbb{F}|^{-2n} \sum_{\xi \in \mathbb{F}^{2n} } f(M\xi) e (M^{T}x \cdot \xi + t \xi^{T} B  \xi  ) \right|^p \right)^{1/p}\]
\[= \left( \sum_{x \in \mathbb{F}^{2n}, t \in \mathbb{F} }\left| |\mathbb{F}|^{-2n} \sum_{\xi \in \mathbb{F}^{2n} } f(M\xi) e (x \cdot \xi + t \xi^{T} B  \xi  ) \right|^p \right)^{1/p}.\]
which is the left side of (\ref{eq:Equiv2}) with $g(\xi) = f(M\xi)$. In addition, note that $||g||_{L^{q}(\mathcal{U},d\sigma)} =||f||_{L^{q}(\mathcal{S},d\sigma)}$. This completes the proof.
\end{proof}

\section{Geometric estimates for the Bochner-Riesz operator}\label{sec:GeoBR}
In this section we will prove estimates for the Bochner-Riesz operator. We let $\mathcal{H} \subseteq \mathbb{F}^{3}$ and $d \sigma$ be defined as above and assume that $\text{char}(\mathbb{F}) >2$. We previously defined the Bochner-Riesz operator as convolution with the kernel $K = (d\sigma)^{\vee}(x_1,x_2,t) - \delta(x_1,x_2,t)$, for $x_1,x_2,t \in \mathbb{F}$. Here it is more convenient to work with the kernel $\tilde{K}(x_1,x_2,t) = (d\sigma)^{\vee}(x_1,x_2,t)$, and the (modified) Bochner-Riesz operator defined by
$$Tf = f * \tilde{K}.$$
For $m \in \mathbb{F}$ let $\mathcal{J}_{m}(x_1,x_2,t)$ denote the characteristic function of the affine subspace containing the subspace $\{(x_1,0,0) : x_1 \in \mathbb{F} \}$ and pointing in the direction $-m$. More formally we define:
$$\mathcal{J}_{m}(x_1,x_2,t):= \delta(x_2/t+m) \text{ for } t\neq 0$$
$$\mathcal{J}_{m}(x_1,x_2,t) := \delta(x_2) \text{ for } t=0. $$
More generally, we consider the set of subspaces $\mathbb{T} := \{\text{supp}\left(\mathcal{J}_{m}(x_1,x_2-x_2',t-t')\right): m \in \mathbb{F}, x_2'\in \mathbb{F} ,t' \in \mathbb{F}\}$  (which we identify with their characteristic functions on $\mathbb{F}^{3}$). We caution that the indexing in the definition of $\mathbb{T}$ does not correspond to unique elements of the set. We will use the symbol $\mathcal{T}$ to denote an element of $\mathbb{T}$. For a given $\mathcal{T}$ we will denote by $\overline{\mathcal{T}}$ the projection of $\mathcal{T}$ (as a set in $\mathbb{F}^3$) onto the $x_2-t$ axes (identified with $\mathbb{F}^2$). Thus $\overline{\mathcal{T}}$ will be a line in $\mathbb{F}^2$. We will further denote the direction of the line $\overline{\mathcal{T}}$ by $m_{\mathcal{T}}$. We will write $\mathcal{T}(x_1,x_2,t)$ for the characteristic function of $\mathcal{T}$, and $\overline{\mathcal{T}}(x_2,t)$ for the characteristic function of $\overline{\mathcal{T}}$.
We now observe that the Bochner-Riesz operator acts in a simple geometric way when applied to $1$-dimensional functions supported on VH lines.
\begin{remark}Notice that the Bochner-Riesz operator is symmetric in the $x_1$ and $x_2$ variables. As a result, the following results regarding functions supported on a line of the form $\{(x_1,\overline{x_2},\overline{t_0}) : x_1 \in \mathbb{F} \}$, have symmetric formulations for functions supported on lines of the form $\{(\overline{x_1},x_2,\overline{t_0}) : x_2 \in \mathbb{F} \}$.
\end{remark}

\begin{lemma}\label{BRolines}Consider the $1$ dimensional function $f(x_1) = \sum_{m \in \mathbb{F}} a_m e(m \cdot x_1) $. Letting $F(x_1,x_2,t) = \delta(x_2-x_2')\delta(t-t') f(x_1)$, we have that
\begin{equation}
  TF(x,y,t)=\sum_{m\in \mathbb{F}^n} a_m  e\left(m  x_1 \right) \mathcal{J}_{m}(x_1,x_2-x_2' ,t- t').
\end{equation}
\end{lemma}
\begin{proof}
Since convolution commutes with translations we may assume that $x_2'=0$, $t'=0$. By linearity it suffices to consider the case $f(x_1) = e(m  x_1)$.
Recall the formula for $K=(d \sigma )^{\vee}$ given in Lemma \ref{lem:FTcomp}.  If $t\neq 0$ we have that
\[T f(x_1,x_2,t) =|\mathbb{F}|^{-1} \sum_{y_{1} \in \mathbb{F} } e(m y_1 - x_2  (x_1-y_1)/t) \]
\[= |\mathbb{F}|^{-1}\sum_{y_{1} \in \mathbb{F} } e\left(y_1 (m + x_2/t) - x_2  x_1/t\right)= \delta(m + x_2/t) e\left( - x_2 x_1/t\right)\]
\[= \delta(m + x_2/t) e\left( m   x_1\right) = \mathcal{J}_m(x_1,x_2,t) e\left( m   x_1\right). \]
If $t = 0$ we have that
\[TF(x_1,x_2,t)  =  \delta(x_2)\delta(t) \sum_{y \in \mathbb{F}} f(x_1-y) \delta(y) = \delta(x_2)\delta(t)F(x_1,x_2,t) .\]
Thus
\[
TF(x_1,0,0)=f(x_1) = \sum_{m\in \mathbb{F}} a_m e(mx_1) =  \sum_{m \in \mathbb{F}} a_m  e\left( m x_1\right) \mathcal{J}_{m}(x_1,0,0)
\]
which completes the proof.
\end{proof}

We now introduce some additional notation. Let $I \subseteq \mathbb{F}^{2}$ and define $\alpha$ by $|I| = |\mathbb{F}|^{\alpha}$. We will write the coordinates of $i \in I$ as $i = (x_{0}^{(i)},s_{0}^{(i)})$. For each $i \in I$ we let $\tilde{f}_i(x_1) = \sum_{m \in \mathbb{F}}a_{m}^{(i)}e(m x_{1}) $ and  $f_i(x_1,x_2,t) = \delta(x_2-x_{0}^{(i)})\delta(t-s_{0}^{(i)}) \tilde{f}(x_1)$. Finally, if $\ell$ is a line in $\mathbb{F}^2$ assume that
$$ |I \cap \ell | \leq |\mathbb{F}|^{u}. $$

\begin{lemma}\label{lem:BR}Assume the notation above, and let $F=\sum_{i \in I} f_i$. We have that
$$||TF||_{L^2(\mathbb{F}^3,dx)} \lesssim |\mathbb{F}|^{(1+u)/2}||F||_{L^2(\mathbb{F}^3,dx)}. $$
\end{lemma}
\begin{proof}Recalling the notation $i = (x_{0}^{(i)},s_{0}^{(i)})$ for $i \in I$, by Lemma \ref{BRolines} we have
$$Tf_i (x_1,x_2,t) = \sum_{\substack{\mathcal{T} \in \mathbb{T} \\ (x_{0}^{(i)},s_{0}^{(i)}) \in \overline{\mathcal{T}}  }} a_{m_\mathcal{T}}^{(i)}  e\left(m_\mathcal{T} x_1 \right) \mathcal{T}(x_1,x_2 ,t).$$
It follows that we may write
$$TF (x_1,x_2,t) = \sum_{\mathcal{T} \in \mathbb{T}} b_{m_\mathcal{T}}  e\left(m_\mathcal{T} x_1 \right) \mathcal{T}(x_1,x_2 ,t)$$
where the coefficients $b_{m_\mathcal{T}}$ are given by the formula
$$b_{\mathcal{T}} = \sum_{\substack{ i \in I \\ (x_{0}^{(i)},s_{0}^{(i)}) \in \overline{\mathcal{T}} }} a_{m_\mathcal{T}}^{(i)}.$$
By orthogonality, using the fact that the set of $\mathcal{T}$ that pass through a single point must all have distinct directions $m_{\mathcal{T}}$, we have
\[ ||TF||_{L^2_{x_1}}^2(x_2,t) = |\mathbb{F}| \sum_{\substack{ \mathcal{T} \in \mathbb{T} \\ (x_2,t) \in \overline{\mathcal{T}}  }} |b_\mathcal{T}|^2.\]
Since there are $|\mathbb{F}|$ points in each line $\overline{\mathcal{T}}$ it follows that
\[ ||TF||_{L^2(\mathbb{F}^3,dx)}^{2} = |\mathbb{F}|^2\sum_{\mathcal{T} \in \mathbb{T} } |b_\mathcal{T}|^2 .\]
From Cauchy-Schwarz and the hypothesis $|I \cap \overline{\mathcal{T}}| \leq |\mathbb{F}|^{u}$, we have that
$$ |b_{\mathcal{T}}|^2 = |\sum_{\substack{ i \in I \\ (x_{0}^{(i)},s_{0}^{(i)}) \in \overline{\mathcal{T}} }} a_{m_\mathcal{T}}^{(i)}|^2 \leq |\mathbb{F}|^{ u} \sum_{\substack{ i \in I \\ (x_{0}^{(i)},s_{0}^{(i)}) \in \overline{\mathcal{T}} }} |a_{m(\mathcal{T})}^{(i)}|^2.$$
Thus
\[ ||TF||_{L^2(\mathbb{F}^3,dx)}^{2} \lesssim |\mathbb{F}|^{2+ u}\sum_{\mathcal{T} \in \mathbb{T} }  \sum_{\substack{ i \in I \\ (x_{0}^{(i)},s_{0}^{(i)}) \in \overline{\mathcal{T}} }} |a_{m(\mathcal{T})}^{(i)}|^2 = |\mathbb{F}|^{2+ u} \sum_{i\in I} \sum_{m \in \mathbb{F}} |a_{m}^{(i)}|^2.  \]
Using the relation $||\tilde{f}_i||_{L^2_{x_1}}^{2} = |\mathbb{F}| \sum_{m \in \mathbb{F}} |a_{m}^{(i)}|^2 $, which implies that $ \sum_{i\in I} \sum_{m \in \mathbb{F}} |a_{m}^{(i)}|^2 = |\mathbb{F}|^{-1}||F||_{L^2(\mathbb{F}^3,dx)}^2 $, we conclude that
\[||TF||_{L^2(\mathbb{F}^3,dx)}^{2} \lesssim |\mathbb{F}|^{1+u}||F||_{L^2(\mathbb{F}^3,dx)}^{2}.   \]
Taking square roots completes the proof.
\end{proof}
We now record the following consequence of Lemma \ref{lem:BR}, which we have formulated in a convenient manner for our application.
\begin{corollary}\label{cor:BR}Let $F$ be a function on $\mathbb{F}^3$ such that $|F| \sim 1$ on its support $E \subseteq \mathbb{F}^3$. Let $|E| = |\mathbb{F}|^{\gamma}$ and assume that $E = \cup_i U_i$ is the union of sets of the form $U_{i}$ where each $U_i$ is supported on a VH line and $|U_i| \geq |\mathbb{F}|^{\beta}$. Moreover, assume that $|P \cap E| \leq |\mathbb{F}|^{\alpha}$ for every VH plane $P$. Then:
$$||TF||_{L^2(\mathbb{F}^3,dx)} \lesssim |\mathbb{F}|^{(1+\alpha- \beta)/2}||F||_{L^2}. $$
In addition, we have that
$$||\widehat{F} ||_{L^2(\mathcal{H},d\sigma)} \lesssim |\mathbb{F}|^{(1+\alpha- \beta)/4} ||F||_{L^2} .$$
\end{corollary}
\begin{proof}We start by proving the first claim. By splitting the function into two parts, we may assume that all of the VH lines in the union defining $E$ are of the same type. By symmetry we may assume that these are all type 1. As above, we let $I$ denote the projection of $E$ onto the $x_2$-$t$ axes. By the hypothesis, each VH plane intersects $E$ on a set of size at most $|\mathbb{F}|^{\alpha}$. Since the intersection of $E$ with a VH plane will contain a point in $U_i$ if and only if it contains all of $U_i$, it follows that
$$ |I \cap \ell| \lesssim |\mathbb{F}|^{\alpha-\beta}  $$
for every line $\ell$ in $\mathbb{F}^2$. Noting that $||F||_{L^2(\mathbb{F}^3,dx)}=|\mathbb{F}|^{\gamma/2}$ the first claim now follows from Lemma \ref{lem:BR}. Applying the first claim, we see that
$$||\widehat{F} ||_{L^2(\mathcal{H},d\sigma)}  = |\left<F , TF \right>|^{1/2} \leq ||F||_{L^2(\mathbb{F}^3, dx)}^{1/2} ||TF||_{L^2(\mathbb{F}^3, dx)}^{1/2} \lesssim |\mathbb{F}|^{(1+\alpha- \beta)/4} ||F||_{L^2(\mathbb{F}^3, dx)}. $$
\end{proof}

\section{An energy estimate for the $3$ dimensional hyperbolic paraboloid}\label{sec:3dEnergy}
Let $E \subseteq \mathcal{H}$. The purpose of this section is to develop an estimate for the additive energy of $E$ defined by
$$\Lambda(E) := \sum_{\substack{a+b = c+d \\ a,b,c,d \in E} }1 = \sum_{\substack{a-d = c-b \\ a,b,c,d \in E} } 1.$$
We wish to understand the ways in which the above quantity can be large. Define $E_j := \{(j,y,jy): (j,y,jy)\in E \}$ and $E^{j} := \{(x,j,xj): (x,j,xj)\in E \}$, that is the intersection of $E$ with a given vertical and horizontal line, respectively. In addition for $a,b,c,d \in E$ we will write $a=(a_1,a_2,a_1\cdot a_2)$, $b=(b_1,b_2,b_1  b_2)$, $c=(c_1,c_2,c_1 c_2)$ and $d=(d_1,d_2,d_1  d_2)$. We then have that

$$\Lambda(E) := \sum_{\substack{a+b = c+d \\ a,b,c,d \in E} }1 \lesssim \sum_{\substack{a-d = c-b \\ a,b,c,d \in E \\ b_1 \neq d_1, b_2 \neq d_2} } 1 +
\sum_{j \in \mathbb{F}} \sum_{\substack{a-d = c-b \\ a,b,c,d \in E \\ b_1 = d_1 = j \\ b_2 \neq d_2} } 1
+\sum_{k \in \mathbb{F}} \sum_{\substack{a-d = c-b \\ a,b,c,d \in E \\ b_2 = d_2 = k \\ b_1 \neq d_1} } 1
+ \sum_{\substack{a-d = c-b \\ a,b,c,d \in E \\b=d \\ a=c}} 1 .$$
Define $\Lambda^{*}(E)$ to be the first quantity on the right side of the above inequality. The purpose of the next lemma is to estimate the second and third terms on the right.
\begin{lemma}With the notation above, we have the estimate
$$ \sum_{\substack{a-d = c-b \\ a,b,c,d \in E \\ b_1 = d_1 = j \\ b_2 \neq d_2} } 1 \lesssim \Lambda(E_{j}).$$
\end{lemma}
\begin{proof}We may rephrase the claim in the following way: if
$$(a_1,a_2,a_1 a_2)-(j,d_2,j d_2) = (c_1,c_2,c_1 c_2)-(j,b_2,j b_2)  $$
then $a_1=c_1 =j$. Considering the coordinates of this relation, we have that
$$a_1 =c_1,\text{   } a_2-d_2 =c_2 -b_2,\text{   } a_1  a_2- c_1  c_2  =  j  (d_2 - b_2 ).  $$
Substituting $a_1 = c_1$ and $d_2 -b_2 = a_2 -c_2$ into the third relation gives
$$ a_1  ( a_2-  c_2)  =  j  (a_2 - c_2 )   $$
which implies that $c_1=a_1=j$ as long as $a_2 - c_2 \neq 0$ (which is implied by the fact that $b_2 \neq d_2$). This completes the proof.
\end{proof}

Let us define the subset $\mathcal{H}^{*} \subset \mathcal{H}$,  by $\mathcal{H}^{*} = \{(\omega_1,\omega_2, \omega_1  \omega_2)  \in \mathcal{H} : \omega_1 \neq 0, \omega_2 \neq 0\} $. That is $\mathcal{H}^{*}$ are the elements of $\mathcal{H}$ with non-zero coordinates.

\begin{proposition}\label{prop:Energy}Let $E \subseteq \mathcal{H}^{*}$ then
$$\Lambda^{*}(E) \lesssim |E|^{5/2}.$$
\end{proposition}
\begin{proof}
We estimate
\[
\Lambda^{*}(E) \leq |E| \max_{b}\sum_{\substack{a-d = \mathcal{H} -b \\ a,d \in E \\ b_1 \neq d_1, b_2 \neq d_2} } 1.
\]
Let $\tau:=(\tau_1,\tau_2)$ (with $\tau_1,\tau_2 \neq 0$) we note that applying the mapping $\tau: (x,y,xy) \rightarrow (\tau_{1}x,\tau_{2}y,\tau_{1}\tau_{2}xy)$ to the set $E$ preserves the number of solutions in the sum above. Thus, after an appropriate application of such a transformation, we may assume that $b=(1,1,1)$, which we will denote as $\mathbf{1}$. The restriction on the sum above now becomes that $d=(d_1,d_2,d_1 d_2)$ satisfies $d_1 \neq 1$ and $d_2 \neq 1$. Elements of $\mathcal{H}-\mathbf{1}$ are of the form $(h_1 -1 ,h_2 -1 , h_1 h_2 -1)$. Thus $(a_1,a_2,a_1 a_2)-(d_1,d_2,d_1 d_2) \in \mathcal{H}-\mathbf{1}$ will occur if and only if \[(a_1 - d_1, a_2 - d_2, a_1 a_2 - d_1 d_2) = \left( (a_1 - d_1 +1)-1, (a_2 - d_2 +1) -1, (a_1 - d_1 +1)(a_2 - d_2 +1) -1 \right) .\]
In other words, we must have:
\[ a_1 a_2 - d_1 d_2 =  (a_1 - d_1 +1)(a_2 - d_2 +1) -1,  \]
\[a_1 d_2 + a_2 d_1 +d_1 +d_2 = 2 d_1 d_2 +a_1 +a_2, \]
\[a_1(d_2 -1) + a_2 (d_1 -1) = 2d_1 d_2 -d_1 -d_2, \]
where we may assume that $d_2-1 \neq 0$ and $d_1 -1 \neq 0$. Now for a fixed $d$ (of this form) consider $\ell(d):= \{(a_1,a_2) \in \mathbb{F}^2 : a_1(d_2 -1) + a_2 (d_1 -1) = 2d_1 d_2 -d_1 -d_2  \}$. Clearly $\ell(d)$ is a line $\mathbb{F}^2$. We wish to understand for how many different choices of $d=(d_1,d_2,d_1 d_2)$ can give rise to the same line under the relation $ d\rightarrow \ell(d)$. The restriction on $d$ ensures that the line will be neither vertical or horizontal.  In other words, each line of this form will have an $x$ and $y$ intercept. Indeed, these points are explicitly given by:
$$(0, (2d_1 d_2 - d_1 -d_2 ) (d_1 -1)^{-1}),$$
$$((2d_1 d_2 - d_1 -d_2 ) (d_2 -1)^{-1}, 0 )$$
These two intercepts will determine the line, unless they both are $0$. This occurs if and only if $2d_1 d_2 -d_1 -d_2 =0$. Let us first assume that this is not the case.  Fix $A \neq 0$ and $B\neq 0$ and consider the equations:
$$(2d_1 d_2 - d_1 -d_2 ) (d_1 -1)^{-1}=A,$$
$$(2d_1 d_2 - d_1 -d_2 ) (d_2 -1)^{-1}=B.$$
From the first equation we deduce that $d_2(2 d_1 -1) = (A+1)d_1 - A$.  Now assume that $d_1 \neq 2^{-1}$. Then $d_2= ((A+1)d_1 -A)(2 d_1 -1)^{-1}$. Substituting this into the second equation we have
$$ ((2d_1 ((A+1)d_1 -A) - d_1(2 d_1 -1) - ((A+1)d_1 -A) )=B((A+1)d_1 -A -2 d_1 +1),  $$
$$2A d_1^2 + (-3A-BA+B)d_1 +(A+AB-B)=0.$$
From the factor theorem we see that there are at most two solutions in $d_1$. Now we consider the case $d_1 = 2^{-1}$. We have then (from the equation involving $B$) that $ -1  =2B(d_2 -1) $ which has only one solution.
Finally, we return to the case in which both intercepts are $0$ which occurs precisely when $2d_1 d_2 -d_1 -d_2 =0$. There are $O(|\mathbb{F}|)$ lines passing through $(0,0)$. To specify such a line we may select its intersection with, say, $x=-1$.  Some algebra shows that this intercept is given by
$$(-1,(d_1-1)^{-1} (2d_1 d_2 -d_1 -1)).$$
Let us set $C=(d_1-1)^{-1} (2d_1 d_2 -d_1 -1))$ (for $C \neq 0$), using the relation $2d_2=2d_1/(2d_1 -1)$ we may conclude that
$$2C d_1^2 -3 d_1 (C+1) +C+1 =0.$$
Again, this equation has at most $2$ solutions in $d_1$.  Since each line $\ell(d)$ has at most two representations, we have
\[\Lambda^{*}(E)  \lesssim |E| \sum_{a, d \in E} 1_{a \in \ell(d)}.\]
We now recall the Cauchy-Schwarz incidence estimate. Let $P \subseteq \mathbb{F}^2$ and $L$ a set of distinct lines in $\mathbb{F}^2$, such that $|P|,|L| \leq N$. Then, one has the incidence estimate
\begin{equation}\label{eq:CSinc}
|\{(x,\ell) \in P \times L :x \in L \} | \lesssim N^{3/2}.
\end{equation}From this, we may conclude that
\[\Lambda^{*}(E)  \lesssim |E|^{5/2}\]
which completes the proof.
\end{proof}
Combining the previous two lemmas, and recalling the trivial estimates $\Lambda(E) \leq |E|^3$ and $ \sum_{\substack{a-d = c-b \\ a,b,c,d \in E \\b=d \\ a=c}} 1 \lesssim |E|^2$ , we may conclude that
\begin{lemma}\label{lem:L52}Let $E \subseteq \mathcal{H}$ then
$$\Lambda(E) \lesssim |E|^{5/2} + \sum_{j \in \mathbb{F}}|E_{j}|^3 + \sum_{k \in \mathbb{F}}|E^{k}|^3.$$
\end{lemma}
Decomposing the complete hyperbolic paraboloid as $\mathcal{H}=\mathcal{H}^{*}\cup \ell_1 \cup \ell_2$ (for isotropic lines $\ell_1$ and $\ell_2$) and applying the quasi-triangle inequality for additive energy (see (\ref{eq:quasitriangleEnergy}), below) we see that the above result holds for the the complete hyperboloid $\mathcal{H}$, after, perhaps, adjusting the implicit constant. We now recall that one may parameterize $E\subseteq \mathcal{H}$ by $\underline{E} \subseteq \mathbb{F}$ using the map $(x_1,x_2) \rightarrow (x_1,x_2,x_1 x_2)$. With this parametrization $\underline{E}_{j}$ and $\underline{E}^{k}$ correspond to $VH$ lines in $\underline{E}$.

\begin{corollary}\label{cor:L4}Let $E \subseteq \mathcal{H}$ such that $|E| \leq |\mathbb{F}|^{\gamma}$. Moreover, assume that $\underline{E}$ is a $\text{VH}(3\gamma/4)$ set and let $f$ be a function such that $|f| \sim 1$ on its support $E$.  Then
\begin{equation}
  || (f d\sigma)^{\vee} ||_{L^4(\mathbb{F}^3,dx)} \lesssim  |\mathbb{F}|^{-5}|\mathbb{F}|^{5\gamma/2}.
\end{equation}
\end{corollary}
\begin{proof}We have that
\[ || (f d\sigma)^{\vee} ||_{L^4(\mathbb{F}^3,dx)}^4 \lesssim \frac{1}{|\mathcal{H}|^4} \sum_{\substack{a,c \in E \\ b,d \in E}} \sum_{x \in \mathbb{F}^3}e(x \cdot (a+b-c-d))  = \frac{|\mathbb{F}|^{3}}{|\mathcal{H}|^4} \Lambda(E) .\]
Using the hypothesis that $E$ is a $VH(3\gamma/4)$ set, we have that $ \sum_{j \in \mathbb{F}}|E_{j}|^3, \sum_{k \in \mathbb{F}}|E^{k}|^3 \lesssim |\mathbb{F}|^{5\gamma/2}$. We may bound the quantity above
\[ \lesssim \frac{|\mathbb{F}|^{3}}{|\mathcal{H}|^4}  |\mathbb{F}|^{5\gamma/2} \lesssim |\mathbb{F}|^{-5} |\mathbb{F}|^{5\gamma/2}. \]
Taking fourth roots completes the proof.
\end{proof}
We now describe how to obtain a slight refinement to the prior estimate in certain cases. These refinements will only be used in proof of Theorem \ref{thm:ResWSP}. The main idea is to replace the use of the trivial incidence estimate (\ref{eq:CSinc}) at the end of the proof of Proposition \ref{prop:Energy} with (a variant of) the Bourgain-Katz-Tao incidence theorem from \cite{BKT}. The incidence estimate in \cite{BKT} is proved in prime order fields, and here we require an estimate in general fields. This was worked out, for essentially the same purpose as we are using it here, in the appendix of \cite{LewkoNew} (see Theorem 19, there). Roughly speaking, the result states that if the inequality (\ref{eq:CSinc}) is nearly sharp, then (after an appropriate transformation) the set of points, $P$, can be effectively contained in the Cartesian product of subfields. Noting that a set $P$ of size $|P|=|\mathbb{F}|^{\gamma}$  for $4/5 - \epsilon < \gamma < 4/5 +\epsilon$ can not be effectively contained in the Cartesian product of subfields (since subfields of a necessary size do not exist), Theorem 19 of \cite{LewkoNew} gives the following lemma.
\begin{lemma}\label{lem:refinedIN}There exists $\delta, \epsilon >0$ such that for every set of points $P$ and lines $L$ in $\mathbb{F}^2$ satisfying $|P| \sim |L| \sim |\mathbb{F}|^{\gamma}$ for $4/5 - \epsilon < \gamma < 4/5 +\epsilon$, we have that
\begin{equation}
|\{(x,\ell) \in P \times L :x \in L \} | \lesssim |\mathbb{F}|^{3\gamma/2 - \delta}.
\end{equation}
\end{lemma}
Using this result, we can obtain the following slight improvement to the estimate above.
\begin{lemma}\label{lem:EnergyRefined}There exists $\delta,\epsilon >0$ such that the following holds. Let $E \subseteq \mathbb{H}$ be a set such that $|E| \leq |\mathbb{F}|^{\gamma}$ with $4/5 - \epsilon < \gamma < 4/5 +\epsilon$.  Moreover, assume that $\underline{E}$ is a $\text{VH}(3\gamma/4-\delta/2)$ set and let $f$ be a function supported on $E$ such that $|f| \sim 1$ on its support.  Then
\begin{equation}
  || (f d\sigma)^{\vee} ||_{L^4(\mathbb{F}^3,dx)} \lesssim  |\mathbb{F}|^{-5/4} |E|^{5/8-\delta}.
\end{equation}
\end{lemma}
\begin{proof}This follows be replacing the application of (\ref{eq:CSinc}) in the proof of Proposition \ref{prop:Energy} with the estimate given up by Lemma \ref{lem:refinedIN}.
\end{proof}
\section{The Mockenhaupt-Tao machine}\label{sec:3dMT}
We now adapt the Mockenhaupt-Tao method \cite{MT} to this setting.
\begin{lemma}\label{lem:mt1} Let $h$ be a function such that $|h| \sim 1$ on its support $E\subset \mathbb{F}^3$. Furthermore, assume that $|E| \lesssim |\mathbb{F}|^{\gamma}$ (with $\gamma \geq 1$) and that $E$ is a VH($(3(\gamma-1)/4$) set. Then
$$||\hat{h}||_{L^{2}(\mathcal{H},d\sigma)} \lesssim |\mathbb{F}|^{11\gamma/16 +1/16}.$$
\end{lemma}
\begin{proof}
Define $h_{z}(x_1,x_2,t) = h(x_1,x_2,t)\delta(t-z)$, and $E_z \subseteq \mathbb{F}^2$ to be the support of $h_z(x_1,x_2,z)$ (for $z$ fixed). Thus $h(x_1,x_2,t) = \sum_{z \in \mathbb{F}} h_{z}(x_1,x_2,t)$. We start by estimating
$$||h_{z} * (d\sigma)^{\vee}||_{L^{4}(\mathbb{F}^3,dx)}.$$
By translation symmetry, we may assume that $z=0$. Recall that $(d\sigma)^{\vee} = \delta_{0} + K$ where $K$ is the Bochner-Riesz kernel associated to $\mathcal{H}$. From its definition we have that $K(\underline{x},t) =  |\mathbb{F}|^{-1} e(-x_2 x_1/t)$ (for $t \neq 0$ and $0$ otherwise). Thus
\[||h_{z} *K ||_{L^{4}(\mathbb{F}^3,dx)} =\left( \sum_{t \in \mathbb{F}, t \neq 0}  \sum_{\underline{x} \in \mathbb{F}^2} \left| |\mathbb{F}|^{-1} \sum_{\underline{y} \in \mathbb{F}^2} h_{0}(\underline{y},0) e( (x_1 - y_1) (x_2 - y_2) /t)    \right|^4  \right)^{1/4}. \]
Using the pseudo-conformal transformation (for $t \neq 0$) $t':=1/t$ and $w:= -(x_2,x_1)/t$  we have that
\[  (x_1 - y_1) (x_2 - y_2) /t = w_1  w_2/t'  + w \cdot y + t' y_1 y_2. \]
Let $\phi : \mathcal{H} \rightarrow \mathbb{F}^2$ denote the map given by $(x,x\cdot x) \rightarrow x$. Given $h_z$, define $\widetilde{h}_{z}$ on $\mathcal{H} $ by $\widetilde{h_z}(x) = h_z(\phi(x))$. We then have that
\[|| h_{z} * K||_{L^{4}(\mathbb{F}^3,dx)} \lesssim |\mathbb{F}| \left( \sum_{(w,t') \in \mathbb{F}^3} |(\widetilde{h_z} d\sigma)^{\vee}(w,t')|^{4}  \right)^{1/4}. \]
By Corollary \ref{cor:L4},
$$\left( \sum_{(w,t') \in \mathbb{F}^3} |(\widetilde{h_z} d\sigma)^{\vee}(w,t')|^{4}  \right)^{1/4} \lesssim |\mathbb{F}|^{-5/4} \left(|E_z|^{5/8} + |\mathbb{F}|^{5(\gamma-1)/8} \right) $$
where the first term on the right appears from considering the case $|E_z| \geq |\mathbb{F}|^{\gamma-1}$ and the second from considering the case $|E_z| \leq |\mathbb{F}|^{\gamma-1}$.
We then have
\[ || h * K||_{L^{4}(\mathbb{F}^3,dx)} \leq  \sum_{z \in \mathbb{F}}|| h_{z} * K||_{L^{4}(\mathbb{F}^3,dx)} \lesssim |\mathbb{F}|^{-1/4} \sum_{z \in \mathbb{F}}|E_{z}|^{5/8}  + |\mathbb{F}|^{5\gamma/8 +1/8}  \]
\[ \lesssim |\mathbb{F}|^{-1/4} \left( \sum_{z \in \mathbb{F}}|E_{z}| \right)^{5/8} |\mathbb{F}|^{3/8} + |\mathbb{F}|^{5\gamma/8 + 1/8} \lesssim |\mathbb{F}|^{5\gamma/8 +1/8}.  \]
By Young's inequality, $||h * \delta_{0}||_{L^4(\mathbb{F}^3, dx)} \lesssim ||h ||_{L^4(\mathbb{F}^3, dx)} \lesssim |\mathbb{F}|^{\gamma/4} $. Thus
\[ || h * (d \sigma)^{\vee}||_{L^{4}(\mathbb{F}^3,dx)} \lesssim |\mathbb{F}|^{5\gamma/8 + 1/8}. \]
By H\"{o}lder's inequality, we have
\[ |\left<h, h * (d \sigma)^{\vee} \right> | \leq  ||h||_{L^{4/3}(\mathbb{F}^3, dx)} || h * (d \sigma)^{\vee}||_{L^{4}(\mathbb{F}^3,dx)} \lesssim |\mathbb{F}|^{11\gamma/8 +1/8} .\]
This may be restated as
\[ ||\hat{h}||_{L^{2}(\mathcal{H},d\sigma)} \lesssim |\mathbb{F}|^{11\gamma/16 +1/16} . \]
\end{proof}
We will now describe a slight refinement which will be used in the proof of Theorem \ref{thm:ResWSP}.
\begin{lemma}\label{lem:mt2}There exist $\delta,\epsilon >0$ such that the following holds. Let $h$ be a function such that $|h| \sim 1$ on its support $E\subset \mathbb{F}^3$, $|E| = |\mathbb{F}|^{\gamma}$. Furthermore, assume that $9/5 - \epsilon < \gamma < 9/5 +\epsilon$ and that $E$ is a VH($(3(\gamma-1 - 2 \delta)/4 $) set. Then,
$$||\hat{h}||_{L^{2}(S,d\sigma)} \lesssim |\mathbb{F}|^{11\gamma/16 +1/16 -\delta}.$$
\end{lemma}
\begin{proof}Let $\delta_0 >0$ be a small real number to be chosen later. We will use the same notation as in the previous proof. We partition $\mathbb{F}= Z_1 \cup Z_2 \cup Z_3$ into a disjoint union of three sets. The set $Z_1$ will consist of $z \in \mathbb{F}$ such that $|E_z| \leq |\mathbb{F}|^{\gamma-1 -\delta_0}$, the set $Z_2$ will consist of $z \in \mathbb{F}$ such that $ |\mathbb{F}|^{\gamma-1-\delta_0} \leq |E_z| \leq |\mathbb{F}|^{\gamma-1+\delta_0}$, and the set $Z_3$ will consist of  $z \in \mathbb{F}$ such that $|E_z| \geq |\mathbb{F}|^{\gamma-1+\delta_0}$. Since $\sum_{z \in Z_1} |E_z|  \leq |\mathbb{F}|^{\gamma-\delta_0}$ we have
\[
\sum_{z \in Z_1 }|| h_{z} * K||_{L^{4}(\mathbb{F}^3,dx)} \lesssim |\mathbb{F}|^{-5/4} \sum_{z \in Z_3 } |E_z|^{5/8} \lesssim |\mathbb{F}|^{5\gamma/8+1/8 - 5\delta_0/8  }   .
\]
On the other hand, since $|Z_3| \leq |\mathbb{F}|^{\gamma-\delta_0} $ we have
\[ \sum_{z \in Z_3 }|| h_{z} * K||_{L^{4}(\mathbb{F}^3,dx)} \lesssim |\mathbb{F}|^{-5/4} \sum_{z \in Z_3 } |E_z|^{5/8} \lesssim |\mathbb{F}|^{5\gamma/8+1/8 - 3\delta_0/8  } .\]
It remains to estimate the contribution from $Z_2$. Here will require $\delta_0$ to be smaller than the $\epsilon >0$ for which Lemma \ref{lem:EnergyRefined} holds. In addition let $\delta_1$ be the $\delta$ given in the hypothesis of Lemma \ref{lem:EnergyRefined}. We then have that
\[ \sum_{z \in Z_2 }|| h_{z} * K||_{L^{4}(\mathbb{F}^3,dx)} \lesssim |\mathbb{F}|^{-5/4} \sum_{z \in Z_2 } |E_z|^{5/8} \lesssim |\mathbb{F}|^{(5/8-\delta_1) \gamma+1/8 } \lesssim |\mathbb{F}|^{5/8 \gamma+1/8 - \delta_2 }, \]
where the last inequality has used the fact that $\gamma$ is restricted to a small range of values. Combining these estimates completes the proof.
\end{proof}
\section{Estimates for planar functions}\label{sec:Planar}
In this section we will derive restriction estimates for functions supported on a hyperplane and, more generally, the union of a small number of hyperplanes.

\begin{lemma}\label{lem:planar3to2}Let $f(y_1,y_2)$ be a function on $\mathbb{F}^2$ and let $\hat{f}(\eta)$ denote the ($2$ dimensional) Fourier transform of $f$. Consider the function (on $\mathbb{F}^3$) defined by $F(x_1,x_2,x_3) := \delta(x_2-ax_3-b)f(x_1,x_3)$ for $a,b \in \mathbb{F}$. Then the ($3$ dimensional) Fourier transform of $F(x_1,x_2,x_3)$ is given by
$$ \hat{F}(\xi_1,\xi_2,\xi_3) = \hat{f}(\xi_1, \xi_3 + a \xi_2) e(-\xi_2 b) .$$
\end{lemma}
\begin{proof} Using the relations
$$f(x_1,x_3) = |\mathbb{F}|^{-2}\sum_{\xi_1,\xi_3 \in \mathbb{F} }\hat{f}(\xi_1,\xi_3) e(x_1 \xi_1 + x_3 \xi_3) $$ and $$\delta(y)= |\mathbb{F}|^{-1}\sum_{\xi_2 \in \mathbb{F} }  e(\xi_2 y )$$ we have that
$$\delta(x_2-ax_3)f(x_1,x_3) = |\mathbb{F}|^{-3}\sum_{\xi_1,\xi_2 \in \mathbb{F} } \hat{f}(\xi_1,\xi_3) e(\xi_1 x_1 +\xi_3 x_3 + \xi_2(x_2-ax_3)  )   $$
$$= |\mathbb{F}|^{-3}\sum_{\xi_1,\xi_2,\xi_3 \in \mathbb{F} } \hat{f}(\xi_1,\xi_3) e(\xi_1 x_1 +\xi_2x_2 + (\xi_3 -a\xi_2) x_3  ) .$$
Applying the change of variables $\eta_1 = \xi_1$, $\eta_2 = \xi_2$, and $\eta_3 = \xi_3 -a \xi_2$, we have
$$\delta(x_2-ax_3)f(x_1,x_3) =|\mathbb{F}|^{-3}\sum_{\eta_1,\eta_2,\eta_3 \in \mathbb{F} } \hat{f}(\eta_1,\eta_3+a\eta_2)e(\eta_1 x_1 + \eta_2 x_2+\eta_3 x_3  ) . $$
Finally, translating $x_2$ by $-b$ gives
$$\delta(x_2-ax_3-b)f(x_1,x_3) = |\mathbb{F}|^{-3}\sum_{\eta_1,\eta_2,\eta_3 \in \mathbb{F} } \hat{f}(\eta_1,\eta_3+a\eta_2)e(-b\eta_2) e(\eta_1 x_1 + \eta_2x_2 +\eta_3 x_3  ).$$
After relabeling variables, this is the claim.
\end{proof}

\begin{lemma}\label{lem:planar}Let  $F:\mathbb{F}^3 \rightarrow \mathbb{C}$ be a function with support $E \subseteq \mathbb{F}^3$. Suppose that $F(x) \sim 1$ for $x \in E$, that $|E|= |\mathbb{F}|^{\gamma}$, and that $E$ has planar entropy at most $e$. For $1\leq p \leq 2$ we have the following estimates. If $\gamma \leq 2$, then
\[ \left( \int_{\xi \in \mathcal{H}} |\hat{F}(\xi)|^{p} d\sigma  \right)^{1/p} \lesssim |\mathbb{F}|^{\gamma-1/p} + |\mathbb{F}|^{\gamma/2 + e/2}.\]
If $\gamma > 2$, then
\[ \left( \int_{\xi \in \mathcal{H}} |\hat{F}(\xi)|^{p} d\sigma  \right)^{1/p} \lesssim   |\mathbb{F}|^{2+ (\gamma-2)/p - 1/p} + |\mathbb{F}|^{\gamma/2 + e/2}. \]
\end{lemma}
\begin{proof}It is no loss of generality to assume that all of the $|\mathbb{F}|^e$ planes in the covering given by the hypothesis are of the same type. If they are not, we may split the function into two functions supported on disjoint sets, apply the triangle inequality and apply the result to each. Moreover, by symmetry, we may assume that all of the planes are type $1$.  A type $1$ VH plane will a have characteristic function of the form
$$\delta(x_2 - a x_3 -b) .$$
We may parameterize the set of type $1$ VH planes by the intercept $b \in \mathbb{F}$ and direction $a \in \mathbb{F}$. Given the covering of size at most $|\mathbb{F}|^{e}$ in the hypothesis, we let $\mathcal{A}$ denote the set of all directions that occur among these planes. For each $a \in \mathcal{A}$ we let $\mathcal{B}(a)$ denote the set of intercepts that occur among the planes with direction $a$.  We may then write
$$F(x_1,x_2,x_3) = \sum_{a \in \mathcal{A}} \sum_{b\in\mathcal{B}(a)} F_{(a,b)}(x_1,x_2,x_3)  $$
where $F_{(a,b)}(x_1,x_2,x_3)$ is supported on a type $1$ VH plane with direction $a$ and intercept $b$. Moreover, by Lemma \ref{lem:planar3to2}, we have $F_{(a,b)}(x_1,x_2,x_3)= \delta(x_2-ax_3-b)f_{(a,b)}(x_1,x_2)$ for some function bivariate function $f_{(a,b)}(x_1,x_2)$, and $\hat{F}_{(a,b)}(\xi_1,\xi_2,\xi_3) = \hat{f}_{(a,b)}(\xi_1, \xi_3 + a \xi_2) e(-\xi_2 b)$.

We note that the planes might intersect, and thus there might be some freedom in the definition of the functions $F_{(a,b)}$. We will only require that this choice is made so that the supports of $F_{(a,b)}$ are disjoint. Any choice of the $F_{(a,b)}$'s of this form will be sufficient. Set $\widehat{F_{(a,b)}}(\xi_1,\xi_2,\xi_3) = \widehat{G_{(a,b)}}(\xi_1,\xi_2,\xi_3) + \widehat{H_{(a,b)}}(\xi_1,\xi_2,\xi_3) $ where
$$ \widehat{G_{(a,b)}}(\xi_1,\xi_2,\xi_3) :=  \delta(\xi_{1}+a)\widehat{F_{(a,b)}}(\xi_1,\xi_2,\xi_3)= \delta(\xi_{1}+a)\hat{f}_{(a,b)}(\xi_1, \xi_3 + a \xi_2) e(-\xi_2 b)  $$
and $\widehat{H_{(a,b)}}= \widehat{F_{(a,b)}}-\widehat{G_{(a,b)}}$. We then have
\[ \int_{\xi \in \mathcal{H}} | \sum_{a \in \mathcal{A}} \sum_{b\in\mathcal{B}(a)} \widehat{G_{(a,b)}}(\xi)|^{p} d\sigma(\xi) = |\mathbb{F}|^{-2}  \sum_{\eta_1,\eta_2 \in \mathbb{F}} |\sum_{a \in \mathcal{A}} \sum_{b\in\mathcal{B}(a)} \delta(\eta_{1}+a) \widehat{f_{(a,b)}}(\eta_1, \eta_1 \eta_2 + a \eta_2 )e(-\eta_2 b)   |^{p}  \]
\begin{equation}\label{eq:plane1}
= |\mathbb{F}|^{-2} \sum_{\eta_2 \in \mathbb{F} }| \sum_{a \in \mathcal{A}}\sum_{b\in\mathcal{B}(a)} \widehat{f_{(a,b)}}(-a, 0 )e(-\eta_2 b)   |^{p}
\end{equation}
Let $E_{(a,b)}$ denote the support of the  the function $F_{(a,b)}$. Using that the supports are disjoint, we have that $\sum_{a,b \in \mathbb{F}}|E_{(a,b)}| = |\mathbb{F}|^{\gamma}$.  Since $|\widehat{f_{(a,b)}}(-a, 0 )| \leq |E_{(a,b)}|$, the quantity (\ref{eq:plane1}) is bounded by
$$|\mathbb{F}|^{-1+p \gamma}.$$
We will use this estimate when $\gamma \leq 2$.  When $\gamma>2$ we can do better by exploiting the fact that the support of $F$ must be spread out over multiple planes.

Assuming $\gamma>2$, let us return to (\ref{eq:plane1}).  Let us write $E_{a} = \cup_{b \in \mathcal{B}(a)} E_{(a,b)}$. We will consider the contribution from terms with $|E_{a}| \leq |\mathbb{F}|^2$ and   $|E_{a}| \geq |\mathbb{F}|^2$ separately. Let us start with the first case. Using, again, $|\widehat{f_{(a,b)}}(-a, 0 )| \leq |E_{(a,b)}|$, we may estimate \eqref{eq:plane1} as
\[ \lesssim   |\mathbb{F}|^{-2}\sum_{\eta_2 \in \mathbb{F} } \sum_{a \in \mathcal{A}}|  E_{a} |^{p}
\lesssim |\mathbb{F}|^{-1}\sum_{a \in \mathcal{A}}|  E_{a} |^{p}\lesssim |\mathbb{F}|^{-1} |\mathbb{F}|^{2(p-1)} \sum_{a \in \mathcal{A}}|  E_{a} | \lesssim   |\mathbb{F}|^{-1 +(2 p-2) + \gamma}.\]

We now consider the contribution from the terms which satisfy $|E_{a}| \geq |\mathbb{F}|^2$. From this hypothesis we can deduce that $|\mathcal{A}| \leq |\mathbb{F}|^{\gamma-2}$.  By H\"{o}lder's inequality we may bound (\ref{eq:plane1}) by
\[\lesssim |\mathbb{F}|^{-2} \sum_{a \in \mathcal{A}} (\sum_{\eta_2 \in \mathbb{F}}| \sum_{b\in\mathcal{B}(a)} \widehat{f_{(a,b)}}(-a, 0 )e(-\eta_2 b)   |^{2})^{p/2} |\mathbb{F}|^{1-p/2}.\]
Using the relation $\sum_{\eta_2 \in \mathbb{F} }| \sum_{b\in\mathcal{B}(a)} \widehat{f_{(a,b)}}(-a, 0 )e(-\eta_2 b)   |^{2}  \leq |\mathbb{F}| \sum_{b \in \mathcal{B}(a)}|E_{(a,b)}|^2$ and the trivial estimate $|E_{(a,b))}| \leq |\mathbb{F}|^2 $, we have that the above is
\[\lesssim |\mathbb{F}|^{-2}|\mathbb{F}|^{p/2} \sum_{a \in \mathcal{A}} ( \sum_{b\in\mathcal{B}(a)} |E_{(a,b)}  |^{2})^{p/2} |\mathbb{F}|^{1-p/2} \lesssim
|\mathbb{F}|^{-1} \sum_{a \in \mathcal{A}} (|\mathbb{F}|^2 \sum_{b\in\mathcal{B}(a)} |E_{(a,b)}  |)^{p/2}\]
\[  \lesssim |\mathbb{F}|^{-1+p} \sum_{a \in \mathcal{A}}  |E_{a} |^{p/2} \lesssim  |\mathbb{F}|^{-1+p} (\sum_{a \in \mathcal{A}}  |E_{a}| )^{p/2} |\mathbb{F}|^{(\gamma-2)(2-p)/2}\]
\[\lesssim |\mathbb{F}|^{-1+p}  |\mathbb{F}|^{\gamma-2 + p} = |\mathbb{F}|^{-1+ (\gamma-2) + 2p}.\]
Next we estimate the contribution from the terms $\widehat{H_{(a,b)}}(\xi)$. Here, we proceed more crudely making no distinction between the roles of the $a$ and $b$ parameters. Recall that $|\mathbb{F}|^e := \sum_{a \in \mathcal{A}} \sum_{b \in \mathcal{B}(a)} 1$ is the total number of possible planes. By the triangle inequality
\begin{equation}\label{ eq:plane2}\left( \int_{\xi \in \mathcal{H}} | \sum_{a \in \mathcal{A}} \sum_{b\in\mathcal{B}(a)} \widehat{H_{(a,b)}}(\xi) |^{p}d\sigma(\xi) \right)^{1/p} \lesssim \sum_{a \in \mathcal{A}}\sum_{b\in\mathcal{B}(a)} \left( \int_{\xi \in \mathcal{H}} | \widehat{H_{(a,b)}}(\xi) |^{p}d\sigma(\xi) \right)^{1/p}.
\end{equation}
We have
\[ \int_{\xi \in \mathcal{H}} | \widehat{H_{(a,b)}}(\xi) |^{p}d\sigma(\xi) =  |\mathbb{F}|^{-2}\sum_{\eta_1 \neq -a } \sum_{\eta_2 \in \mathbb{F} } | \widehat{f_{(a,b)}}(\eta_1, \eta_2 (\eta_1  + a)  )e(-\eta_2 b)   |^{p}. \]
Next, using the fact the restriction that $a \neq \eta_1$, we may make the change of variables $\eta_1=\omega_1$ and $\eta_2(\eta_1 +a)=\omega_2$. This allows us to bound the quantity above by
\[  |\mathbb{F}|^{-2} \sum_{\omega_1,\omega_2 \in \mathbb{F}} |\widehat{f_{(a,b)}}(\omega_1, \omega_2  ) |^{p} \lesssim |\mathbb{F}|^{-2}\left(\sum_{\omega_1,\omega_2 \in \mathbb{F} } |\widehat{f_{(a,b)}}(\omega_1, \omega_2  )|^{2}\right)^{p/2} |\mathbb{F}|^{2-p}.  \]
Using that $\sum_{\omega_1,\omega_2 \in \mathbb{F}} |\widehat{f_{(a,b)}} (\omega_1, \omega_2  )|^{2} = |\mathbb{F}|^2 \sum_{x_1,x_2 \in \mathbb{F}}|f_{(a,b)}(x_1, x_2  )|^{2} \lesssim |\mathbb{F}|^2 |E_{(a,b)}|   $, the above quantity is
\[ \lesssim |E_{(a,b)}|^{p/2}. \]
From this, we have
\[\sum_{a \in \mathcal{A}}\sum_{b\in\mathcal{B}(a)} \left( \int_{\xi \in \mathcal{H}} | \widehat{H_{(a,b)}}(\xi) |^{p}d\sigma(\xi) \right)^{1/p} \lesssim \sum_{a,b \in \mathbb{F}}  |E_{(a,b)}|^{1/2} \lesssim |\mathbb{F}|^{\gamma/2 + e/2} .\]
Putting these estimates together concludes the proof.
\end{proof}

Given $0<\gamma<3$ (which will be the dimension of a set) and $1\leq q \leq 2$ (which will be the right exponent of a restriction inequality) in what follows the quantity  $|\mathbb{F}|^{-\gamma/q}$ will occur repeatedly.  This is the largest value that a function $F$ may take on a level set of size $|\mathbb{F}|^{\gamma}$ if $||F||_{L^{q}(\mathbb{F}^3,dx)}=1$. In what follows we will always take $q= 2p/(2p-1)$. Recall that this is the sharp right exponent for a $L^p$ restriction inequality.

\begin{lemma}\label{lem:planeL}Let $3/2 < p < 2$, $q= 2p/(2p-1)$, and $F$ a function on $\mathbb{F}^3$ such that $||F||_{L^q}=1$. In addition, let $\gamma < 2 $ and suppose that
$$  |\mathbb{F}|^{(1-2p)/p} \leq |F| \leq |\mathbb{F}|^{\gamma (1-2p)/2p} $$
holds on the support of $F$. Moreover, assume that the planar entropy $e$ of the support of $F$ satisfies $e \leq \gamma( p-1)/p$. Then $F$ satisfies the (sharp) restriction estimate
$$||\hat{F}||_{L^p(S,d\sigma)} \lesssim ||F||_{L^q} .$$
\end{lemma}
\begin{proof}Following the standard dyadic pigeonholing method, define $S_i \subseteq \mathbb{F}^3$ to be the set where $F(x) \sim 2^{-i}$. Now $|S_{i}| \leq 2^{i2p/(2p-1)
}$. Letting $1_{S_i}$ denote the characteristic function of the set $S_i$, by the first part of Lemma \ref{lem:planar} we have that
$$ ||\hat{F}||_{L^p(S,d\sigma)} \leq \sum_{i \in \mathbb{N}}  2^{-i} ||2^{i}F 1_{S_{i}}||_{L^{p}}$$

$$\lesssim |\mathbb{F}|^{-1/p} \sum_{\substack{i \\ |\mathbb{F}|^{(1-2p)/p}   \leq 2^{-i} \leq |\mathbb{F}|^{\gamma (1-2p)/2p} }}  2^{-i} 2^{i 2p/(2p-1)} + |\mathbb{F}|^{e/2}\sum_{\substack{i \\ |\mathbb{F}|^{(1-2p)/p}   \leq 2^{-i} \leq |\mathbb{F}|^{\gamma (1-2p)/2p} }}  2^{-i} 2^{i p/(2p-1)}$$
 $$ \lesssim |\mathbb{F}|^{(2p-1)/p \cdot (2p/(2p-1) -1) - 1/p} +|\mathbb{F}|^{e/2 + (p-1)/(2p-1) \cdot \gamma (1-2p)/2p } \lesssim 1. $$
\end{proof}

\begin{lemma}\label{lem:planeS}Let $3/2 < p < 2$, $q= 2p/(2p-1)$, and $F$ a function on $\mathbb{F}^3$ such that $||F||_{L^q}=1$. In addition, let $\gamma > 2 $ and suppose that
$$  |\mathbb{F}|^{\gamma (1-2p)/2p} \leq |F| \leq |\mathbb{F}|^{(1-2p)/p}  $$
holds on the support of $F$. Moreover, assume that the planar entropy $e$ of the support of $F$ satisfies $e \leq 2( p-1)/p$. Then $F$ satisfies the (sharp) restriction estimate
$$||\hat{F}||_{L^p(S,d\sigma)} \lesssim ||F||_{L^q} .$$
\end{lemma}
\begin{proof}As above, define $S_i \subseteq \mathbb{F}^3$ to be the set where $F(x) \sim 2^{-i}$, which implies that $|S_{i}| \leq 2^{i2p/(2p-1)
}$. By the second part of Lemma \ref{lem:planar} we have that
$$ ||\hat{F}||_{L^p(S,d\sigma)} \leq \sum_{i \in \mathbb{N} }  2^{-i} ||2^{i}F 1_{S_{i}}||_{L^{p}}$$

$$\lesssim |\mathbb{F}|^{2-3/p} \sum_{\substack{i \\ |\mathbb{F}|^{\gamma (1-2p)/2p}  \leq 2^{-i} \leq |\mathbb{F}|^{(1-2p)/p}   }}  2^{-i} 2^{i 2/(2p-1)} + |\mathbb{F}|^{e/2} \sum_{\substack{i \\ |\mathbb{F}|^{\gamma (1-2p)/2p}  \leq 2^{-i} \leq |\mathbb{F}|^{(1-2p)/p}   }}  2^{-i} 2^{i p/(2p-1)}$$
 $$ \lesssim |\mathbb{F}|^{2-3/p +(3-2p)/p} + |\mathbb{F}|^{e/2 +(p-1)/(2p-1) \cdot (1-2p)/p }  \lesssim 1 .$$
\end{proof}
\section{Proof of theorem \ref{thm:Res3}}\label{sec:3dthm}
We now are ready to put the above estimates together to prove Theorem \ref{thm:Res3}.
\newtheorem*{thm:Res3}{Theorem \ref{thm:Res3}}
\begin{thm:Res3}In the case of the $3$-dimensional hyperbolic paraboloid $\mathcal{H}$, we have that
\[\mathcal{R}^{*}(9/4 \rightarrow 18/5) \lesssim 1 .\]
\end{thm:Res3}
\begin{proof}We will prove the inequality in its dual form
$$||\widehat{F}||_{L^{9/5}(\mathcal{H},d\sigma)} \lesssim ||F||_{\frac{18}{13}} .$$
Let $F$ be an arbitrary function on $\mathbb{F}^3$ such that $||F||_{\frac{18}{13}}=1$. Our goal is to prove
\[ ||\widehat{F}||_{L^{9/5}(\mathcal{H},d\sigma) }  \lesssim 1. \]
Let $F_{L} = F 1_{ \{x : |f(x)| \geq |\mathbb{F}|^{-13/10}\}}$.  It follows from Lemma \ref{Lem:STsupp} (with $q/\theta = 18/5$) that $||\widehat{F_{L}}||_{L^2(\mathcal{H},d\sigma)} \lesssim ||F||_{\frac{18}{13}}$.  Next, let $F_{s} = F 1_{ \{x : |f(x)| \leq |\mathbb{F}|^{-13/8} \}} $.  Then by Lemma \ref{Lem:STinfty} with $\mathcal{R}(2 \rightarrow 2) \lesssim |\mathbb{F}|^{1/2}$ and $\theta = 5/9$ we have that $||\widehat{F_{s}}||_{L^2(\mathcal{H},d\sigma)} \lesssim ||F||_{\frac{18}{13}}$. Now let $F' = F - F_{L} - F_{s}$.  For $i$ such that  $|\mathbb{F}|^{-13/8} \leq 2^{-i} \leq  |\mathbb{F}|^{-13/10} $ we consider the level sets $S_i = \{x : |F'(x)| \sim 2^{-i} \}$. We decompose this set into a disjoint union as $S_{i} =S^{(1)}_{i} \cup S_{i}^{(2)}$ as follows: let $H_{i} \subset \mathbb{F}^3$ denote the union of every VH line $\ell$ in $\mathbb{F}^3$ for which $|\ell \cap S_{i}| \geq 2^{i 36/65} |\mathbb{F}|^{-1/5}$ (there is some flexibility in this choice, which will be motivated shortly). We then define $S_{i}^{(1)}= S_{i} \cap H$ and $S_{i}^{(2)} = S_{i} \setminus S_{i}^{(1)} $. Let $S^{(1)} = \bigcup_{i} S_{i}^{(1)}$ and $S^{(2)} = \bigcup_{i} S_{i}^{(2)}$ and define $F^{(1)} = F' 1_{S^{(1)}}$ and $F^{(2)} = F' 1_{S^{(2)}}$. First, we estimate the contribution from $F^{(2)}$ using Lemma \ref{lem:mt1}. Since $||F||_{\frac{18}{13}}=1$, we have $|S^{(2)}_{i}| \leq 2^{i18/13}$. Define $\gamma_i$ by $2^{i18/13}= |\mathbb{F}|^{\gamma_i}$. By construction, we have that $S^{(2)}$ is a VH($(2\gamma_{i}-1)/5$) set. The hypothesis of Lemma \ref{lem:mt1} requires that the set be a VH($3(\gamma_{i}-1)/4$) set. However, one has that $(2\gamma_{i}-1)/5 \leq 3(\gamma_{i}-1)/4$ whenever $\gamma_i \geq 11/7$. This is sufficient since we are currently interested in the case that $|\mathbb{F}|^{-13/8} \leq 2^{i} \leq  |\mathbb{F}|^{-13/10}$ which implies that $9/5 \leq \gamma_i \leq 9/4$. Thus, invoking Lemma \ref{lem:mt1}, in the relevant range we have
$$||2^{i}\widehat{F^{(2)} 1_{S_{i}^{(2)}}}||_{L^2(\mathcal{H},d\sigma)} \lesssim 2^{i 18/13 \cdot 11/16}|\mathbb{F}|^{1/16} \lesssim |\mathbb{F}|^{1/16} 2^{i 99/104}.$$
We estimate
\[||\widehat{F^{(2)}}||_{L^2(\mathcal{H},d\sigma)} \lesssim  \sum_{\substack{ i \\ |\mathbb{F}|^{-13/8} \leq 2^{-i} \leq  |\mathbb{F}|^{-13/10}  }}  2^{-i} ||2^{i}\widehat{F^{(2)} 1_{S_{i}^{(2)}}}||_{L^2(\mathcal{H},d\sigma)}   \]
\[\lesssim  |\mathbb{F}|^{1/16} \sum_{\substack{ i \\ |\mathbb{F}|^{-13/8} \leq 2^{-i} \leq  |\mathbb{F}|^{-13/10}  }}  2^{-i 5/104}  \lesssim
|\mathbb{F}|^{1/16} |\mathbb{F}|^{- (13/10) \cdot(5/104)} \sim 1. \]
It remains to estimate the contribution from $F^{(1)}$. We will now need to further decompose the sets $S_{i}^{(1)}$.  Let $\tilde{W_{i}} \subseteq \mathbb{F}^3$ denote a set which is the union of $|\mathbb{F}|^{4/5}$ distinct VH planes which maximizes the quantity $|\tilde{W_{i}} \cap S_{i}^{(1)}|$. We define $W_{i} = \tilde{W_{i}} \cap S_{i}^{(1)}$ and $U_{i} = S_{i}^{(1)} \setminus W_{i}$. Now, for each VH plane we have that $|U_{i} \cap P| \lesssim 2^{i 18/13}|\mathbb{F}|^{-4/5}$. To see this, assign each point in $W_i$ to one of the planes (among the$|\mathbb{F}|^{4/5}$ planes comprising $\tilde{W}_i$) which contain it (assigning points which are contained in multiple planes to a single plane arbitrarily). Since $|S_i^{(1)}| \leq 2^{i18/13}$, at least one plane, say $\tilde{P}$, must have $\leq 2^{i 18/13}|\mathbb{F}|^{-4/5}$ associated points. Now if $|U_{i} \cap P| > 2^{i 18/13}|\mathbb{F}|^{-4/5}$ for some plane $P$, we could replace $\tilde{P}$ with $P$ in the composition of $\tilde{W_{i}}$ to increase the quantity $|\tilde{W_{i}} \cap S_{i}^{(1)}|$. We define
$$W^{s} = \bigcup_{\substack{i \\ |\mathbb{F}|^{-13/8} \leq 2^{-i} \leq  |\mathbb{F}|^{-13/9} }}W_i,$$
 $$W^{L} = \bigcup_{\substack{i \\   |\mathbb{F}|^{-13/9}  \leq 2^{-i} \leq |\mathbb{F}|^{-13/10}}} W_i, $$
$$U = \bigcup_{\substack{i \\ |\mathbb{F}|^{-13/8} \leq 2^{-i} \leq |\mathbb{F}|^{-13/10}}}U_i. $$

Now it follows from Lemma \ref{lem:planeL} (with $p=9/5$, $\gamma=9/5$ and $e=4/5$) that $|| F^{(1)}1_{W^L} ||_{L^{9/5}(\mathcal{H},d\sigma) } \lesssim ||F||_{L^{18/13}(\mathbb{F}^3,dx)}$. Similarly, from Lemma \ref{lem:planeS} (with $p=9/5$, $\gamma=9/4$, and $e=8/9$) we have that $|| F^{(1)}1_{W^S} ||_{L^{9/5}(\mathcal{H},d\sigma) } \lesssim ||F||_{L^{18/13}(\mathbb{F}^3,dx)}$.

Finally, we estimate  $|| F^{(1)}1_{U} ||_{L^{2}(\mathcal{H},d\sigma) }$.  By Corollary \ref{cor:BR} we have that
$$||2^{i}\widehat{F^{(1)} 1_{U_{i}}}||_{L^2(\mathcal{H},d\sigma)} \lesssim |\mathbb{F}|^{-3/20} 2^{i 54/260} 2^{i 18/26} = |\mathbb{F}|^{-3/20} 2^{i 9/10}.$$
We estimate
\[|| F^{(1)}1_{U} ||_{L^{2}(\mathcal{H},d\sigma) } \lesssim \sum_{\substack{ i \\ |\mathbb{F}|^{-13/8} \leq 2^{-i} \leq  |\mathbb{F}|^{-13/10}  }}  2^{-i} ||2^{i}\widehat{F^{(1)} 1_{U_{i}}}||_{L^2(\mathcal{H},d\sigma)}   \]
\[\lesssim |\mathbb{F}|^{-3/20} \sum_{\substack{ i \\ |\mathbb{F}|^{-13/8} \leq 2^{-i} \lesssim  |\mathbb{F}|^{-13/10}  }}  2^{-i/10} \lesssim |\mathbb{F}|^{-1/50} \lesssim 1. \]
This completes the proof.
\end{proof}
\section{Further improvements in $3$-d based on sum-product/incidence theory}\label{sec:3dthmSumPro}
In this section we prove Theorem \ref{thm:ResWSP}, which gives a slight improvement to the previous theorem. Our goal here is simplicity and we do not attempt to optimize the calculations. I believe with more care one can recover essentially the same dependencies on the incidence hypothesis as given in \cite{LewkoNew}.
\newtheorem*{thm:ResWSP}{Theorem \ref{thm:ResWSP}}
\begin{thm:ResWSP}In the case of the $3$-dimensional hyperbolic paraboloid $\mathcal{H}$, there exists a $\delta>0$ such that the following holds
$$\mathcal{R}^{*}\left(\frac{18-5\delta}{8-5\delta} \rightarrow 18/5-\delta\right) \lesssim 1.$$
\end{thm:ResWSP}
\begin{proof}The proof will closely follow the previous proof. We again work in the dual formulation. We prove that there exists a $\delta >0$ such that
$$||\widehat{F}||_{L^{\frac{18+13\delta}{10+26\delta}}(\mathcal{H},d\sigma)} \lesssim ||F||_{\frac{18}{13}+\delta} .$$
The quantity $\delta>0$ will be chosen later. In addition, we will introduce a finite sequence of related small quantities, $\delta_1,\delta_2,\ldots,\delta_7 >0$, where $\delta_i$ will be allowed to be an arbitrary continuous function of $\delta$ as long as $\delta_i \rightarrow 0$ as $\delta \rightarrow 0$. Let $F$ be an arbitrary function on $\mathbb{F}^3$ such that $||F||_{\frac{18}{13}+\delta}=1$. Our goal is to prove
\[ ||\widehat{F}||_{L^{\frac{18+13\delta}{10+26\delta}}(\mathcal{H},d\sigma)}  \lesssim 1. \]

If $\delta>0$ is sufficiently small, then by Lemma \ref{Lem:STsupp} there exists a $\delta_1 >0$ such that the function $F_{L} = F 1_{ \{x : |f(x)| \geq |\mathbb{F}|^{-13/10 +\delta_1} \}}$ satisfies
$$||\widehat{F_{L}}||_{L^2(\mathcal{H},d\sigma)} \lesssim ||F||_{\frac{18}{13}+\delta}.$$ Similarly, by Lemma \ref{Lem:STinfty} (with $\mathcal{R}(2 \rightarrow 2) \lesssim |\mathbb{F}|^{1/2}$) there exists a $\delta_2 >0$ such that the function $F_{s} = F 1_{ \{x : |f(x)| \leq |\mathbb{F}|^{-13/8-\delta_2} \}} $ satisfies
$$||\widehat{F_{s}}||_{L^2(\mathcal{H},d\sigma)} \lesssim ||F||_{\frac{18}{13}+\delta}.$$
Now let $F' = F - F_{L} - F_{s}$.  For $i$ such that  $|\mathbb{F}|^{-13/8-\delta_2} \leq 2^{-i} \leq  |\mathbb{F}|^{-13/10+\delta_1} $ we consider the level sets $S_i = \{x : |F'(x)| \sim 2^{-i} \}$. We decompose this set into a disjoint union as $S_{i} =S^{(1)}_{i} \cup S_{i}^{(2)}$ as follows: let $H_{i} \subset \mathbb{F}^3$ denote the union of every VH line $\ell$ in $\mathbb{F}^3$ for which $|\ell \cap S_{i}| \geq 2^{i 36/65} |\mathbb{F}|^{-1/5}$. We then define $S_{i}^{(1)}= S_{i} \cap H$ and $S_{i}^{(2)} = S_{i} \setminus S_{i}^{(1)} $. Let $S^{(1)} = \bigcup_{i} S_{i}^{(1)}$ and $S^{(2)} = \bigcup_{i} S_{i}^{(2)}$ and define $F^{(1)} = F' 1_{S^{(1)}}$ and $F^{(2)} = F' 1_{S^{(2)}}$.

First, we estimate the contribution from $F^{(2)}$ using Lemma \ref{lem:mt1} and Lemma \ref{lem:mt2}. Since $||F||_{\frac{18}{13}+\delta}=1$, we have $|S^{(2)}_{i}| \leq 2^{i(18/13+\delta)}$. Define $\gamma_i$ by $2^{i(18/13+\delta)}= |\mathbb{F}|^{\gamma_i}$. By construction, we have that $S^{(2)}_{i}$ is a VH($(2\gamma_{i}-1)/5$) set. The hypothesis of Lemma \ref{lem:mt1} requires that the set be a VH($3(\gamma_{i}-1)/4$) set, and the hypothesis of Lemma \ref{lem:mt2} requires that the set be a VH($3(\gamma_{i}-1-\delta)/4$) set. Note that $(2\gamma_{i}-1)/5 \leq 3(\gamma_{i}-1 - \delta)/4$ whenever $\gamma_i \geq 11/7+\epsilon'$ (for some small $\epsilon' >0$, depending only on $\delta$). From this, we see that the (VH) hypothesis of both Lemmas will be satisfied, since we are currently interested in the case that $|\mathbb{F}|^{-13/8-\delta_3} \leq 2^{i} \leq  |\mathbb{F}|^{-13/10+\delta_4}$ which implies that $9/5-\delta_5 \leq \gamma_i \leq 9/4 + \delta_6$. We will further divide this case into two. We partition $S^{(2)} = S_L^{(2)} \cup S_U^{(2)}$ where $S_L^{(2)} := \{x \in \mathbb{F}^3 :  |\mathbb{F}|^{-13/8-\delta_2} \leq |F| \leq  |\mathbb{F}|^{-13/10-\delta^{1/2}}  \} $ and $S_U^{(2)} := \{x \in \mathbb{F}^3 :  |\mathbb{F}|^{-13/10-\delta^{1/2}} \leq |F| \leq |\mathbb{F}|^{-13/10+\delta_4} \} $.  We define $F_{L}^{(2)}$ and $F_{U}^{(2)}$, analogously. Invoking Lemma \ref{lem:mt1}, in the relevant range we have
$$||2^{i}\widehat{F_{L}^{(2)} 1_{S_{i}^{(2)}}}||_{L^2(\mathcal{H},d\sigma)} \lesssim 2^{i (18/13+\delta) \cdot 11/16}|\mathbb{F}|^{1/16} \lesssim |\mathbb{F}|^{1/16} 2^{i (99/104+11\delta/16)}.$$
We estimate
\[||\widehat{F^{(2)}_{L}}||_{L^2(\mathcal{H},d\sigma)} \lesssim  \sum_{\substack{ i \\ |\mathbb{F}|^{-13/8-\delta_2} \leq 2^{-i} \leq  |\mathbb{F}|^{-13/10-\delta^{1/2}  }}}  2^{-i} ||2^{i}\widehat{F^{(2)} 1_{S_{i}^{(2)}}}||_{L^2(\mathcal{H},d\sigma)}   \]
\[\lesssim  |\mathbb{F}|^{1/16} \sum_{\substack{ i \\ |\mathbb{F}|^{-13/8-\delta_2} \leq 2^{-i} \leq  |\mathbb{F}|^{-13/10-\delta^{1/2}  }}}  2^{-i (5/104-11\delta/16)}  \lesssim |\mathbb{F}|^{1/16} |\mathbb{F}|^{ (13/10-\delta^{1/2}) \cdot(-5/104+11\delta/16)}\]
\[\lesssim |\mathbb{F}|^{11 \delta^{3/2}/16 + 143 \delta/160 - 5 \delta^{1/2}/104}  \sim 1 \]
provided $\delta \leq \frac{26167 - 13 \sqrt{4027881}}{28600}$ ($\delta \leq .002$ suffices).
Next we consider $F_{U}^{(2)}$. Here, we apply Lemma \ref{lem:mt2}. Denoting $\epsilon, \tilde{\delta}$ the parameters given in the conclusion of that lemma, provided $\delta$ is small enough compared to $\epsilon$, we have
$$||2^{i}\widehat{F_{U}^{(2)} 1_{S_{i}^{(2)}}}||_{L^2(\mathcal{H},d\sigma)} \lesssim 2^{i (18/13+\delta) \cdot 11/16}|\mathbb{F}|^{1/16- \tilde{\delta}} \lesssim |\mathbb{F}|^{1/16 -\tilde{\delta}} 2^{i (99/104+11\delta/16)}.$$
We estimate
\[||\widehat{F^{(2)}_{U}}||_{L^2(\mathcal{H},d\sigma)} \lesssim  \sum_{\substack{ i \\ |\mathbb{F}|^{-13/10-\delta^{1/2}} \leq 2^{-i} \leq  |\mathbb{F}|^{-13/10+\delta_4}  }}  2^{-i} ||2^{i}\widehat{F^{(2)} 1_{S_{i}^{(2)}}}||_{L^2(\mathcal{H},d\sigma)}   \]
\[|\mathbb{F}|^{1/16 -\tilde{\delta}}  \sum_{\substack{ i \\ |\mathbb{F}|^{-13/10-\delta^{1/2}} \leq 2^{-i} \leq  |\mathbb{F}|^{-13/10+\delta_4}  }}2^{-i (5/104-11\delta/16)}\]
\[\lesssim |\mathbb{F}|^{1/16 -\tilde{\delta}} |\mathbb{F}|^{(13/10-\delta_4)(-5/104+11\delta/16) }  \lesssim |\mathbb{F}|^{143 \delta_4/ 160 + 5 \delta/104 - 11 \delta_4 \delta /16 - \tilde{\delta} } \lesssim 1  \]
provided that $\delta$ and $\delta_4$ are sufficiently small compared to $\tilde{\delta}$.

It remains to estimate the contribution from $F^{(1)}$. We proceed as before. We decompose the sets $S_{i}^{(1)}$.  Let $\tilde{W_{i}} \subseteq \mathbb{F}^3$ denote a set which is the union of $|\mathbb{F}|^{4/5-\delta_7}$ distinct VH planes which maximizes the quantity $|\tilde{W_{i}} \cap S_{i}^{(1)}|$. We define $W_{i} = \tilde{W_{i}} \cap S_{i}^{(1)}$ and $U_{i} = S_{i}^{(1)} \setminus W_{i}$. Now, for each VH plane we have that $|U_{i} \cap P| \lesssim 2^{i 18/13}|\mathbb{F}|^{-4/5}$. We define
$$W^{s} = \bigcup_{\substack{i \\ |\mathbb{F}|^{-13/8-\delta_2} \leq 2^{-i} \leq  |\mathbb{F}|^{-26/(18 +13 \delta)} }}W_i,$$
 $$W^{L} = \bigcup_{\substack{i \\   |\mathbb{F}|^{-26/(18 +13 \delta)}  \leq 2^{-i} \leq |\mathbb{F}|^{-13/10+\delta_1}}} W_i, $$
$$U = \bigcup_{\substack{i \\ |\mathbb{F}|^{-13/8-\delta_2} \leq 2^{-i} \leq |\mathbb{F}|^{-13/10 + \delta_1}}}U_i. $$

As before,  it follows from Lemma \ref{lem:planeL} that $|| F^{(1)}1_{W^L} ||_{L^{\frac{18+13\delta}{10+26\delta}}(\mathcal{H},d\sigma)} \lesssim ||F||_{L^{18/13+\delta}(\mathbb{F}^3,dx)}$. We note that it is to satisfy the hypothesis of Lemma \ref{lem:planeL} that we had to reduce the planar entropy from $4/5$ to $4/5 -\delta_7$. Similarly, from Lemma \ref{lem:planeS} we have that $|| F^{(1)}1_{W^S} ||_{L^{\frac{18+13\delta}{10+26\delta}}(\mathcal{H},d\sigma)} \lesssim ||F||_{L^{18/13+\delta}(\mathbb{F}^3,dx)}$.

Finally, we estimate  $|| F^{(1)}1_{U} ||_{L^{2}(\mathcal{H},d\sigma) }$.  By Corollary \ref{cor:BR} we have that
$$||2^{i}\widehat{F^{(1)} 1_{U_{i}}}||_{L^2(\mathcal{H},d\sigma)} \lesssim |\mathbb{F}|^{-3/20+\delta_7/4} 2^{i 54/260} 2^{i 18/26} = |\mathbb{F}|^{-3/20} 2^{i 9/10}.$$
We estimate
\[|| F^{(1)}1_{U} ||_{L^{2}(\mathcal{H},d\sigma) } \lesssim \sum_{\substack{ i \\ |\mathbb{F}|^{-13/8-\delta_2} \leq 2^{-i} \leq  |\mathbb{F}|^{-13/10 + \delta_4}  }}  2^{-i} ||2^{i}\widehat{F^{(1)} 1_{U_{i}}}||_{L^2(\mathcal{H},d\sigma)}   \]
\[\lesssim |\mathbb{F}|^{-3/20} \sum_{\substack{ i \\ |\mathbb{F}|^{-13/8 -\delta_2} \leq 2^{-i} \lesssim  |\mathbb{F}|^{-13/10 + \delta_4}  }}  2^{-i/10} \lesssim |\mathbb{F}|^{-1/50 + \delta_4/10 +\delta_7/4 } \lesssim 1 \]
provided $\delta_4,\delta_7$ are sufficiently small. This completes the proof.
\end{proof}
\section{Background for higher dimensional arguments}\label{sec:HighBack}
We now develop a higher dimensional version of the $3$ dimensional arguments just presented. While these arguments will be applied to the restriction problem for the paraboloid in odd dimensions, we must study quadratic surfaces in all dimensions due to the inductive nature of the arguments. In this section we recall various properties of quadratic forms over finite fields which we will need later. One might refer to Chapter $1$ of \cite{EKM} for a more thorough introduction to this subject.

Let $A$ be a symmetric $m \times m$ matrix with coefficients in $\mathbb{F}$. We will consider the associated bilinear form $x \circ y: \mathbb{F}^m \times  \mathbb{F}^{m} \rightarrow  \mathbb{F}$ given by
$$x \circ y := x^{T} A y.$$
Associated to this bilinear form is the $m$ dimensional quadratic form $Q(x):= x \circ x$.

We say that $Q$ (respectively, $A$ or the bilinear form $\circ$) is non-degenerate if $A$ has full rank. We say $Q$ is fully degenerate if $Q=0$ identically. We say that $Q$ is partially degenerate and has rank $m$ if $A$ has rank $m < n$. In a finite field, quadratic forms may always be diagonalized (see Theorem 6.21 in \cite{LN}).
\begin{lemma}\label{lem:diag}Let $A$ denote a symmetric matrix with entries in $\mathbb{F}$ and $Q$ the associated  quadratic form. Then $Q$ is equivalent to a diagonalized quadratic form.
\end{lemma}
Given a non-degenerate quadratic form $Q: \mathbb{F}^{d-1} \rightarrow \mathbb{F}$, we define a $d$ dimensional quadratic surface to be the subset $S \subseteq \mathbb{F}^d$ given by
$$\mathcal{S} := \{ x\in \mathbb{F}^{d-1} : (x,x\circ x) \}.$$
The surface $\mathcal{S}$ is a subset of $\mathbb{F}^d$, however it is parameterized by $\mathbb{F}^{d-1}$. It will be convenient to work with both parameterizations of $\mathcal{S}$. Given $x \in \mathcal{S} \subseteq \mathbb{F}^d$ we will denote $\underline{x} \in \mathbb{F}^{d-1}$ to be the first $d-1$ coordinates of $x$.

Given a quadratic form $Q: \mathbb{F}^m \rightarrow \mathbb{F}$, a subspace $V \subseteq \mathbb{F}^{m}$ is called isotropic if $Q(v)=0$ for some $v \in V$ such that $v\neq 0$. A space is called totally isotropic if $Q(v)=0$ for every $v\in V$. An affine subspace $W \subseteq \mathbb{F}^{m}$ is a set of the form $t + V$ where $V$ is a subspace of $\mathbb{F}^m$. We will say an affine subspace is isotropic (respectively, totally isotropic) if $V$ is isotropic (respectively, totally isotropic). The following lemma records some properties of isotropic subspaces.
\begin{lemma}\label{lem:iostropicContain}Consider a non-degenerate quadratic from $Q : \mathbb{F}^{m} \rightarrow \mathbb{F}$. Every totally isotropic subspace (with respect to $Q$) is contained in a maximal isotropic subspace. Moreover, every maximal isotropic subspace has the same dimension.
\end{lemma}
This is a consequence of the Witt decomposition theorem. See Proposition 8.11 in \cite{EKM}. Given $Q$, the dimension of the maximal isotropic subspaces of $\mathbb{F}^{m}$ is called the Witt index of $Q$. The quadratic forms with a given Witt index are characterized by the following propositions. The claims regarding properties of the Witt index can be found in chapter two of \cite{KL}, and the claims regarding equivalence classes of quadratic forms can be found in Section 6.2 of \cite{LN}.
\begin{proposition}\label{lem:WittCompute}Let $Q(x)= x^{T} A x$ be a $2n$ dimensional non-degenerate quadratic form over the field $\mathbb{F}$ of odd characteristic. The following statements hold:
\begin{enumerate}
 \item If $\det(A)$ is a square and $\frac{n(|\mathbb{F}|-1)}{2}$ is even, then the Witt index of $Q$ is $n$, \item If $\det(A)$ is a non-square and $\frac{n(|\mathbb{F}|-1)}{2}$ is odd, then the Witt index of $Q$ is $n$, \item If $\det(A)$ is a non-square and $\frac{n(|\mathbb{F}|-1)}{2}$ is even, then the Witt index of $Q$ is $n-1$, \item If $\det(A)$ is a square and $\frac{n(|\mathbb{F}|-1)}{2}$ is odd, then the Witt index of $Q$ is $n-1$.
\end{enumerate}
Moreover, all forms with the same Witt index are equivalent.
\end{proposition}

\begin{proposition}\label{lem:WittComputeOdd}Let $Q$ be a $d=2n+1$ dimensional non-degenerate quadratic form over the field $\mathbb{F}$ of odd characteristic. The Witt index of $Q$ is $n$ and there are exactly two classes of equivalent forms. One class includes the form $Q(x_1,\ldots,x_{d} )=\sum_{i=1}^{d} x_i^2$, and the other class includes the form $Q(x_1,\ldots,x_{d} ) = a x_{d}^2+  \sum_{i=1}^{d-1} x_i^2$, where $a$ is any non-square in $\mathbb{F}$.
\end{proposition}

Let $A$ be a non-degenerate $m \times m$ symmetric matrix over $\mathbb{F}$, and let  $x\circ y$ be the associated bilinear form. Given a subset $W \subseteq \mathbb{F}^m$ we define the orthogonal complement (with respect to the bilinear form $\circ$) by
$$W^{\perp}_{\circ} := \{ v \in \mathbb{F}^m : v \circ w = 0 \text{ for every } w \in W \}.$$
Assuming that the bilinear form $\circ$ is non-degenerate, one has the formula (see \cite{EKM}, Proposition 1.5)
\begin{equation}
\text{dim}(W)+\text{dim}(W^{\perp}_{\circ})=m.
\end{equation}
Given a subspace $W \subseteq \mathbb{F}^m$ we say that $V$ is a complementary subspace if
$$\mathbb{F}^m = W \oplus V.$$
We caution that over finite fields the notion of an orthogonal complement and complementary subspace do not coincide. For instance, it is possible for the orthogonal complement of a subspace to be itself.

We recall the easily verified fact that if $W\subseteq \mathbb{F}^m$ is a subspace, then
\begin{equation}\label{eq:expOC}
 1_{W^{\perp}_{\circ}}(x) = |W|^{-1} \sum_{w \in W} e\left( x \circ w \right).
\end{equation}
The following lemma is an easy consequence of Witt's decomposition theorem (see Theorem 6.2 in \cite{Clark} for a full proof).
\begin{lemma}\label{lem:isoCompl}Let $A$ be a non-degenerate $2n \times 2n$ symmetric matrix over $\mathbb{F}$, and let $x\circ y$ be the associated bilinear form. Moreover, assume the Witt index of $A$ is $n$. Let $W \subseteq \mathbb{F}^{2n}$ be a $n$ dimensional totally isotropic subspace, with basis $w_1,\ldots,w_n$. Then there exists a disjoint totally isotropic $n$ dimensional subspace $V$ with basis $v_1,\ldots,v_n$ such that $w_i\circ v_j =\delta_{ij}$ for every $1\leq i,j \leq n$.
\end{lemma}
The following lemma easily follows from this result.
\begin{lemma}\label{lem:subFour}Let $A$ be a non-degenerate $2n \times 2n$ symmetric matrix over $\mathbb{F}$, and let $x\circ y$ be the associated bilinear form. Moreover, assume the Witt index of $A$ is $n$. Let $W \subseteq \mathbb{F}^{2n}$ be a $n$ dimensional totally isotropic subspace. Let $V$ be the complementary space given by Lemma \ref{lem:isoCompl}. Then the set of functions, indexed by $v \in V$, given by
$$e(v\circ x) 1_{W}(x) : W \rightarrow \mathbb{C}$$
form an orthonormal basis for the set of complex-valued functions on $W$. Moreover, we may write $f:W \rightarrow \mathbb{C}$ as
$$f(x) = \sum_{m \in V} a_m e(m\circ x)$$
where
$$||f||_{\ell^2(W)}= |W|^{1/2} \left(\sum_{m\in V} |a_m|^2\right)^{1/2}.$$
\end{lemma}
\section{The structure of energetic subsets of quadratic surfaces}\label{sec:HighEnergy}
In this section we develop a structure theorem of subsets of quadratic surfaces with large additive energy.

Given a quadratic surface $\mathcal{S} := \{ \underline{x} \in \mathbb{F}^{d-1} : (\underline{x} ,\underline{x} \circ \underline{x} ) \}$, we will associated to each $x \in \mathcal{S}$ ($x \neq 0$) a hyperplane $H(x)$ in $\mathbb{F}^{d-1}$ as follows
$$H(x) := \{y \in \mathbb{F}^{d-1} : \underline{x}\circ y = \underline{x}\circ \underline{x} \}.$$
\begin{lemma}\label{lem:distPlane}If $x, x' \in \mathcal{S}$ with $x \neq x'$, then $H(x) =H(x')$ occurs if and only if $\underline{x}$ and $\underline{x}'$ are both contained in some $1$ dimensional isotropic subspace (and thus are scalar multiples of each other).
\end{lemma}
\begin{proof}Clearly if $\underline{x}$ and $\underline{x}'$ are multiples of the same isotropic vector then $H(x)=H(x')$. We now consider the converse. From the definition $H(x) := \{y \in \mathbb{F}^{d-1} : \underline{x}\circ y = \underline{x} \circ \underline{x} \}$, it is immediate that $H(x)$ is an affine subspace in $\mathbb{F}^{d-1}$ of dimension $\mathbb{F}^{d-2}$. If $\underline{x}\circ \underline{x}=0$ then $H(x)$ is the subspace $H(x) := \{y \in \mathbb{F}^{d-1} : \underline{x}\circ y = 0 \}$ and it follows that the spaces $H(x)$ and $H(x')$ are distinct for linearly independent $x$ and $x'$.

If $\underline{x}\circ \underline{x} \neq 0$ then $H(x)$ is an affine subspace of the form $t+V$, where $V$ is a $d-2$ dimensional subspace of $\mathbb{F}^{d-1}$ and $t$ is a non-zero vector in the  subspace complementary to $V$. One then has that $t \circ \underline{x} = \underline{x}  \circ \underline{x} $ and that $v \circ x = 0$ of every $v \in V$. It follows that $V$ must be the unique subspace given by the relation $\{v \in \mathbb{F}^{d-1} :  v \circ x = 0 \}$. This implies that if $H(z)=H(x)$, then $\underline{z}$ must be a scalar multiple of $\underline{x} $.  Next observe that if $H(z) = H(x)$ we have that $\underline{x} \circ \underline{x} = \underline{x} \circ \underline{z} = \underline{z} \circ \underline{z}$. Assuming that $\underline{z}=\lambda \underline{x}$, this gives us that $\lambda^2 \underline{x}  \circ \underline{x}  = \lambda\underline{x}  \circ \underline{x}  = \underline{x}  \circ \underline{x} $, which implies that $\lambda=1$. This completes the proof.
\end{proof}
\begin{lemma}\label{lem:AffineChar}Let $\mathcal{S}$ be the quadratic surface associated to a $d-1$ dimensional non-degenerate quadratic form, $Q$. Given an affine subspace $\underline{W} \subseteq \mathbb{F}^{d-1}$, then $\mathcal{S}_{W} := \{(\underline{x},Q(\underline{x})) : \underline{x} \in \underline{W} \} \subseteq \mathcal{S}$ is an affine subspace of $\mathbb{F}^d$ if and only if $\underline{W}$ is totally isotropic.
\end{lemma}
\begin{proof}Let $\underline{W}=\underline{t}+\underline{V}$ where $\underline{V}$ is a totally isotropic subspace and $\underline{t}$ is in the complementary subspace to $\underline{V}$. Let $\underline{w}=\underline{t}+\underline{v}$ for $\underline{v} \in \underline{V}$. We then have that
$$Q(\underline{w}) =(\underline{v}+\underline{t})\circ(\underline{v}+\underline{t}) = \underline{v}\circ \underline{v} + 2 \underline{t}\circ \underline{v} +\underline{t}\circ \underline{t}.$$
Note that $\mathcal{S}_{W}$ will be an affine subspace of $\mathbb{F}^d$ if and only if $Q(\underline{w})$ is an affine linear function of $\underline{v}$ for $\underline{v}\in \underline{V}$. Thus, $\mathcal{S}_{\underline{W}}$ is an affine subspace of $\mathbb{F}^d$ if and only if $\underline{v} \circ \underline{v} =0$ for every $\underline{v} \in \underline{V}$. This implies the claim.
\end{proof}

Given two subsets $A,B \subseteq \mathbb{F}^d$ we will be interested in the additive energy defined by
\begin{equation}\label{def:addEn}
\Lambda(A,B) := \sum_{\substack{a+b = c+d \\ a,c \in A \\ b,d \in B} }1 = \sum_{\substack{a-d = c-b \\ a,c \in A \\ b,d \in B} }1.
\end{equation}
We define the additive energy for a single set $A$ to be the quantity $\Lambda(A):=\Lambda(A,A)$. It isn't hard to show (see \cite{TaoVu}, Corollary 2.10 for a proof) that the additive energy satisfies
\begin{equation}\label{eq:EnergyTIn}
\Lambda(A,B) \leq \sqrt{\Lambda(A)} \sqrt{\Lambda(B)}.
\end{equation}
If $E = \cup_{i \in I} E_i$ is the disjoint union of sets, then one has the inequality
$$ \left( \Lambda(E) \right)^{1/4} = \sum_{i \in I} \left(\Lambda(E_i)\right)^{1/4}.$$
This can be directly seen from the relation of the additive energy and the $L^4$ norm of a trigonometric sum (as in the proof of Corollary \ref{cor:L4}). Alternatively, this may be proven directly using the definition of additive energy and several applications of the triangle inequality.

From this, one sees the following quasi-triangle inequality
\begin{equation}\label{eq:quasitriangleEnergy}
 \Lambda(E) \leq |I|^4 \max_{i \in I} \Lambda(E_i).
\end{equation}
We now show that an additive energy problem for a subset of a quadratic surface, can be transformed into an additive energy problem on an equivalent quadratic surface.
\begin{lemma}\label{lem:eqivEnergy}Let $M$, $N$ and $A$ be $d-1$ dimensional matrices over $\mathbb{F}$ such that $A$ is invertible and  $A^{T} M A = N$. Let $\mathcal{S}$ and $\mathcal{S}'$ be the quadratic surfaces associate to $M$ and $N$, respectively. Let $E \subseteq \mathcal{S}$, and let $U \subseteq \mathcal{S}'$ be defined by $\underline{U}=A^{-1}\underline{E}$. Then
$$\Lambda(E) = \Lambda(U).$$
\end{lemma}
\begin{proof}We have
$$\Lambda(E) = \sum_{\substack{ w+x=y+z \\ w^{T}Mw+x^{T}Mx=y^{T}My+z^{T}Mz \\ w,x,y,z \in \underline{E}}} 1  = \sum_{\substack{ w'+x'=y'+z' \\ w'^{T}N w+x'^{T}Nx'=y'^{T}Ny'+z'^{T}Nz' \\ w',x',y', z' \in A^{-1}\underline{E}}} 1$$
$$ = \Lambda(U). $$
\end{proof}
Given $t \in \mathcal{S}$, we define the associated Galilean transform $\tau_{t}: \mathcal{S} \rightarrow \mathcal{S}$ by
$$ \tau_{t}: (\underline{x}, \underline{x} \circ \underline{x}) \rightarrow \left(\underline{x}+\underline{t}, (\underline{x}+\underline{t})\circ(\underline{x}+\underline{t})\right).$$
When $A,B \subseteq \mathcal{S}$, the additive energy is invariant under Galilean transformations.
\begin{lemma}\label{lem:Gallinv}Let $A,B \subseteq \mathcal{S}$ and $\tau$ a Galilean transformation of $\mathcal{S}$. Then
$$\Lambda(A,B)=\Lambda(\tau(A),\tau(B)).$$
\end{lemma}
\begin{proof}We can reformulate the claim as:
$$\sum_{\substack{\underline{a}+\underline{b} = \underline{c}+d \\ \underline{a} \circ \underline{a} + \underline{b} \circ \underline{b} = \underline{c}\circ \underline{c} +\underline{d}  \circ \underline{d}   \\ \underline{a},\underline{c} \in \underline{A} \\ \underline{b},\underline{d} \in \underline{B} }} 1
= \sum_{\substack{(\underline{a}+ \underline{t}) +(\underline{b}+ \underline{t}) = (\underline{c}+ \underline{t}) + (\underline{d} + \underline{t}) \\ (\underline{a}+t) \circ (\underline{a}+ \underline{t}) + (\underline{b}+t) \circ (\underline{b}+t) = (\underline{c}+ \underline{t})\circ (\underline{c}+ \underline{t}) + (\underline{d}+ \underline{t}) \circ (\underline{d} + \underline{t})  \\ \underline{a},\underline{c} \in A \\ \underline{b},\underline{d}  \in B }} 1.
$$
Note that
$$(\underline{a}+ \underline{t}) \circ (\underline{a}+ \underline{t}) + (\underline{b}+ \underline{t}) \circ (\underline{b}+ \underline{t}) = (\underline{c}+\underline{t})\circ (\underline{c}+\underline{t}) + (\underline{d} + \underline{t}) \circ (\underline{d} + \underline{t})$$
is equivalent to
$$ \underline{a} \circ \underline{a}  + \underline{b} \circ \underline{b} + 2 (\underline{a}+\underline{b})\circ  \underline{t}= \underline{c}\circ \underline{c} + \underline{d}  \circ \underline{d}  + 2 (\underline{c}+d)\circ  \underline{t}.$$
Next, note that $\underline{a}+\underline{b} = \underline{c}+\underline{d} $ if and only if $(\underline{a}+ \underline{t}) +( \underline{b}+ \underline{t}) = (\underline{c}+ \underline{t}) + (\underline{d} + \underline{t})$. Moreover, if $\underline{a}+b=\underline{c}+\underline{d} $ then  $2(\underline{a}+\underline{b}) \circ  \underline{t} = 2 (\underline{c}+\underline{d} )\circ \underline{t} $. Combining these observations verifies the above equality.
\end{proof}
We also find it convenient to record the following related observation.
\begin{lemma}\label{lem:Gallieaninv}Let $A,B \subseteq \mathcal{S}$ and $\tau$ a Galilean transformation of $\mathcal{S}$. Then, for $b \in B$,
$$\sum_{\substack{a-d \in \mathcal{S}-b \\ a \in A \\ d \in B }} 1 = \sum_{\substack{a' -d' \in \mathcal{S}-\tau(b) \\ a' \in \tau(A)  \\ d' \in \tau(B) }} 1 .$$
\end{lemma}
\begin{proof}The first expression counts solutions to  $(\underline{a},\underline{a}\circ \underline{a}) - (\underline{d}, \underline{d}\circ \underline{d}) \in \mathcal{S} -(\underline{b},\underline{b}\circ \underline{b})$ which is equivalent to
$$\left(\underline{a}+\underline{b}-\underline{d}, \underline{a}\circ \underline{a} + \underline{b}\circ \underline{b} -\underline{d}\circ \underline{d}\right) = \left(\underline{a}+\underline{b}-\underline{d},(\underline{a}+\underline{b}-\underline{d})\circ(\underline{a}+\underline{b}-\underline{d}) \right).$$
The second expression counts solutions to $(\underline{a}+\underline{t},(\underline{a}+\underline{t})\circ (\underline{a}+\underline{t}) ) - (\underline{d} +\underline{t}, (d+\underline{t})\circ (\underline{d}+\underline{t})) \in \mathcal{S} -(\underline{b}+\underline{t},(\underline{b}+\underline{t})\circ (\underline{b}+\underline{t})$ which is equivalent to
$$\left(\underline{a}+\underline{b}-\underline{d} +\underline{t}, (\underline{a}+\underline{t})\circ(\underline{a}+\underline{t}) + (\underline{b}+\underline{t})\circ (\underline{b}+\underline{t}) -(\underline{d} +\underline{t})\circ (\underline{d} +\underline{t})\right) $$
$$= \left(\underline{a}+\underline{b}-\underline{d}+\underline{t},(\underline{a}+\underline{b}-\underline{d} +\underline{t})\circ(\underline{a}+\underline{b}-\underline{d} +\underline{t}) \right).$$
Expanding out the bilinear forms, we can rewrite this condition as
$$\left(\underline{a}+\underline{b}-\underline{d}  +\underline{t} , \underline{a}\circ \underline{a} + \underline{b}\circ \underline{b} - \underline{d}  \circ \underline{d}   + 2\underline{t}\circ(\underline{a}+\underline{b}-\underline{d} ) + \underline{t} \circ \underline{t}\right)$$
$$= \left(\underline{a}+\underline{b}-\underline{d} +\underline{t},(\underline{a}+\underline{b}-\underline{d} )\circ(\underline{a}+\underline{b}-\underline{d})+2\underline{t}\circ(\underline{a}+\underline{b}-\underline{d} ) + \underline{t} \circ \underline{t} \right).$$
Subtracting $(\underline{t},2\underline{t}\circ(\underline{a}+\underline{b}-\underline{d}) + \underline{t}\circ \underline{t})$ from each side, we see this is satisfied if and only if the first expression is.
\end{proof}

We now prove that one can control the additive energy of two subsets of $\mathcal{S}$ by the number of solutions to a related incidence problem. First we need some additional notation. Given $\mathcal{S}$ a $d$ dimensional quadratic surface and subsets $A,B \subseteq \mathcal{S}$, let $L_{A}:=\{H(a): a \in A \}$ denote the multiset of hyperplanes in $\mathbb{F}^{d-1}$ associated to $A$ and $P_{B} = \{\underline{b} : b \in B\}$ the set of points in $\mathbb{F}^{d-1}$ associated to $B$. Let $I(L_{A},P_{B}) := \{(p,l) \in P_{B} \times L_{A} : p \in l \}$ denote the set of incidences between $L_{A}$ and $P_{B}$.

\begin{lemma}\label{lem:enToInc}Let $\mathcal{S}$ be a $d$ dimensional non-degenerate quadratic surface and let $A,B \subseteq \mathcal{S}$. There exists a Galilean transformation $\tau(\cdot)$ such that $A' = \tau(A)$ and $B' = \tau(B)$ satisfies
$$\Lambda(A,B) \lesssim |L_{B'}| |I(L_{B'},P_{A'})|.$$
\end{lemma}
\begin{proof}
$$\Lambda(A,B) := \sum_{\substack{a+b = c+d \\ a,c \in A, b,d \in B} }1 = \sum_{\substack{a-d = c-b \\ a,c \in A, b,d \in B} }1  $$
$$\leq |L_{B'}| \max_{b \in B} \sum_{\substack{a-d = c-b \\ a,c \in A, d \in B} }1 \leq |L_{B'}| \max_{b \in B} \sum_{\substack{a-d = \mathcal{S} -b \\ a \in A, d \in B} } 1.$$
Let $\tau$ denote the Galilean transformation that satisfies $\tau(b)=0$. Applying Lemma \ref{lem:Gallieaninv}, the above quantity is then bounded by
$$\Lambda(\tau(A),\tau(B)) \leq \sum_{\substack{a-d \in \mathcal{S} \\ a \in \tau(A), d \in \tau(B)} } 1.$$
The condition $a-d \in \mathcal{S}$ is equivalent to
$$Q(\underline{a}) - Q(\underline{d}) = Q(\underline{a} -\underline{d})$$
or
$$\underline{a}\circ \underline{a} -\underline{a}\circ \underline{d} = (\underline{a} -\underline{d})\circ (\underline{a} -\underline{d}).$$
Rearranging, we have
$$\underline{a}\circ \underline{d} = \underline{d}\circ \underline{d}.$$
This is now equivalent to $\underline{a} \in H(\underline{d})$.
\end{proof}
We now prove a variant of the so-called Cauchy-Schwarz incidence estimate.
\begin{lemma}\label{lem:doublecount}Let $L$ denote a multi-set of subsets of $\mathbb{F}^n$. Let $D$ denote the set of distinct elements of $L$. For $\ell \in D$ let $N_{\ell}$ denote the multiplicity with which $\ell$ occurs in $L$, and assume that $N_{\ell} \leq C_2$ for all $\ell$. Let $P$ denote a set of distinct points in $\mathbb{F}^n$. Finally, assume that $|\ell \cap \ell' \cap P| \leq C_1 $ if $\ell,\ell' \in D$. One then has the incidence estimate
$$|I(L,P)| \leq  C_1^{1/2}|P|^{1/2}|L| + C_2 |P|.$$
\end{lemma}
\begin{proof}Let
$$|I(L,P)| = \sum_{p \in P} \sum_{\ell \in L}  1_{p \in \ell} \leq |P|^{1/2} (\sum_{p \in P} |\sum_{\ell \in D} 1_{p \in \ell} N_{\ell}  |^2  )^{1/2} $$
$$\leq |P|^{1/2} (\sum_{\substack{\ell, \ell' \in D \\ \ell \neq \ell'}} |\{p \in P : p \in \ell, p \in \ell' \}|N_{\ell} N_{\ell'} + \sum_{\ell \in D} N_{\ell}^2 |\{p \in P : p \in \ell \}| )^{1/2}.$$
Since $I(L,P)=\sum_{\ell \in D}N_{\ell}|\{p \in P : p \in \ell \}| $, the above is
$$\leq |P|^{1/2} (\sum_{\substack{\ell, \ell' \in D \\ \ell \neq \ell'}} |\{p \in P : p \in \ell, p \in \ell' \}|N_{\ell} N_{\ell'} + C_2 |I(L,P)| )^{1/2}.$$
Thus,
$$|I(L,P)|  \leq |P|^{1/2} ( C_1|L|^2 + C_2 |I(L,P)|)^{1/2}.$$
Squaring both sides and dividing by $|I(L,P)|$ gives
$$|I(L,P)| \leq \frac{ C_1|P| |L|^2 + C_2 |P| |I(L,P)|}{ |I(L,P)|} $$
Assume that $|I(L,P)| \geq C_1^{1/2}|P|^{1/2}|L| + C_2 |P|$, otherwise the result is proven. However we then have,
$$|I(L,P)| \leq \frac{ C_1 |P| |L|^2}{ C_1^{1/2}|P|^{1/2}|L|} + C_2|P|  \leq  C_1^{1/2}|P|^{1/2}|L| + C_2 |P|.$$
This completes the proof.
\end{proof}

\begin{lemma}\label{lem:energyInc1}Let $A,B \subseteq \mathcal{S}$ be subsets of a $d$ dimensional non-degenerate quadratic surface. Assume that $|\underline{A} \cap W| \leq C_1$ for every $d-3$ dimensional affine subspace $W$. In addition, assume that $|\underline{A} \cap j|\leq C_2$ for every affine isotropic line $j$. Then,
$$\Lambda(A,B) \lesssim C_1^{1/2}|B|^{2}|A|^{1/2} + C_2 |A| |B|.$$
\end{lemma}
\begin{proof}Applying Lemma \ref{lem:enToInc}, there exists $\tau(\cdot)$ such that for $A' = \tau(A)$ and $B' = \tau(B)$, one has
$$\Lambda(A,B) \lesssim |L_{B'}| |I(L_{B'},P_{A'})|.$$
From Lemma \ref{lem:distPlane}, the second assumption in the hypothesis implies that each element of the multiset $L_{U'}$ occurs with multiplicity at most $C_2$. Note if $\tau(\cdot)$ is a Galilean transformation, then $\underline{\tau(A)}$ (respectively, $\underline{\tau(B)}$) is a translation of $\underline{A}$ (respectively, $\underline{B}$). The assumption that $|\underline{A} \cap W| \leq C_1$ holds for every $d-3$ dimensional affine subspace $W$, then implies that $|\underline{\tau(A)} \cap W| \leq C_1$ also for every $d-3$ dimensional $W$.  Since the intersection of two distinct hyperplanes $\ell,\ell' \in L_{B'}$ is an affine subspace of dimension at most $d-3$ (if it is non-empty), it follows that $|\ell \cap \ell' \cap P_{A'}| \leq C_1$. Invoking Lemma \ref{lem:doublecount}, gives us that
$$\Lambda(A,B) \lesssim C_1^{1/2}|B|^{2}|A|^{1/2} + C_2 |A||B|.$$
This completes the proof.
\end{proof}

We say that a surface $\mathcal{S}$ has energy exponent $\Psi$ if the following holds: for every $E \subseteq \mathcal{S}$ such that $|E \cap j| \leq |E|^{\alpha}$ holds for every maximal totally isotropic affine subspaces $j$, then $\Lambda(E) \leq |E|^{\Psi(\alpha)}$. We will always have that $\Psi(\alpha): [0,1] \rightarrow [0,3]$ is continuous, increasing and satisfies  $\Psi(\alpha) <3$ for $\alpha <1$ and $\Psi(1)=3$. We now consider $3$ dimensional quadratic surfaces:

\begin{lemma}\label{lem:3denergy}Let $E \subseteq \mathcal{S}$ be a subset of a $3$ dimensional non-degenerate quadratic surface. If the Witt index of the associated quadratic form is $0$ then
$$\Lambda(E) \lesssim |E|^{5/2}.$$
If the Witt index of the associated quadratic form is $1$ and $|E \cap j| \leq C$ for every isotropic line $j$, then
$$\Lambda(E) \lesssim |E|^{5/2} + C |E|^{2}.$$
In particular, if $|E \cap j| \leq |E|^{1/2}$ then
$$\Lambda(E) \lesssim |E|^{5/2}.$$
\end{lemma}
\begin{proof}We will apply Lemma \ref{lem:energyInc1}. Since the only $d-3$ dimensional subspaces are points, we may take $C_1=1$ in Lemma \ref{lem:energyInc1}. If the Witt index is $0$, then there are no isotropic lines and we may further take $C_1=1$. This immediately gives the first result.

Next assume that the Witt index is $1$. Taking $C_2=C$, which follows from the hypothesis gives the second result. The final result follows from observing that  $C|E|^{2} \leq |E|^{5/2}$ if $C \leq |E|^{1/2}$.
\end{proof}
It turns out that with more care, one can improve the condition $|E \cap j| \leq |E|^{1/2}$ in the last part of the previous lemma to $|E \cap j| \leq |E|^{3/4}$. Indeed this is an immediate consequence of Lemma \ref{cor:L4}, after recalling that any $3$ dimensional non-degenerate quadratic surface with Witt index $1$ is equivalent to the $3$ dimensional hyperbolic paraboloid. Formally, this result can be stated as follows.
\begin{lemma}\label{lem:strong3dw1}Let $E \subseteq \mathcal{S}$ be a subset of a $3$ dimensional non-degenerate quadratic surface with Witt index $1$. If $|E \cap j| \leq |E|^{3/4}$ for every isotropic line $j$, then
$$\Lambda(E) \lesssim |E|^{5/2}.$$
If one assumes the weaker hypothesis that $|E \cap j | \leq |E|^{\alpha}$ for $3/4 \leq \alpha$, then one has that
$$\Lambda(E) \lesssim |E|^{1+2\alpha}.$$
In other words, $\mathcal{S}$ has energy exponent $\Psi(\alpha) =1+2\alpha$.
\end{lemma}
Roughly speaking, note that proof (implicitly) exploits the fact that a $2$ dimensional quadratic form with Witt index $1$ has only a fixed number (indeed, two) of isotropic subspaces. The analogous statement isn't true for higher dimensional forms, which makes the proofs of the higher dimensional results below more complicated and less efficient. Also, recall that it is also possible to reduce the exponent $5/2$ in certain situations by appealing to sum-product/incidence theory.

Before moving on to the case of $5$ dimensional quadratic surfaces, we must first take a brief detour to consider $2$ dimensional quadratic surfaces. These are surfaces  of the form
$$\mathcal{S} := \{ (\underline{x},mx^2) : \underline{x} \in \mathbb{F} \}.$$
with $m \neq 0$. We have that
\begin{lemma}\label{lem:2dEnergy}Let $E \subseteq \mathcal{S}$ be a subset of a non-degenerate $2$ dimensional quadratic surface. Then
$$ \Lambda(E) \lesssim |E|^{2}.$$
\end{lemma}
\begin{proof}From the definition of $\Lambda(E)$, we see that this quantity is equal to the number of simultaneous solutions to the equations
$$a+b=c+d,  a^2+b^2 = c^2+d^2$$
with $a,b,c,d \in E$. From the relation $a=c+d-b$, this is equivalent to counting solutions to the equation
$$(c+d-b)^2 + b^2 = c^2 +d^2$$
with $b,c,d \in E$.  Simplifying gives us the equation $c(d-b)= bd-b^2$. It follows that there is at most one $c$ that satisfies this equation, for fixed $b$ and $d$ with $b\neq d$. Thus, the number of solutions with $(b,c,d)$ with $b\neq d$ is at most $|E|^2$.  On the other hand, the number of solutions with $b=d$ is also at most $|E|^2$. This proves the lemma.
\end{proof}
We now show how to reduce problems about degenerate quadratic surfaces to problems about non-degenerate quadratic surfaces in lower dimensions.

\begin{lemma}\label{lem:degEnergy}Let $Q$ be an $r$ dimensional non-degenerate quadratic form. Assume that the surface associated to $Q$
 $$\mathcal{S} := \{(\underline{x},Q(\underline{x})) : \underline{x} \in \mathbb{F}^r \}$$
has energy exponent $\Psi(\alpha)$. For $s  \in \mathbb{N}$, let $d=s+r$. The degenerate quadratic surface
$$\mathcal{U} := \{(\underline{x},\underline{y},Q(\underline{x})) : \underline{x} \in \mathbb{F}^r, \underline{y} \in \mathbb{F}^{s} \} $$
has energy exponent $\theta(\alpha) = 3\alpha + \Psi(\alpha)(1-\alpha)$.
\end{lemma}
\begin{proof}Restating the hypothesis, if $E \subseteq \mathcal{S}$ such that $|\underline{E} \cap j| \leq |E|^{\alpha}$ ($\alpha\leq 1$) for every affine totally isotropic subspace $j$, one has that $\Lambda(E) \lesssim |E|^{\beta}$ where $\beta = \Psi(\alpha)$. Recall that an energy exponent is assumed to be monotonically increasing.

Let $\gamma$ be defined by $|A| = |\mathbb{F}|^\gamma$. Express the points in $\underline{A}$ in the form $(w,a) \in \mathbb{F}^{r} \times \mathbb{F}^{s}$. Let $\underline{A}_{0} \subseteq \mathbb{F}^{r}$ be the set of $w \in \mathbb{F}^r$ such that $(w,a) \in \underline{A}$ for some $a \in \mathbb{F}^s$. For every $w \in \underline{A}_{0}$, let $B_{w}= \{ a \in \mathbb{F}^s : (w,a) \in \underline{A}\}$. Let us partition $\underline{A}$ into $\lesssim \log(|E|)$ sets $\underline{A}^{(i)}$ where
$$\underline{A}^{(i)} := \{(w,a) :|B_{w}| \sim 2^{i}\}. $$
Define $\underline{A}_{0}^{(i)} \subseteq \mathbb{F}^{r}$ to be the set of $w \in \mathbb{F}^r$ such that $(w,a) \in \underline{A}^{(i)}$ for some $a \in \mathbb{F}^s$.
Let $|\underline{A}^{(i)}|= |\mathbb{F}|^{\gamma_i}$, where $\gamma_i =\theta_{i}+\omega_{i}$ with $|\mathbb{F}|^{\theta_{i}} \sim 2^{i}$ and $|\underline{A}_{0}^{(i)}| \sim |\mathbb{F}|^{\omega_{i}}$.  It follows that for any affine totally isotropic subspace $j$ one has $|\underline{A_{0}}^{(i)} \cap j | \lesssim |\mathbb{F}|^{\alpha (\gamma_i -\theta_i)} \lesssim |\underline{A_{0}}^{(i)} |^{\alpha}$. By the hypothesis, we then have $\Lambda( A^{(i)}_{0}) \lesssim |A^{(i)}_{0}|^{\beta}$. Now, consider the additive energy of each of these sets:
$$\Lambda(A^{(i)}) = \sum_{\substack{ w+x=y+z \\Q(w)+Q(x)=Q(y)+Q(z)\\ a+b=c+d \\ (w,a),(x,b),(y,c),(z,d) \in \underline{A}^{(i)}}} 1.$$
Rearranging the sum, we have that
$$\Lambda(A^{(i)}) = \sum_{\substack{ w+x=y+z \\Q(w)+Q(x)=Q(y)+Q(z)\\  w,x,y,z \in \underline{A}^{(i)}}} \sum_{\substack{ a+b=c+d \\ a \in B_{w} \\ b\in B_{x} \\ c \in B_{y} \\ d \in B_{z} }} 1.$$
$$\lesssim  \Lambda( A^{(i)}_{0}) |\mathbb{F}|^{3 \theta_i} \lesssim |\mathbb{F}|^{ (3- \beta) \theta_i + \beta \gamma  }. $$
It follows from the hypothesis that $\theta_i \leq \alpha \gamma$, thus we have
$$\Lambda(A^{(i)})  \lesssim |\mathbb{F}|^{ \gamma( 3\alpha  + \beta(1-\alpha) )} = |A|^{3\alpha  + \beta(1-\alpha)}.$$
Using \eqref{eq:quasitriangleEnergy}, one has
$$\Lambda(A) \lesssim_{s} \log^{O(1)}(|A|)  |A|^{ 3\alpha  + \beta(1-\alpha) }.  $$
This completes the proof.
\end{proof}

Combining Lemma \ref{lem:3denergy}, \ref{lem:2dEnergy} and \ref{lem:degEnergy} (and using the fact that every quadratic form over a finite field may be diagonalized, Lemma \ref{lem:diag}) gives the following estimates.
\begin{lemma}\label{lem:1r2renergy}Let $\mathcal{S} \subset \mathbb{F}^d$ be a rank $1$ quadratic surface. Let $E \subseteq \mathcal{S}$ such that $|\underline{E}\cap j | \leq |E|^{\alpha}$ for every totally isotropic affine subspace. It follows that
$$\Lambda(E) \lesssim_{d} \log^{O(1)}(|E|) |E|^{2+\alpha}.$$
Let $\mathcal{S} \subset \mathbb{F}^d$ be a rank $2$ quadratic surface, and $E \subseteq \mathcal{S}$ such that $|\underline{E} \cap j | \leq |\underline{E}|^{\alpha}$ with $\alpha \geq 3/4$  for every totally isotropic affine subspace. One then has that
$$\Lambda(E) \lesssim_{d} \log^{O(1)}(|E|) |E|^{1+4\alpha - 2\alpha^2}.$$
\end{lemma}
We will also need the following results regarding the restriction of a quadratic surface to a subspace.

\begin{lemma}\label{lem:restSubSurface}Consider the surfaces $\mathcal{S}:=\{(x,x\circ x) : x \in \mathbb{F}^n \}$. Let $W=V+t$ denote a $m$ dimensional affine subspace of $\mathbb{F}^n$ (where $V$ is a proper subspace). Let $\mathcal{S}_{W}$ and $\mathcal{S}_{V}$ denote the restriction of $\mathcal{S}$ to $W$ and $V$, respectively. If $\Psi$ is an energy exponent for $\mathcal{S}_{V}$ then it is also an energy exponent for $\mathcal{S}_W$.
\end{lemma}
\begin{proof}This follows immediately from  Lemma \ref{lem:Gallinv} once one observes that a Galilean transformation maps totally isotropic affine subspaces to (and from)
 totally isotropic affine subspaces.
\end{proof}

\begin{lemma}\label{lem:classifysubsurface}Let $\mathcal{S}$ be a non-degenerate $d$ dimensional quadratic surface with quadratic form $P$. Let $V$ denote a $d-3$ dimensional subspace of $\mathbb{F}^{d-1}$, and consider $\mathcal{S}_{V}$ to be the restriction of $\mathcal{S}$ to $V$. Furthermore, assume that $\mathcal{S}_{V}$ is not fully degenerate. Then the quadratic form $Q$ associated to $\mathcal{S}_{V}$ will satisfy
$$\max\{1,d-7\} \leq \text{\textnormal{Rank}}(Q) \leq d-3.$$
More specifically, the following tables give the possible values of the rank and Witt index of the associated form. Here $s=d-3-r$ is the dimension of the fully degenerate component of $\mathcal{S}_{V}$. A strikethrough indicates the possibility is eliminated by dimensional considerations.
\end{lemma}

\begin{table}[h!] \label{tab:possibleRanksOddPlus}
\centering
\begin{tabular}{|l|l|l|}  \hline
$\text{\textnormal{Rank}}(Q)$ & $s$ & Possible Witt
index of $Q$ \\
\hline
$d-3$ & $0$ & $\frac{d-3}{2}$ ($+$) or $\frac{d-5}{2}$ ($-$)\\
\hline
$d-4$ & $1$ & $\frac{d-5}{2}$ \\
\hline
$d-5$ & $2$ & $\frac{d-5}{2}$ ($+$) or $\frac{d-7}{2}$ ($-$) \\
\hline
$d-6$ & $3$ & $\frac{d-7}{2}$ \\
\hline
$d-7$ & $4$ & $\cancel{\frac{d-7}{2}}$ ($+$) or $\frac{d-9}{2}$ ($-$) \\
\hline
\end{tabular}
\caption{$d$ odd, $\mathcal{S}$ ($+$) type, (Witt index of $\frac{d-1}{2}$ of $P$)}
\end{table}

\begin{table}[h!] \label{tab:possibleRanksOddMinus}
\centering
\begin{tabular}{|l|l|l|} \hline
$\text{\textnormal{Rank}}(Q)$ & $s$ & Possible Witt
index of $Q$ \\
\hline
$d-3$ & $0$ & $\frac{d-3}{2}$ ($+$) or $\frac{d-5}{2}$ ($-$)\\
\hline
$d-4$ & $1$ & $\frac{d-5}{2}$ \\
\hline
$d-5$ & $2$ & $\cancel{\frac{d-5}{2}}$ ($+$) or $\frac{d-7}{2}$ ($-$) \\
\hline
$d-6$ & $3$ & $\cancel{\frac{d-7}{2}}$ \\
\hline
$d-7$ & $4$ & $\cancel{\frac{d-7}{2}}$ ($+$) or $\cancel{\frac{d-9}{2}}$ ($-$) \\
\hline
\end{tabular}
\caption{$d$ odd, $\mathcal{S}$ ($-$) type, (Witt index $\frac{d-2}{2}$  of $P$)}
\end{table}

\begin{table}[h!]\label{tab:possibleRanksEven}
\centering
\begin{tabular}{|l|l|l|}  \hline
$\text{\textnormal{Rank}}(Q)$ & $s$ & Possible Witt
index of $Q$\\
\hline
$d-3$ & $0$ & $\frac{d-4}{2}$\\
\hline
$d-4$ & $1$ & $\frac{d-4}{2}$ ($+$) or $\frac{d-6}{2}$ ($-$)\\
\hline
$d-5$ & $2$ & $\frac{d-6}{2}$ ($-$) \\
\hline
$d-6$ & $3$ & $\cancel{\frac{d-6}{2}}$ ($+$) or $\frac{d-8}{2}$ ($-$) \\
\hline
$d-7$ & $4$ & $\cancel{\frac{d-8}{2}}$ \\
\hline
\end{tabular}
\caption{$d$ even, (Witt index $\frac{d-2}{2}$  of $P$)}
\end{table}
\begin{proof}The fact that the rank is at most $d-3$ trivially follows from the fact that $\mathcal{S}_{V}$ is a $d-2$ dimensional surface. Similarly, the fact that the rank is at least $1$ follows from the hypothesis that $\mathcal{S}_{V}$ is not fully degenerate. We now consider the lower estimate. Since the largest totally isotropic affine subspace contained in the surface $\mathcal{S}$ can have dimension at most $\frac{d-1}{2}$ or $\frac{d-3}{2}$ by Lemma \ref{lem:WittCompute}, it follows that the dimension of the largest totally isotropic affine subspace contained in $\mathcal{S}_{V}$ is also bounded above by this quantity. On the other hand, if $Q$ has rank $r$ define $s$ by $r+s = d-3$. Then, by Lemma \ref{lem:WittCompute}, the dimension of the largest affine subspace contained in $\mathcal{S}_{V}$ will be either $r/2 - 1 +s = d - r/2-4$ or $r/2 +s = d-r/2-3$ (depending on the Witt index of the non-degenerate part of $\mathcal{S}_{V}$). From this it follows that $d-7 \leq r$. Indeed, considering each possible combination of rank and Witt index in this range gives the results summarized in the tables.
\end{proof}
We will repeatedly apply the two following elementary facts.
\begin{lemma}\label{lem:vectdecomp}Let $E \subseteq \mathbb{F}^{d-1}$. Let $0 < \rho < 1$ be a real number and $0< c < d-1$ a positive integer. We may decompose $E$ as the disjoint union of two sets $E= E_{c} \cup E_{u}$, with the following properties:
\begin{enumerate}
 \item The set $ E_{c} $ can be written as a disjoint union of $E_c =\bigcup_{i \in I} \Omega_i$ where $|\Omega_i| \geq |E|^{\rho}$ and each $\Omega_i$ is contained in an affine subspace of dimension $c$. Moreover, $|I| \leq |E|^{1-\rho}$.
 \item Given any $c$ dimensional affine subspace $j \subset \mathbb{F}^{d-1}$, we have $|E_{u} \cap j| \leq |E|^{\rho}$.
\end{enumerate}
\end{lemma}
\begin{proof}We start by initializing $E_{u}=E$ and $E_{c} =\emptyset$. If for every $c$ dimensional affine subspace $j$, one has  $|j \cap E_{u}| \leq |E|^{\rho}$ the proof is complete. We assume this is not the case. We then select a $c$ dimensional affine subspace $j$ that maximizes $|j \cap E_{u}|$. We let $\Omega_1 = j \cap E_{u}$, and we add the elements of $\Omega_1$ to $E_{c}$ and remove them from $E_{u}$. We then repeat this process until $|j \cap U| \leq |E|^{\rho}$ holds for every $j$. It's clear that $E_{u}$ has the desired property. It is also clear that $E_{c}$ can be covered by the  sets $\Omega_i$. Moreover, since the sets $\Omega_i$ are disjoint and satisfy $|\Omega_i| \geq |E|^{\rho}$ it follows that there can be at most $|E|^{1-\rho}$ such sets. This completes the proof.
\end{proof}
\begin{lemma}\label{lem:energyIn}Let $E \subseteq \mathbb{F}^{d-1}$ has decomposition in terms of disjoint sets $E=\bigcup_{i\in I} \Omega_i$. Furthermore, assume that $\Lambda(\Omega_i) \leq C |\Omega_i|^{\beta}$ for each $\Omega_i$ and some $0 < \beta \leq 4$. Then
$$\Lambda(E) \leq C|I|^{4-\beta} |E|^{\beta}.$$
\end{lemma}
\begin{proof}We estimate
$$\Lambda(E) = \left(\sum_{i \in I} \left(\Lambda(\Omega_i)\right)^{1/4} \right)^{4} \leq \left( \sum_{i \in I} |\Omega_i|^{\beta/4}  \right)^{4} \leq \left( \sum_{i\in I} C^{1/4} |\Omega_i|^{\beta/4}  \right)^{4} \lesssim  C |I|^{4-\beta} |E|^{\beta}.$$
\end{proof}
The above lemmas will be used in the following way. We will be given a non-degenerate $d$ dimensional quadratic surface $\mathcal{S}$ and a subset $E \subseteq \mathcal{S}$. We will then apply Lemma \ref{lem:vectdecomp} to decompose $\underline{E} \subseteq \mathbb{F}^{d-1}$ with respect to $c=d-3$ dimensional affine subspaces. Given $\Omega_i$ in the decomposition given by Lemma \ref{lem:vectdecomp} we consider the surface $\mathcal{S}_{\Omega_i}$. We will write $\Omega^{[r,s,w]}$ to be the union of $\Omega_i$ such that the rank of $\mathcal{S}_{\Omega_i}$ is $r$ and the Witt index of the associated quadratic form is $w$. We included the variable $s=d-3-r$ in the notation for convenience, although it is fully determined by $r$ and $d$. We note that Lemma \ref{lem:classifysubsurface} specifies the finite possible combination of tuples $[r,s,w]$ for a given quadratic surface. We will use the notation $\Omega^{[r,s,w]}_i$ to denote the elements of $\Omega^{[r,s,w]}$.

The results discussed thus far give a complete understanding of surfaces of dimension $3$ and lower. We now consider higher dimensional surfaces.
\begin{lemma}Let $E \subseteq \mathcal{S}$ be a subset of a $4$ dimensional non-degenerate quadratic surface. Assume that $3/5\leq \alpha < 1$ and that $|\underline{E} \cap j| \leq |E|^{\alpha}$ for every affine totally isotropic subspace $j$ of dimension $1$.  Then,
$$\Lambda(E) \lesssim |E|^{5/2+\alpha /2}.$$
In other words, $\Psi(\alpha)= 5/2 + \alpha/2$ is an energy exponent for $\mathcal{S}$.
\end{lemma}
\begin{proof}By Proposition \ref{lem:WittComputeOdd}, the Witt index of $\mathcal{S}$ is $1$. We start by decomposing $\underline{E}=\underline{E_c} \cup \underline{E_u}$ using Lemma \ref{lem:vectdecomp}, with $c=1$ and $\rho=\alpha$. By the hypothesis we see that $|E\cap j| \leq |E|^{\alpha}$ must hold for every isotropic line $\ell$. By Lemma \ref{lem:energyInc1} we have that
$$\Lambda(E_{u}) \lesssim |E|^{5/2+\alpha/2} + |E|^{2+\alpha} \lesssim |E|^{5/2+\alpha/2}.$$
Let $\Omega_i$ be the decomposition of $\underline{E_c}$ from Lemma \ref{lem:vectdecomp}. By the hypothesis, we may assume that none of the $1$ dimensional affine spaces associated to $\Omega_i$ are isotropic. By Lemma \ref{lem:2dEnergy} we then have $\Lambda(\Omega_i) \lesssim |\Omega_i|^{2}$, which by Lemma \ref{lem:energyIn} implies
$$\Lambda(E_c) \lesssim |E|^{4-2\alpha}.$$
Combining the two estimates with (\ref{eq:quasitriangleEnergy}), and using that $4-2\alpha \leq 5/2 + \alpha/2$ for $3/5\leq \alpha <1$ completes the proof.
\end{proof}

We now consider the case of $5$ dimensional quadratic surfaces.
\begin{lemma}\label{lem:5dEnergy}Let $9/16\leq \alpha \leq 1$ and $E \subseteq \mathcal{S}$ be a subset of a $5$ dimensional non-degenerate quadratic surface, with Witt index $2$. Assume that $|\underline{E} \cap j| \leq |E|^{\alpha}$ for every affine totally isotropic subspace of dimension $2$.  Then,
$$\Lambda(E) \lesssim_{\alpha}  |E|^{19/7 +2\alpha/7 }.$$
In other words, $\Psi(\alpha)= 19/7 +2\alpha/7 $ is an energy exponent for $\mathcal{S}$.
\end{lemma}
\begin{proof}We start by decomposing $\underline{E}=\underline{E_c} \cup \underline{E_u}$ using Lemma \ref{lem:vectdecomp}, with $c=2$ and $\rho = (4\alpha+3)/7$. By Lemma \ref{lem:energyInc1} we have that
$$\Lambda(E_{u}) \lesssim |E|^{19/7 + 2\alpha/7} + |E|^{2+\alpha} \lesssim |E|^{19/7 + 2\alpha/7}.$$
By Lemma \ref{lem:classifysubsurface} we may decompose $E_c = \Omega^{[2,0,1]} \cup \Omega^{[2,0,0]} \cup \Omega^{[1,1,0]}$. Since $\Lambda(\Omega^{[2,0,0]}_i) \lesssim |\Omega^{[2,0,0]}_i|^{5/2}$ by Lemma \ref{lem:3denergy}, Lemma \ref{lem:energyIn} implies that
$$\Lambda(\Omega^{[2,0,0]}) \lesssim |I|^{3/2}|E|^{5/2}\lesssim |E|^{\frac{47}{14}-\frac{12\alpha }{14}}.$$
Next, we have that $|\Omega^{[2,0,1]}_i| \geq |E|^\rho$ and $|\Omega^{[2,0,1]}_i \cap j| \leq |E|^{\alpha}$. This implies that $|\Omega^{[2,0,1]}_i \cap j| \leq |\Omega^{[2,0,1]}_i|^{\frac{\alpha}{\rho}} =|\Omega^{[2,0,1]}_i|^{\frac{7\alpha}{4\alpha+3} } $ for every totally isotropic affine ($1$ dimensional) subspace $j$. Lemma \ref{lem:strong3dw1} then gives us (note that $3/4 \leq \frac{7\alpha}{4\alpha+3}  \leq 1$, for $\alpha\geq 9/16$) that
$$ \Lambda(\Omega^{[2,0,1]}_i) \leq |\Omega^{[2,0,1]}_i|^{1+2\frac{7\alpha}{4\alpha+3}}.$$
By Lemma \ref{lem:energyIn}, this implies
$$\Lambda(\Omega^{[2,0,1]}) \lesssim |I|^{(1-\rho)(3-2\alpha/\rho)}|E|^{1+2\alpha/\rho}= |E|^{2\alpha/7 +19/7}.$$
Finally, by Lemma \ref{lem:1r2renergy} we have that $\Lambda(\Omega^{[1,1,0]}_i) \lesssim \log^{O(1)}(|E|) |\Omega^{[1,1,0]}_i|^{2+\alpha} $.  Applying Lemma \ref{lem:energyIn} again, this implies
$$ \Lambda(\Omega^{[1,1,0]}) \lesssim \log^{O(1)}(|E|) |E|^{(1-\rho)(2-\alpha)} |E|^{2+\alpha} = \log^{O(1)}(|E|)  |E|^{\frac{4\alpha^2-5\alpha+22}{7}}.$$
Combining these estimates (using (\ref{eq:quasitriangleEnergy})) we have that
$$\Lambda(E) \lesssim |E|^{47/14-12\alpha/14 } + |E|^{19/7 +2\alpha/7 }+  \log^{O(1)}(|E|)  |E|^{\frac{4\alpha^2-5\alpha+22}{7}} \lesssim  |E|^{19/7 +2\alpha/7 }. $$
This completes the proof.
\end{proof}
We now give an abstract version of the argument employed above.

\begin{lemma}\label{lem:dimInduct}Let $\Psi$ be a continuous, increasing function $\Psi:[0,1]\rightarrow \mathbb{R}$ such that $\Psi(0)=3-\delta$ and $\Psi(1)=3$ with $\Psi(\alpha)<3$ for $\alpha<1$. Furthermore, let $d>4$ and consider a non-degenerate $d$ dimensional quadratic surface $\mathcal{S}$. Assume that for every $b \leq d-3$ one has that  every surface of rank $b$ has energy exponent $\Psi$. It follows that $\mathcal{S}$ has energy exponent $\Psi'(\alpha)$ where $\Psi'(0)=3-\delta'$ and $\Psi'(1)=3$ with $\Psi'(\alpha)<3$ for $\alpha<1$.
\end{lemma}
\begin{proof}Let $E \subset \mathcal{S}$. We start by decomposing $\underline{E}=\underline{E_c} \cup \underline{E_u}$ using Lemma \ref{lem:vectdecomp}, with $c$ equal to the Witt index of $\mathcal{S}$ and $\alpha  \leq \rho < 1$ to be specified later.  By Lemma \ref{lem:energyInc1}
$$\Lambda(E_u) \lesssim |E|^{5/2 + \rho/2} + |E|^{2+\alpha}  \lesssim |E|^{5/2 + \rho/2}. $$
By Lemma \ref{lem:classifysubsurface} we may decompose $E_c = \bigcup_{t \in T} \Omega^{[t]}$ where $t \in T$ ranges over a finite number of surface types. The components $\Omega_i^{[t]}$ satisfy $|\Omega_i^{[t]}| \geq |E|^{\rho}$ and, by the hypothesis, $|\Omega_i^{[t]} \cap j| \leq |E|^{\alpha}$ for every totally isotropic subspace. This implies that $$|\Omega_i^{[t]} \cap j| \leq |\Omega_i^{[t]} |^{\frac{\alpha}{\rho}}.$$
Thus
$$\Lambda(\Omega_i^{[t]}) \lesssim |\Omega_i^{[t]}|^{\Psi(\frac{\alpha}{\rho})}.$$
Moreover, by Lemma \ref{lem:energyIn} we have that
$$\Lambda(\Omega^{[t]}) \lesssim |I|^{4-\Psi(\frac{\alpha}{\rho})} \times |E|^{\Psi(\frac{\alpha}{\rho})}.$$
Thus by (\ref{eq:quasitriangleEnergy}), we have
$$\Lambda(E) \lesssim |E|^{5/2 + \rho/2}+ |E|^{ (4-\Psi(\frac{\alpha}{\rho}))\times (1-\rho) + \Psi(\frac{\alpha}{\rho})} \leq |E|^{5/2 + \rho/2}+ |E|^{ 4\times (1-\rho) + \Psi(\frac{\alpha}{\rho})}.$$
Defining $\theta(\alpha)$ to be the value of $\rho$ that equalizes the two terms on the right and taking $\Psi'(\alpha) = \frac{5+\theta(\alpha)}{2}$ completes the proof.
\end{proof}
Applying this inductively we may prove the following proposition for arbitrary non-degenerate quadratic surfaces.
\begin{proposition}\label{prop:EnergyGen}Let $\mathcal{S}$ be a non-degenerate quadratic surface of dimension $d$. There exists a continuous increasing function $\Psi:[0,1]\rightarrow \mathbb{R}$ such that $\Psi(0)=3-\delta$,$\Psi(1)=3$ and $\Psi(\alpha)<3$ for $\alpha<1$ such that the following holds. If $E \subseteq \mathcal{S}$ such that $|E \cap j|\leq |E|^{\alpha}$ for every maximal totally isotropic affine subspace $j$, then
$$\Lambda(E) \leq |E|^{\Psi(\alpha)}.$$
Moreover, $\Psi$ and $\delta$ depend only on $d$ and the Witt index of $\mathcal{S}$.
\end{proposition}
\begin{proof}To prove the result for a non-degenerate $d$ dimensional surface, by Lemma \ref{lem:dimInduct}, it suffices to prove the result for all surfaces, including degenerate surfaces (but not fully degenerate), of lower dimension. Note that if $\Psi(\alpha)$ is as given in the hypothesis, then $\theta(\alpha):= 3 \alpha + \Psi(\alpha)(1-\alpha)$ also satisfies the same conditions. Thus, by Lemma \ref{lem:degEnergy} it suffices to prove the result for all non-degenerate surfaces of lower dimension. Since we have already proven the result in $5$ and lower dimensions, the full result follows by induction.
\end{proof}
\section{The Mockenhaupt-Tao machine in higher dimensions}\label{sec:HighMT}
We now adapt the Mockenhaupt-Tao method \cite{MT} to higher dimensions.

We start with some notation. We will write $x=(\underline{x},t) \in \mathbb{F}^d$ where $\underline{x} \in \mathbb{F}^{d-1}$ and $t \in \mathbb{F}$. For $h:\mathbb{F}^d \rightarrow \mathbb{C}$ with support $E \subseteq \mathbb{F}^d$ we define (for $z \in \mathbb{F}$) $h_{z}(\underline{x},t) = h(\underline{x},t)\delta(t-z)$, and $E_z \subseteq \mathbb{F}^{d-1}$ to be the support of $h_z(\underline{x},z)$ (for $z$ fixed). Thus $h(\underline{x},t) = \sum_{z \in \mathbb{F}} h_{z}(\underline{x},t)$.

\begin{lemma}\label{lem:mtHD}Let $\mathcal{P}$ denote the $d=2n+1$ dimensional paraboloid with surface measure $d\sigma$. Let $h$ be a function such that $|h| \sim 1$ on its support $E\subseteq \mathbb{F}^d$. Furthermore, assume that $|E| \lesssim |\mathbb{F}|^{\gamma}$ (with $\gamma \geq 1$) and let $Z \subseteq \mathbb{F}$ be the set of $z$ such that $|E_z| \neq 0$. Define $0 \leq s \leq 1$ by $|\mathbb{F}|^s = |Z|$. In addition, for $z \in Z$, assume that
$$||(h_zd\sigma)^{\vee}||_{L^p(\mathbb{F}^d)} \lesssim  |\mathbb{F}|^{\alpha}.$$
Then
$$||\hat{h}||_{L^{2}(\mathcal{P},d\sigma)} \lesssim |\mathbb{F}|^{\gamma/2p'+n/2  +\alpha/2 + s/2} + |\mathbb{F}|^{\gamma/2}.$$
\end{lemma}
\begin{proof}As in $3$ dimensions, the proof will use the relation
$$||h||_{L^{2}(\mathcal{P},d\sigma)} = |\left<h, h * (d \sigma)^{\vee} \right>|^{1/2}.$$
Recall that $(d\sigma)^{\vee} = \delta_{0} + K$ where $K$ is the Bochner-Riesz kernel associated to $\mathcal{P}$, and $K$ is explicitly given in Lemma \ref{lem:FTcompP}. We start by collecting a few related estimates.

By translation symmetry, we may assume that $z=0$. We use the formula $K(\underline{x},t) =  |\mathbb{F}|^{-n} S(t) e(\underline{x}\cdot  \underline{x}/ 4t)$ (for $t \neq 0$ and $0$ otherwise). Thus
\[||h_{z} *K ||_{L^{p}(\mathbb{F}^d,dx)} =\left( \sum_{t \in \mathbb{F}, t \neq 0}  \sum_{\underline{x} \in \mathbb{F}^{2n}} \left| |\mathbb{F}|^{-n} \sum_{\underline{y} \in \mathbb{F}^{2n}} h_{0}(\underline{y},0) e( (\underline{x}- \underline{y}) \cdot  (\underline{x}- \underline{y}) /t)    \right|^p  \right)^{1/p}, \]
where we have discarded the term $S(t)$ since it has modulus $1$.
Using the pseudo-conformal transformation (for $t \neq 0$) $t':=1/4t$ and $\underline{w}:= -\underline{x}/2t$  we have that
\[  (\underline{x}- \underline{y}) \cdot  (\underline{x}- \underline{y}) /4t = \underline{w} \cdot \underline{w}   + \underline{w} \cdot \underline{y} + t' \underline{y} \cdot \underline{y}. \]
This allows us to bound
\[|| h_{z} * K||_{L^{p}(\mathbb{F}^d,dx)} \lesssim |\mathbb{F}|^{n} \left( \sum_{(w,t') \in \mathbb{F}^{2n}\times\mathbb{F}}  |(h_z d\sigma)^{\vee}(w,t')|^{p}  \right)^{1/p}. \]
Using the assumption that $||(h_zd\sigma)^{\vee}||_{L^p(\mathbb{F}^d)} \lesssim |\mathbb{F}|^{\alpha}$ we have
\[ || h * K||_{L^{p}(\mathbb{F}^d,dx)} \leq  \sum_{z \in Z}|| h_{z} * K||_{L^{p}(\mathbb{F}^3,dx)} \lesssim |\mathbb{F}|^{n + \alpha + s}. \]
Also, $||h * \delta||_{L^p(\mathbb{F}^d, dx)} = ||h ||_{L^p(\mathbb{F}^d, dx)} \lesssim |\mathbb{F}|^{\gamma/p} $.
Now we bound
$$ ||h||_{L^{2}(\mathcal{P},d\sigma)} = |\left<h, h * (d \sigma)^{\vee} \right>|^{1/2} \lesssim |\left<h, h * K \right>|^{1/2} + |\left<h, h * \delta \right>|^{1/2}$$
$$ \lesssim |\left<h, h * K \right>|^{1/2}  + |\mathbb{F}|^{\gamma/2}.$$
Now we may bound the first term as
$$\lesssim |\left<h, h * K \right>|^{1/2} \leq ||h||_{L^{p'}(\mathbb{F}^d, dx)}^{1/2} || h * K||_{L^{p}(\mathbb{F}^d,dx)}^{1/2} \lesssim |\mathbb{F}|^{\gamma/2p'+n/2  + \alpha/2 + s/2}. $$
Thus
\[ ||h||_{L^{2}(\mathcal{P},d\sigma)}  \lesssim |\mathbb{F}|^{\gamma/2p'+n/2 + \alpha/2 + s/2}+ |\mathbb{F}|^{\gamma/2}. \]
\end{proof}

\begin{corollary}\label{cor:HDMT}Let $\mathcal{P}$ be the $d=2n+1$ dimensional paraboloid with energy exponent $\Psi$. Let $h: \mathbb{F}^d \rightarrow \mathbb{C}$ with  $h \sim 1$ on its support $E \subseteq \mathbb{F}^d$, where $|E| = |\mathbb{F}|^\gamma$ (with $\gamma \geq 1$). Moreover, assume that $|E_z|\neq 0$ for $z \in Z \subseteq \mathbb{F}$ with $|Z| \leq |\mathbb{F}|^{s}$. Furthermore, for $z \in Z$, assume that $|E_z| \sim |\mathbb{F}|^{\beta}$ and that $|E_z \cap j| \leq |E_z|^{\alpha}$ for every maximal totally isotropic subspace of $j \subset \mathbb{F}^{d-1}$ (with respect to the bilinear form $\cdot$). Then,
$$||\hat{h}||_{L^{2}(\mathcal{P},d\sigma)} \lesssim |\mathbb{F}|^{\frac{\gamma(3+\Psi(\alpha))}{8}+ \frac{ (4-\Psi(\alpha))}{8} - \frac{d}{8} + \frac{1}{4} }  + |\mathbb{F}|^{\gamma/2}.$$
\end{corollary}
\begin{proof}We have that
$$||(h_zd\sigma)^{\vee}||_{L^4(\mathbb{F}^d)} \lesssim \frac{|\mathbb{F}|^{\frac{d}{4}}}{|\mathcal{P}|} \left(\Lambda(E_z)\right)^{1/4}\lesssim |\mathbb{F}|^{1- \frac{3d}{4}} |\mathbb{F}|^{\frac{\beta \Psi(\alpha)}{4}}$$
$$= |\mathbb{F}|^{1 - \frac{3d}{4} + \frac{\beta \Psi(\alpha)}{4}}.$$
Combining this with Lemma \ref{lem:mtHD} gives us that
$$||\hat{h}||_{L^{2}(\mathcal{P},d\sigma)} \lesssim |\mathbb{F}|^{\frac{3\gamma}{8}+\frac{d-1}{4} + \frac{s}{2} + \frac{1}{2} - \frac{3d}{8}+ \frac{\beta \Psi(\alpha)}{8}}  + |\mathbb{F}|^{\gamma/2}.$$
Reorganizing the exponent, using the relation $\gamma = s+\beta$, we have that this is
$$\lesssim |\mathbb{F}|^{\frac{7\gamma}{8}- \frac{(\gamma-s) (4-\Psi(\alpha))}{8} - \frac{d}{8} + \frac{1}{4} }  + |\mathbb{F}|^{\gamma/2}.$$
Since this is maximized when $s=1$, we may further estimate this as
$$\lesssim |\mathbb{F}|^{\frac{7\gamma}{8}- \frac{(\gamma-1) (4-\Psi(\alpha))}{8} - \frac{d}{8} + \frac{1}{4} }
\lesssim  |\mathbb{F}|^{\frac{\gamma(3+\Psi(\alpha))}{8}+ \frac{ (4-\Psi(\alpha))}{8} - \frac{d}{8} + \frac{1}{4} }  + |\mathbb{F}|^{\gamma/2}. $$
\end{proof}
\section{Mixed norm restriction estimates based on Kakeya maximal operator estimates}\label{sec:MNKak}
In this section we will restrict attention to the $d=2n+1$ dimensional paraboloid $\mathcal{P}:=\{(x,x\cdot x) : x \in \mathbb{F}^{2n}\}$ over fields in which the associated quadratic form has Witt index $n$ (that is fields in which $-1$ is a square).

Let $W$ be a totally isotropic subspace of $\mathbb{F}^{2n}$ (with respect to the standard dot product $\cdot$) and $V$ the complementary totally isotropic subspace given by Lemma \ref{lem:isoCompl}. Note that $W^{\perp}=W$, $V^{\perp}=V$ and $\text{dim}(V)=\text{dim}(W)=n$. Let $x, \xi \in \mathbb{F}^{2n}$, and write $x = x_W+x_V$, $\xi=\xi_W+\xi_V$ where $x_W, \xi_W \in W$ and $x_V,\xi_V \in V$. Then
$$\xi \cdot x= (\xi_W+\xi_V)\cdot(x_W + x_V) = \xi_W \cdot x_V + \xi_V \cdot x_W $$
and
$$\xi \cdot \xi = (\xi_W+\xi_V)\cdot (\xi_W+\xi_V) = 2 \xi_W \cdot \xi_V.$$
Identifying $\mathcal{P}$ with $\mathbb{F}^{2n}$ in the obvious way, recall that the extension operator $(fd \sigma)^{\vee}(x,t): \mathbb{F}^{2n} \times \mathbb{F} \rightarrow \mathbb{C}$ is defined by
$$
(fd \sigma)^{\vee}(x,t) := \frac{1}{|\mathbb{F}|^{2n}} \sum_{\xi \in \mathbb{F}^{2n}} f(\xi) e(\xi \cdot x + \xi \cdot \xi t).
$$
Given two complementary $n$ dimensional subspaces, $V$ and $W$, of $\mathbb{F}^{2n}$ (thus, $\mathbb{F}^{2n}= V \oplus W$), we may define a norm on functions $F: \mathbb{F}^{2n} \times \mathbb{F} \rightarrow \mathbb{C}$ by
$$|| F ||_{L^{q}_{V,t} L^p_{W}} =  \left( \sum_{v \in V, t \in \mathbb{F}} \left( \sum_{w \in W} \left| F(w+ v,t) \right|^p \right)^{q/p}\right)^{1/q}.$$
Similarly, given $f: \mathcal{P} \rightarrow \mathbb{C}$, after identifying $\mathcal{P}$ with $\mathbb{F}^{2n}$ as above, we may consider $f : \mathbb{F}^{2n} \rightarrow \mathbb{C}$. We then define the norm associated to the normalized counting measure on $\mathbb{F}^{2n}$:
$$|| f ||_{L^{q}_{V} L^p_{W}(\mathcal{P},d\sigma)} =  \left( |\mathbb{F}|^{-n} \sum_{v \in V} \left( |\mathbb{F}|^{-n} \sum_{w \in W} \left| f(w+v) \right|^p \right)^{q/p}\right)^{1/q}.$$

Using the algebraic identities above, for $x_1 \in V$ and $x_2 \in W$, we may parameterize this operator as $(fd \sigma)^{\vee}(x_1 + x_2, t): \mathbb{F}^{2n} \times \mathbb{F} \rightarrow \mathbb{C}$ as
\begin{equation}\label{eq:reparam}
(fd \sigma)^{\vee}(x_1+x_2,t) = \frac{1}{|\mathbb{F}|^{2n}} \sum_{\substack{\xi_1 \in W \\ \xi_2 \in V}} f(\xi_1+ \xi_2) e(\xi_1 \cdot x_1 + \xi_2 \cdot x_2 + 2t \xi_1 \cdot \xi_2 ).
\end{equation}
We now show that (in this setting) the Kakeya maximal operator estimate (\ref{eq:dualKakeya}) is equivalent to a mixed norm restriction estimate for $\mathcal{P}$. This is somewhat analogous to the square function estimates (see \cite{Brest}) in the Euclidean setting. We will only prove that the Kakeya estimate (\ref{eq:dualKakeya}) implies the mixed norm restriction estimate (which is all we will use), however the reverse implication can be easily seen from the proof.
\begin{proposition}Let $2n$ be an even integer, and $\cdot$ a $n\times n$ non-degenerate bilinear form with Witt index $n$. Let $W$ and $V$ be a $n$-dimensional totally isotropic subspace (with respect to $\cdot$) and $(fd \sigma)^{\vee}(x_1,x_2,t)$ as defined in (\ref{eq:reparam}).  For a function $f(\xi_1 + \xi_2) :  W \times V \rightarrow \mathbb{C}$ we have that
\begin{equation}\label{eq:ExpKak}
||(f d\sigma)^{\vee} ||_{L^{\frac{2d+2}{d-1}}_{V,t} L^2_{W}} \lesssim ||f||_{ L^{\frac{2d+2}{d-1} }_{V} L^2_{W}}.
\end{equation}
\end{proposition}
\begin{proof}
We work with the parameterization given in (\ref{eq:reparam}).  We partition $\mathbb{F}^{2n} = W \oplus V$ into $\mathbb{F}^{n}$ `caps' of the form $\{(\xi_1+\alpha, 2 \xi_1 \cdot \alpha) : \xi_1 \in W\}$ indexed by $\alpha \in V$. Similarly, we may define the restriction of $f$ to a cap associated to $\alpha$ by
$$f_{\alpha}(\xi_1 + \xi_2) = f(\xi_1 + \xi_2) \delta(\xi_2-\alpha)$$
where $\delta$ denotes the Dirac delta function on $\mathbb{F}^{2n}$. We define $h_{\alpha}(\eta)$ (for $\eta \in V$ and $\alpha \in V$) to be the $n$-dimensional Fourier inverse transform associated to $f_{\alpha}(\xi_1+\alpha)$ (as given by Lemma \ref{lem:subFour}). In other words, $h_{\alpha}(\eta)$ is defined by the relation
$$ f_{\alpha}(\xi_1+\alpha) = \sum_{\eta \in V} h_{\alpha}(\eta)e(\xi_1 \cdot \eta).$$
We now define
\begin{equation}\label{eq:faOrth}
\mathcal{A}(\alpha) := |\mathbb{F}|^{-n} \sum_{\xi_1 \in W}|f_{\alpha}(\xi_1+ \alpha)|^2 = \sum_{\eta \in V} |h_{\alpha}(\eta)|^2.
\end{equation}
The last equality holds by orthogonality of characters. We have that
\[ (f_{\alpha} d\sigma)^{\vee}(x_1+x_2,t) =   \frac{1}{|\mathbb{F}|^{2n}} \sum_{\eta \in V } h_{\alpha}(\eta) \sum_{\xi_1 \in W } e(\xi_1 \cdot \eta) e(\xi_1 \cdot x_1 + \alpha \cdot x_2 +  t\xi_1 \cdot \alpha )   \]
\[= \frac{e(\alpha \cdot x_2) }{|\mathbb{F}|^{2n}}\sum_{\eta \in V } h_{\alpha}(\eta) \sum_{\xi_1 \in W } e(\xi_1 \cdot (\eta  + x_1 +  2t \alpha ))    \]
\[= \frac{e(\alpha \cdot x_2) }{|\mathbb{F}|^{n}}\sum_{\eta \in V } h_{\alpha}(\eta) \delta(\eta  + x_1 +  2 t \alpha )\]
\[ = \frac{e(\alpha \cdot x_2) }{|\mathbb{F}|^{n}} \sum_{\eta \in V} h_{\alpha}(\eta) 1_{\ell_{(-\eta,-2 \alpha)} } (x_1,t) .\]
Here $\delta$ can be taken to be the Dirac delta function on $V$ and $\ell_{(-\eta,-2\alpha)}$ is the line in $V \times \mathbb{F}$ with direction $2\alpha$ and intercept $-\eta$. We now wish to estimate:
\[\left|\left| \sum_{\alpha \in \mathbb{F}^{n} } (f_{\alpha} d \sigma)^{\vee}(x_1+x_2, t)    \right|\right|_{L^{\frac{2d+2}{d-1}}_{V,t} L^2_{W}}   =  \left|\left| \sum_{\alpha \in \mathbb{F}^n} \frac{e(\alpha \cdot x_2) }{|\mathbb{F}|^{n}} \sum_{\eta \in V } h_{\alpha}(\eta) 1_{\ell_{(-\eta,-\alpha)} } (x_1,t)  \right|\right|_{L^{(2d+2)/(d-1)}_{V,t}L^{2}_{W}} .\]
From orthogonality we have
\[ \left|\left| \sum_{\alpha \in V} \frac{e(\alpha \cdot x_2) }{|\mathbb{F}|^{n}} \sum_{\eta \in V } h_{\alpha}(\eta) 1_{\ell_{(-\eta,-2\alpha)} } (x_1,t) \right|\right|_{L^{2}_{W}} = \left(\frac{1}{|\mathbb{F}|^{n}}\sum_{\alpha \in V} \sum_{\eta \in V }  h_{\alpha}^2(\eta) 1_{\ell(-\eta,-2\alpha)}(x_1,t) \right)^{1/2}. \]
Thus
\[\left|\left| \sum_{\alpha \in V } (f_{\alpha} d \sigma)^{\vee}    \right|\right|_{L^{\frac{2d+2}{d-1}}_{V,t} L^2_{W}}  =  \left|\left| \frac{1}{|\mathbb{F}|^{n}}\sum_{\alpha \in V} \sum_{\eta \in V}  h_{\alpha}^2(\eta) 1_{\ell(-\eta,-2\alpha)}(x_1,t)   \right|\right|_{L^{(d+1)/(d-1)}_{V,t}}^{1/2}\]
Note that $(d+1)/(d-1) = (n+1)/n $. We now perform a randomization trick to reduce to the case when there is only one line in every direction. This will then allow us to apply the Kakeya estimate. Consider the random function
$$ \frac{1}{|\mathbb{F}|^{n}}\sum_{\alpha \in V} \mathcal{A}(\alpha) 1_{\ell(-\eta_{\omega},-2\alpha)}(x_1,t) $$
where $\mathcal{A}(\alpha)$ is defined by (\ref{eq:faOrth}) and for each $\alpha$ the translation $\eta_\omega$ is randomly selected with probability $\frac{|h_{\alpha}(\eta)|^2}{\mathcal{A}(\alpha)}$. We then have the pointwise equality
$$ \mathbb{E}_{\omega} \left[ \frac{1}{|\mathbb{F}|^{n}}\sum_{\alpha \in W} \mathcal{A}(\alpha) 1_{\ell(-\eta_{\omega},-2\alpha)}(x_1,t)\right]=
\frac{1}{|\mathbb{F}|^{n}}\sum_{\alpha \in W} \sum_{\eta \in V}  h_{\alpha}^2(\eta) 1_{\ell(-\eta,2-\alpha)}(x_1,t).$$
Thus, we may bound the expression above by
\[ \left|\left| \mathbb{E}_{\omega} \left[ \frac{1}{|\mathbb{F}|^{n}}\sum_{\alpha \in W} \mathcal{A}(\alpha) 1_{\ell(-\eta_{\omega},-2\alpha)}(x_1,t) \right]  \right|\right|_{L^{(n+1)/n}_{V,t}}^{1/2}\]
\[\leq \left( \mathbb{E}_{\omega} \left[ \left|\left| \frac{1}{|\mathbb{F}|^{n}}\sum_{\alpha \in W} \mathcal{A}(\alpha) 1_{\ell(-\eta_{\omega},-2\alpha)}(x_1,t) \right]  \right|\right|_{L^{(n+1)/n}_{V,t}} d\omega \right)^{1/2}.  \]
Applying the Kakeya estimate (\ref{eq:KakeyaMaxIn}) (in dual form (\ref{eq:dualKakeya})) to the inner quantity, we may estimate this as
\[\lesssim \left( |\mathbb{F}|^{-n} \sum_{\alpha \in V } |\mathcal{A}(\alpha)|^{(n+1)/n } \right)^{n/(2n+2)} =
\left( |\mathbb{F}|^{-n} \sum_{\alpha  \in V } \left| |\mathbb{F}|^{-n} \sum_{\xi_1 \in W }|f_{\alpha}(\xi_1+ \alpha)|^2   \right|^{(n+1)/n } \right)^{n/(2n+2)}.  \]
\[= ||f||_{ L^{\frac{2d+2}{d-1} }_{V} L^2_{W}}.\]
This completes the proof.
\end{proof}

We now deduce some consequences of this result.  In what follows we always assume that $\mathbb{F}$ is a field in which $-1$ is a square. Moreover, we assume that $V$ and $W$ are complementary totally isotropic subspaces of dimension $n$.  We will use the dual form of (\ref{eq:ExpKak}).  Let $f: \mathbb{F}^d \rightarrow \mathbb{C}$ with $d$ dimensional Fourier transform $\hat{F}$. Parameterizing $f(x_1 + x_2,t)$ with $x_1 \in V$, $x_2 \in W$,and $t\in \mathbb{F}$, and the restriction $\hat{f}$ to $\mathcal{P}$ as $\hat{f}\left(\xi_1+\xi_2, (\xi_1+\xi_2)\cdot (\xi_1+\xi_2)\right) = \hat{f}\left(\xi_1+\xi_2, 2 \xi_1\cdot \xi_2 \right)$ for $\xi_1 \in W$ and $\xi_2 \in V$. We will prove the inequality
\[ \left( |\mathbb{F}|^{-n} \sum_{\xi_2  \in V} \left(|\mathbb{F}|^{-n} \sum_{ \xi_1  \in W } |\hat{f} (\xi_1+\xi_2, 2 \xi_1 \cdot \xi_2 )|^{2} \right)^{\frac{d+1}{d+3}} \right)^{\frac{d+3}{2d+2} }\]
\begin{equation}\label{eq:dualKakeyaExp}
\lesssim  \left( \sum_{x_2  \in V ,t  \in \mathbb{F} } \left( \sum_{x_1  \in W} |f(x_1 +x_2,t)|^2  \right)^{\frac{d+1}{d+3} } \right)^{\frac{d+3}{2d+2}}
\end{equation}
which we notate as:
$$ || \hat{f} ||_{L^\frac{2d+2}{d+3}_V L^{2}_W (\mathcal{P},d\sigma)} \lesssim ||f||_{L^{\frac{2d+2}{d+3}}_{V,t}L^{2}_{W}}. $$
For completeness we include the  derivation of this inequality from (\ref{eq:ExpKak}).
\begin{lemma}\label{lem:MixedResForm}In the notation above, we have the inequality
$$ || \hat{f} ||_{L^\frac{2d+2}{d+3}_V L^{2}_W (\mathcal{P},d\sigma)} \lesssim ||f||_{L^{\frac{2d+2}{d+3}}_{V,t}L^{2}_{W}}. $$
This implies the weaker inequality:
$$ || \hat{f} ||_{L^\frac{2d+2}{d+3}(\mathcal{P},d\sigma)} \lesssim ||f||_{L^{\frac{2d+2}{d+3}}_{V,t}L^{2}_{W}}. $$
\end{lemma}
\begin{proof}Note that the extension and restriction operators satisfy the duality relation
$$\left<(gd\sigma)^{\vee},f\right>_{(\mathbb{F}^d,dx)} = |\mathcal{P}|^{-1} \sum_{\xi \in \mathcal{P}} g(\xi) \overline{\widehat{f}(\xi)}= \left< g, \hat{f} \right>_{(\mathcal{P},d\sigma)}.$$
Using the fact that a mixed norm may be written as a supremum of the inner product with elements of the dual ball (see Theorem 1 of \cite{BP}), we now have that
$$|| \hat{f} ||_{L^\frac{2d+2}{d+3}_V L^{2}_W (\mathcal{P},d\sigma)}  = \sup_{\substack { g \\ ||g||_{L^{\frac{2d+2}{d-1} }_{V} L^{2}_{W}(\mathcal{P},d\sigma)} =1 } } <g, \hat{f} >_{(\mathcal{P},d\sigma)} = \sup_{\substack { g \\ ||g||_{L^{\frac{2d+2}{d-1} }_{V} L^{2}_{W}(\mathcal{P},d\sigma)} =1 }} \left<(gd\sigma)^{\vee}, f \right>_{(\mathbb{F}^d,dx)}.$$
Applying the mixed norm H\"older inequality (see (1) of \cite{BP}) and (\ref{eq:ExpKak}) we may estimate this
$$\leq \sup_{\substack { g \\ ||g||_{L^{\frac{2d+2}{d-1} }_{V} L^{2}_{W}} =1 }}   ||(g d\sigma)^{\vee} ||_{L^{\frac{2d+2}{d-1} }_{V,t} L^{2}_{W}} ||f||_{L^{\frac{2d+2}{d+3} }_{V,t} L^{2}_{W}}  \leq C||f||_{L^{\frac{2d+2}{d+3} }_{V,t} L^{2}_{W}}.$$
\end{proof}
Let us also recall the following property of cosets/affine subspaces of a vector space.
\begin{lemma}\label{lem:elcosets}Let $V$ be a finite dimensional vector space. Let $A$ and $B$ denote subspaces of $V$. Let $C$ denote the subspace $A\cap B =C$. Then, either the intersection of cosets of $A$ and $B$, say $(x+A) \cap (y+B)$, is empty or
$$|(x+A)\cap(y+B)|=|C|.$$
\end{lemma}
\begin{proof}Assume that $(x+A)\cap(y+B) \neq \emptyset$, and consider $z \in (x+A)\cap(y+B)$. Now $(x+A)=(z+A)$ and $(y+B)=(z+B)$, so $(x+A)\cap(y+B)=(z+A)\cap(z+B)$. From this it follows that $(x+A)\cap(y+B) = z+(A\cap B)$. This implies that $|(x+A)\cap(y+B)|=|C|$.
\end{proof}
\begin{corollary}\label{cor:KakeyaRegSet}Let $F : \mathbb{F}^d \rightarrow \mathbb{C}$ such that $F \sim 1$ on its support $E$, with $|E| \leq |\mathbb{F}|^{\gamma}$. For $z \in \mathbb{F}$, denote the restriction of $F$ to the hyperplane $\{(x,z):x \in \mathbb{F}^{d-1}\}$ as $F_z$ and let $E_{z} \subset \mathbb{F}^{d-1}$ denote the support of $F_z$. Assume that $E_z$ can be written as a disjoint union $E_z = \cup_{j\in A(z)} U_j$, where each $U_j$ is the subset of a distinct maximal totally isotropic affine subspace $\Omega_j$, with $|A(z)| \leq |\mathbb{F}|^{e}$.
Then
$$||\hat{F}||_{L^{\frac{2d+2}{d+3}}(\mathcal{P},d\sigma)} \lesssim |\mathbb{F}|^{\frac{\gamma}{2} + \frac{e+1}{d+1} + \frac{d-3}{2d+2} }.$$
\end{corollary}
\begin{proof}From Lemma \ref{lem:MixedResForm} (noting \eqref{eq:dualKakeyaExp}) we have
\begin{equation}\label{eq:mixedExp}||\hat{F}||_{L^{\frac{2d+2}{d+3}}(\mathcal{P},d\sigma)} \lesssim \left( \sum_{x_2  \in V ,t  \in \mathbb{F} } \left( \sum_{x_1  \in W} |f(x_1 +x_2,t)|^2  \right)^{\frac{d+1}{d+3} } \right)^{\frac{d+3}{2d+2}}.
\end{equation}
It will thus suffice to estimate the right side of this inequality. Recall that if $W$ is a maximal totally isotropic subspaces (in a $2n$ dimensional quadratic space with Witt index $n$) then the complementary subspace $V$ is also a maximal totally isotropic subspace. Thus, after splitting $F$ into two functions with disjoint supports, we may assume that no pair of the sets $U_{j}$ arise from complementary subspaces. We now let $W$ be a maximal totally isotropic subspace associated to some $U_{j}$.

Let $\{W_i \}_{i \in \mathcal{C}}$ be an enumeration of the cosets of $W$. Given a (non-complementary) subspace $V$, it follows that the set of $i$'s such that $W_{i} \cap V \neq \emptyset$ has size $\frac{|\mathbb{F}|^n}{|W\cap V|}\leq |\mathbb{F}|^{n-1}$, using Lemma \ref{lem:elcosets} and the fact that the spaces are not complementary. We may estimate the right side of (\ref{eq:mixedExp}) as
$$ \left( \sum_{z \in Z} \sum_{i \in \mathcal{C}}  \left( \sum_{j \in A(z)} |W_i \cap U_{j}| \right)^{\frac{d+1}{d+3}} \right)^{\frac{d+3}{2d+2}}
\leq \left( \sum_{z \in Z} \sum_{i \in \mathcal{C}} \sum_{j \in A(z)}  |W_i \cap U_{j}|^{\frac{d+1}{d+3}} \right)^{\frac{d+3}{2d+2}}.
$$
In the first inequality we have used the assumption that the sets $\{U_{j}\}$ are disjoint, and in the second inequality we have used the reverse form of Minkowski's inequality, since $\frac{d+1}{d+3} <1$. This last term can be reorganized as
$$ \left( \sum_{z \in Z}  \sum_{j \in A(z)} \sum_{\substack{ i \in \mathcal{C} \\ W_i \cap \Omega_j \neq \emptyset }  } |W_i \cap U_{j}|^{\frac{d+1}{d+3}} \right)^{\frac{d+3}{2d+2}}.$$
Since, for $\Omega_j$ fixed, we have $|\{ i \in \mathcal{C}  : W_i \cap \Omega_j \}| \leq \frac{|\mathbb{F}|^{n}}{|W\cap \Omega_j|}$, H\"{o}lder's inequality gives us that
$$\sum_{\substack{ i \in \mathcal{C} \\ W_i \cap \Omega_j \neq \emptyset }  } |W_i \cap U_{j}|^{\frac{d+1}{d+3}} \leq \left(\sum_{\substack{ i \in \mathcal{C} \\ W_i \cap \Omega_j \neq \emptyset }  }|W_i \cap U_{j}| \right)^{\frac{d+1}{d+3}} \frac{|\mathbb{F}|^{\frac{2n}{d+3}} }{|W\cap \Omega_j|^{\frac{2}{d+3}}} $$
$$\lesssim |U_{j}|^{\frac{d+1}{d+3}}\frac{|\mathbb{F}|^{ \frac{2n}{d+3}} }{|W\cap \Omega_j|^{\frac{2}{d+3}}} \leq |U_{j}|^{\frac{d+1}{d+3}} |\mathbb{F}|^{(n-1)\frac{2}{d+3}}.$$
Here we have used that $|W \cap \Omega_j| \geq |\mathbb{F}|$. In addition, using that
$$\sum_{z \in Z}\sum_{j \in A(z)}|U_j|^{\frac{d+1}{d+3}} \leq \left( \sum_{z \in Z} \sum_{j \in A(z)}|U_j|\right)^{\frac{d+1}{d+3}} |\mathbb{F}|^{\frac{2 e +2}{d+3} } \leq |\mathbb{F}|^{\gamma \frac{d+1}{d+3} + \frac{2 e+2}{d+3}}, $$
we have
$$||\hat{F}||_{L^{\frac{2d+2}{d+3}}(\mathcal{P},d\sigma)} = \left(|\mathbb{F}|^{\gamma \frac{d+1}{d+3} + \frac{2 e+2}{d+3} +\frac{d-3}{d+3} } \right)^{\frac{d+3}{2d+2}}
\lesssim |\mathbb{F}|^{\frac{\gamma}{2} + \frac{e+1}{d+1} + \frac{d-3}{2d+2} }.$$
This completes the proof.
\end{proof}
\section{Proof of Theorem \ref{thm:Main}}\label{sec:MainProof}
We are now ready to assemble the estimates proven in the previous sections to obtain an improvement to the Stein-Tomas inequality. We will prove the estimate in dual form.
\newtheorem*{thm:Main}{Theorem \ref{thm:Main}}
\begin{thm:Main}\normalfont{(dual form)} Let $d=2n+1 \geq 3$ be an odd integer and $\mathcal{P}$ the $d$ dimensional paraboloid over a field $\mathbb{F}$ in which $-1$ is a square. We have the inequality
$$||\hat{F}||_{L^{\frac{2d+2}{d+3}}(\mathcal{P},d\sigma)} \lesssim ||F||_{L^{\frac{2d+2}{d+3} + \delta_{d}}  }$$
for some $\delta_d>0$.
\end{thm:Main}
\begin{proof}We restrict attention to $d\geq 5$ since the $d=3$ case has already been proven. By dyadic pigeonholing and $\epsilon$ removal, it suffices to assume that $F: \mathbb{F}^d \rightarrow \mathbb{C}$ such that $F \sim 1$ on its support $E \subseteq \mathbb{F}^d$, where $|E| \sim |\mathbb{F}|^\gamma$. Moreover, we may assume that if $E_z \subseteq \mathbb{F}^{d-1}$ denotes the slices of $E$ as above, then $\gamma=s+\beta$ where $|Z|=|\mathbb{F}|^s$  and $|E_z| \sim |\mathbb{F}|^\beta$ (for all $z$ such that $|E_z| \neq 0$).
First we dispense with the case $\gamma \leq \frac{d+1}{2} - \frac{1}{d^2}$. The term $\frac{1}{d^2}$ has not been chosen optimally (since we are not attempting to produce an explicit estimate). If $\gamma \leq \frac{d+1}{2} - \frac{1}{d^2}$, Lemma \ref{lem:stdecay} gives
$$||\hat{F}||_{L^{2}(\mathcal{P},d\sigma)} \lesssim||F||_{L^{\frac{4\gamma}{4\gamma-d+1}}} +||F||_{L^2} = ||F||_{L^{\frac{2d^2+4d+4}{d^2+4d+4}}} +||F||_{L^2} = ||F||_{L^{\frac{2d+2}{d+3}+\frac{4}{(d+2)^2(d+3) }}}+||F||_{L^2}.$$
We assume $\gamma \geq \frac{2d+2}{d+3} - \frac{1}{d^2}$ throughout the rest of the argument. Using the argument of Lemma \ref{lem:vectdecomp} (with respect to maximal totally isotropic affine subspaces instead of arbitrary affine subspaces) on each slice $E_z$, for any $\alpha <1$, we may partition $E = E^{(c)} \cup E^{(u)}$ such that:
\begin{enumerate}
\item For $z \in Z$, the set $E^{(c)}_{z}$ can be written as a disjoint union of $E^{(c)}_{z} =\bigcup_{i \in I} \Omega_i$ where $|\Omega_i| \geq |E^{(c)}_{z}|^{\alpha}$ and each $\Omega_i$ is contained in a maximal totally isotropic affine subspace of dimension $n$. Moreover, $|I| \lesssim |E^{(c)}_{z}|^{1-\alpha}$,
 \item Given any ($n$ dimensional) maximal totally isotropic affine subspace $j \subset \mathbb{F}^{d-1}$, we have $|E^{(u)}_z \cap j| \lesssim |E^{(u)}_z|^{\alpha}$.
\end{enumerate}
We will apply the above decomposition with $\alpha = 1 - 1/d^2$. Define $F_{u}$ (respectively $F_{c}$) to be the restriction of $F$ to $E^{(c)}$ (respectively $E^{(u)}$) defined above. Let $\Psi$ be the energy exponent given by proposition \ref{prop:EnergyGen}, thus $\Psi(\alpha)=3 - \eta$ for some $\eta >0$. These, and all other constants in this section, are allowed to depend on the dimension $d$. From Corollary \ref{cor:HDMT} we have
$$ ||\hat{F_{u}}||_{L^{2}(\mathcal{P},d\sigma)} \lesssim |\mathbb{F}|^{\frac{7\gamma}{8}- \frac{\beta (4-\Psi(\alpha))}{8} - \frac{d}{8} + \frac{1}{4} }$$
This estimate degrades as $\beta$ decreases, so we may assume that $\beta = \gamma-1$, which gives
$$ ||\hat{F_{u}}||_{L^{2}(\mathcal{P},d\sigma)} \lesssim |\mathbb{F}|^{\frac{7\gamma}{8}- \frac{(\gamma-1) (4-\Psi(\alpha))}{8} - \frac{d}{8} + \frac{1}{4} }\lesssim |\mathbb{F}|^{\frac{7\gamma}{8}- \frac{(\gamma-1) (1+\eta )}{8} - \frac{d}{8} + \frac{1}{4} } $$
$$\lesssim |\mathbb{F}|^{\frac{6\gamma -d +3}{8} - \frac{\eta(\gamma-1)}{8} }.  $$
Here we have used that $\gamma \geq \frac{2d+2}{d+3} - \frac{1}{d^2}$, for some $\eta' >0$ this is $\lesssim |\mathbb{F}|^{\frac{6\gamma -d +3}{8} - \eta' }.$  This implies, for some $\eta'' >0$, that
$$||\hat{F_{u}}||_{L^{2}(\mathcal{P},d\sigma)}  \lesssim ||F||_{L^{\frac{8\gamma}{6\gamma - d+3} + \eta'' }}. $$
On the other hand, Lemma \ref{lem:RestL2} gives
$$ ||\hat{F_{u}}||_{L^{2}(\mathcal{P},d\sigma)} \lesssim  ||F||_{\frac{2\gamma}{1+\gamma}}.$$
Since $\frac{8\gamma}{6\gamma - d+3} + \eta''$ is decreasing in $\gamma$ and $\frac{2\gamma}{1+\gamma}$ is increasing in $\gamma$, and their intercept is strictly greater than $(d+1)/2$,  we may conclude that
$$||\hat{F_{u}}||_{L^{2}(\mathcal{P},d\sigma)}  \lesssim ||F_{u}||_{L^{\frac{2d+2}{d+3} + \eta'''}}.$$
Next we consider $F_{c}$. Corollary \ref{cor:KakeyaRegSet}, with $e=(1-\alpha)(\gamma-s)$, gives
$$||\hat{F_{c}}||_{L^{\frac{2d+2}{d+3}}(\mathcal{P},d\sigma)} \lesssim  |\mathbb{F}|^{\frac{\gamma}{2} + \frac{e+1}{d+1} + \frac{d-3}{2d+2} } \lesssim |\mathbb{F}|^{\gamma\frac{d+1}{2d+2} + \frac{2 (1-\alpha)(\gamma-s)+d-1}{2d+2}}. $$
Clearly this estimate is least favorable when $s=0$. Thus we can bound this by
$$\lesssim |\mathbb{F}|^{\gamma\frac{d+1}{2d+2} + \frac{2 (1-\alpha)\gamma+d-1}{2d+2}} \lesssim  |\mathbb{F}|^{\gamma \frac{d+3 }{2d+2}  -\frac{2 \alpha \gamma -d +1}{2d+2} }.$$
We now claim that $2\alpha  \gamma -d  +1 > \delta' >0$ for all $\gamma \geq \frac{d+1}{2} - \frac{1}{d^2}$. Indeed, we have
$$2\alpha  \gamma -d  +1  \geq \left(1-\frac{1}{d^2}\right) \left(d+1 -\frac{2}{d^2} \right) -d +1$$
$$= 2 -\frac{d^3+3d^2-2}{d^4} \geq \delta' >0$$
as desired. Thus
$$||\hat{F_{c}}||_{L^{\frac{2d+2}{d+3}}(\mathcal{P},d\sigma)} \lesssim |\mathbb{F}|^{-\delta''} ||F||_{L^{\frac{2d+2}{d+3}}}.$$
H\"older's inequality then implies
$$||\hat{F_{c}}||_{L^{\frac{2d+2}{d+3}}(\mathcal{P},d\sigma)} \lesssim ||F||_{L^{\frac{2d+2}{d+3}+\delta'''}}.$$
This completes the proof.
\end{proof}
We now sketch an explicit form of the above argument in $5$ dimensions. By Lemma \ref{lem:stdecay} we may assume that, say, $\gamma >2.5$. Arguing as above, we have that
$$||\hat{F}||_{L^{\frac{3}{2}}(\mathcal{P},d\sigma)} \lesssim ||\hat{F_{c}}||_{L^{\frac{3}{2}}(\mathcal{P},d\sigma)} + ||\hat{F_{u}}||_{L^{2}(\mathcal{P},d\sigma)} \lesssim  |\mathbb{F}|^{\frac{\gamma(4-2\alpha)+2 }{6} } +  |\mathbb{F}|^{\frac{7\gamma}{8}- \frac{(\gamma -1)(4-\Psi(\alpha))}{8} - \frac{5}{8} + \frac{1}{4} } $$
$$\lesssim |\mathbb{F}|^{\frac{\gamma(4-2\alpha)+2 }{6} } +  |\mathbb{F}|^{\frac{3\gamma + \Psi(\alpha) (\gamma-1) +1}{8} } $$
with $\Psi(\alpha)= \max\left(\frac{19+2\alpha}{7},\frac{23}{8}\right)$ using Lemma \ref{lem:5dEnergy}. Optimizing the parameter $0\leq \alpha \leq 1$ then gives
$$||\hat{F}||_{L^{\frac{3}{2}}(\mathcal{P},d\sigma)} \lesssim |\mathbb{F}|^{ \frac{-22\gamma^2 + 7 \gamma +1}{3-31\gamma} } + |\mathbb{F}|^{\frac{47\gamma-15}{64}}.  $$
This implies that
$$||\hat{F}||_{L^{\frac{3}{2}}(\mathcal{P},d\sigma)} \lesssim ||F||_{L^{\frac{31 \gamma^2 -3 \gamma}{22\gamma^2 - 7\gamma -1}  }} +||F||_{L^{\frac{64\gamma}{47\gamma-15}}}.  $$
Optimizing this with Lemma \ref{lem:RestL2} in the form
$$||\hat{F}||_{L^{2}(\mathcal{P},d\sigma)} \lesssim ||F||_{L^{\frac{2\gamma}{1+\gamma}}(\mathbb{F}^5,dx)} $$
gives the following result.
\begin{theorem}For $p  < 47/31 = 1.51613\ldots$ we have the inequality
$$||\hat{F}||_{L^{\frac{3}{2}}(\mathcal{P},d\sigma)} \lesssim ||F||_{L^{p}(\mathbb{F}^5,dx)}.$$
\end{theorem}
\section{Further remarks}\label{sec:Remarks}
We conclude with some additional remarks:
\begin{enumerate}
 \item The structure theorem for energetic subsets of quadratic surfaces presented in Section \ref{sec:HighEnergy} is probably quite inefficient. Improving these estimates is likely the most efficient path to improving the exponent in the main result.
 \item We have focused exclusively on the case of finite fields. In \cite{WrightNotes} (see also \cite{WrightExp}) Wright has observed that replacing a finite field $\mathbb{F}$ with the ring of residues $\mod N$ is more analogous to the Euclidean case. Indeed, the divisors of $N$ introduce an analog of scales, which is not present in the finite field case. The presence of multiple scales seems crucial for the Kakeya phenomena. Indeed in finite fields, which lack multiple scales, Kakeya sets must have positive measure. In the function field setting, where there is an infinite number of scales, Dummit and Hablicsek \cite{DH} have recently proven that Kakeya sets with measure $0$ exist. It is an open problem to determine if Kakeya sets with measure $0$ exist in the p-adic setting (where there also is an infinite number of scales) however it seems the answer is likely `yes'. Note the $p$-adic setting is similar to working in the ring of residues $\mod N$ with $N=p^{k}$ asymptotically as $k \rightarrow \infty$ (although the measures are different). If there does exist $2$ dimensional $p$-adic Kakeya sets of measure $0$, likely one can then use the arguments presented here to disprove the endpoint $3$ dimensional restriction inequality for the paraboloid over the ring $R$ of residues $\mod p^k$:
 $$|| (fd\sigma)^{\vee}||_{L^{3}(R^{3},dx) } \lesssim ||f||_{L^{\infty}(\mathcal{P},d\sigma)}. $$
 This would be analogous to the Euclidean result of Beckner, Carbery, Semmes and Soria \cite{BCSS} based on C. Fefferman's disc multiplier counterexample \cite{FeffermanD}. \textbf{Update:} The author has recently learned that J. Wright has proven the existence of $p$-adic Kakeya sets of measure $0$.
 \item The arguments in Section \ref{sec:KaktoRes} indicate that certain restricted $k$-plane operators might also play a role in finite field restriction problems. We note that while sharp estimates are known for the finite field Kakeya maximal operator, the problem of determining the $L^p$ mapping properties for (unrestricted) $k$-plane transforms in finite fields is not fully resolved. See \cite{Bueti}, \cite{OberlinThesis} and Section 4.12 of \cite{EOT} for a discussion of these operators.
 \item A somewhat different application of quadratic form theory and incidence geometry to restriction problems (in the lattice setting) has recently appeared in the works of Bourgain and Demeter \cite{BourgainDemeter} and Demeter \cite{Demeter}.
 \item We briefly remark on some analogies between the finite field and Euclidean settings. In the Euclidean setting functions supported on small caps play a fundamental role. Indeed such examples demonstrate the optimality of the Stein-Tomas inequality and are the building blocks that allow one to embed Kakeya-type objects into the restriction problem. One insight of our current work is that affine subspaces contained within a surface play an analogous roles in the finite field setting. This analogy is not perfect, however, as evidenced by the fact that the finite field restriction conjecture is related to the Kakeya problem in a lower dimension, unlike the Euclidean setting.

    A key component of our analysis of the hyperbolic paraboloid $\mathcal{H}$ was Lemma \ref{lem:L52} which, roughly speaking, shows that near extermizers of the (extension form of) the Stein-Tomas inequality must concentrate on the lines contained within $\mathcal{H}$. This is somewhat analogous to the mixed-norm inequalities of Bourgain \cite{Br3} and Moyua-Vargas-Vega \cite{MVV}, which imply that near extermizers to the (extension form of the) Euclidean Stein-Tomas inequality must concentrate on caps. Curiously, the proof of both Lemma \ref{lem:L52} and these Euclidean mixed-norm inequalities proceed by reducing matters to an incidence-type estimate.

    It is interesting to speculate on the applicability of the ideas presented here to the Euclidean setting. Indeed, one can modify the Mockenhaupt-Tao slicing argument to the Euclidean setting. In fact this technique was first introduced by Carbery \cite{Carbery} in his proof that the restriction conjecture implies the Bochner-Riesz conjecture for the (Euclidean) paraboloid. Combining this, using the basic strategy introduced here, with, say, the Moyua-Vargas-Vega mixed-norm estimate should allow one to deduce a significant amount of structure to a hypothetical extremizers to the (restriction form of) the Stein-Tomas inequality. One could then hope to use this structure to obtain an improved restriction estimate.

    This would be significantly distinct from the original application of the Moyua-Vargas-Vega inequality to the restriction problem and perhaps more efficient(if the numerology here is any guide). This would, however, likely fail to be competitive with the recent multilinear approach to these problems, such as \cite{BG} and \cite{GuthPoly}.
 \item Recently, there has been a lot of interest in classifying extremals and near externals to Fourier analytic inequalities. For instance, two recent papers by Christ and Shao, \cite{CS1} and \cite{CS2}, study the extremizers of the (Euclidean) Stein-Tomas inequality for sphere. The ideas presented there might well have some application to the current setting or vice versa.
 \item We have focused exclusively on the Paraboloid in this work. The explicit nature of the Fourier transform of its characteristic function (Lemma \ref{lem:FTcomp}) simplifies many issues compared to other surfaces such as the sphere. In the case of the sphere, one must appeal to deeper results regarding exponential sums even to obatin Stein-Tomas-type estimates. We refer the reader to \cite{KangKoh} for a more detailed discussion of the Fourier restriction phenomenon for the finite field sphere.
 \item \emph{Note added:} The current work crucially relies on the polynomial method through the application of the Kakeya maximal operator estimates of Ellenberg, Oberlin and Tao \cite{EOT} for which there is currently no other known proof. More recently, Guth \cite{GuthPoly} has found an application of the polynomial partitioning method to the Euclidean restriction problem.
\item \emph{Note added:} Recently Koh \cite{KohNH} obtained new restriction estimates in certain high dimensional cases, improving on the results of \cite{LL} and \cite{IKParaboloid}. These arguments rely on, among other things, a high dimensional analog of the approach from \cite{LewkoNew}. In addition, Koh points out there that the necessary conditions in certain high dimensional cases were misstated in an earlier version of this manuscript. This has now been corrected.

\end{enumerate}

\section{Summary of progress on the finite field restriction problem}

\begin{center}
\begin{figure}[H]\label{fig:parab}
\begin{tabular}{|l|l|l|} \hline
Dimension /& Range of $p$ & \\
Field &  & \\
\hline
$d = 2$ & $p=4, q=2$ (S-T) & Mockenhaupt and Tao \cite{MT} (optimal)\\
\hline
$d = 3$ & $p = 4$, $q=2 $ (S-T) &  Mockenhaupt and Tao \cite{MT} \\
$d = 3$ & $p > 3.6$, $q=2 $ &  Mockenhaupt and Tao \cite{MT} \\
-1 not square   & $p \geq 3.6$, $q=2 $ & Bennett, Carbery, Garrigos, Wright \cite{BCGW} \\
 &    & (see also \cite{LL})   \\
 & $p \geq 3.6 - \delta $, $q=2 $ (for $\delta >0$)  &Lewko \cite{LewkoNew} \\
  & $p \geq 3 $, $q=2 $   & (conjectured) \\
\hline
$d = 3$ & $p = 4$, $q=2 $ (S-T) &  Mockenhaupt and Tao \cite{MT} \\
-1 a square   & $p = 3.6$, $q=9/4 $ &  Theorem \ref{thm:Res3} \\
  & $p \geq 3.6 - \delta $, $q=9/4 +\delta' $ (for $\delta, \delta' >0$)  &  Theorem \ref{thm:ResWSP} \\
  & $p \geq 3  $, $q=3 $   & (conjectured) \\
\hline
$d > 3$ & $p \geq \frac{2d+2}{d-1}$, $q=2$ (S-T) & Mockenhaupt and Tao \cite{MT}   \\
d even & $p >\frac{2d^2}{d^2-2d+2}$, $q=2$  &Iosevich and Koh \cite{IKParaboloid}  \\
& $p \geq \frac{2d^2}{d^2-2d+2}$, $q=2$  &Lewko and Lewko \cite{LL}  \\
& $p= \frac{2d}{d-1} $, $q= \frac{2d^2}{d^2-d+2} $  & (conjectured)\\
\hline
$d > 3$ & $p \geq \frac{2d+2}{d-1}$, $q=2$ (S-T) & Mockenhaupt and Tao \cite{MT}   \\
d odd & $p >\frac{2d^2}{d^2-2d+2}$, $q=2$ (for $d \equiv (3 \mod 4)$) &Iosevich and Koh \cite{IKParaboloid}  \\
-1 not a square & $p \geq \frac{2d^2}{d^2-2d+2}$, $q=2$ (for $d \equiv (3 \mod 4)$)  &Lewko and Lewko \cite{LL}  \\
& $p= \frac{2d}{d-1} $, $q= \frac{2d^2+2d}{d^2+3} $  (for $d \equiv (3 \mod 4)$) & (conjectured)\\
& $p= \frac{2d}{d-1} $, $q= \frac{2d}{d-1} $  (for $d \equiv (1 \mod 4)$) & (conjectured)\\
\hline
$d > 3$ & $p \geq \frac{2d+2}{d-1}$, $q=2$ (S-T) & Mockenhaupt and Tao \cite{MT}  \\
d odd & $p >\frac{2d+2}{d-1} - \delta_d$, $q=\frac{2d+2}{d-1}$  & Theorem \ref{thm:Main}  \\
-1 a square & $p= \frac{2d}{d-1} $, $q= \frac{2d}{d-1} $ & (conjectured)\\
\hline
\end{tabular}
\caption{Progress on the finite field restriction conjecture for the paraboloid. Here we emphasized the $p$ exponent and have omitted certain improvements in the $q$ exponent for inferior values of $p$. We use (S-T) to indicate the exponents given by the Stein-Tomas method.  }
\end{figure}
\end{center}

\texttt{M. Lewko}

\textit{mlewko@gmail.com}
\end{document}